\documentclass[journal]{IEEEtran}% twocolumns
\usepackage{epsfig,graphics,subfigure}
\usepackage{amsmath,amssymb,amsthm,multirow,balance,dsfont,hyperref}
\usepackage{cite,cleveref}
\usepackage[table]{xcolor}
\usepackage{colortbl}
\usepackage{array}
\usepackage[ruled,vlined]{algorithm2e}
\usepackage{mathrsfs}
\usepackage{empheq}
\usepackage[mathscr]{euscript}
\usepackage{rotating}
\DeclareFontFamily{U}{mathx}{\hyphenchar\font45}
\DeclareFontShape{U}{mathx}{m}{n}{
      <5> <6> <7> <8> <9> <10>
      <10.95> <12> <14.4> <17.28> <20.74> <24.88>
      mathx10
      }{}
\DeclareSymbolFont{mathx}{U}{mathx}{m}{n}
\DeclareFontSubstitution{U}{mathx}{m}{n}
\DeclareMathAccent{\widecheck}{0}{mathx}{"71}

\newtheorem{theorem}{Theorem}

\newtheorem{Property}{Property}

\newtheorem{assumption}{Assumption}

\usepackage{color}
\def\cblue{\textcolor{black}}

\newcommand{\boldh}{\boldsymbol{h}}

\newcommand{\bd}{\boldsymbol{d}}
\newcommand{\bs}{\boldsymbol{s}}
\newcommand{\bsb}{\overline{\bs}}
\newcommand{\bsc}{\widecheck{\bs}}

\newcommand{\bu}{\boldsymbol{u}}
\newcommand{\bv}{\boldsymbol{v}}
\newcommand{\bw}{\boldsymbol{w}}
\newcommand{\bz}{\boldsymbol{z}}
\newcommand{\bzb}{\overline{\bz}}
\newcommand{\bzc}{\widecheck{\bz}}
\newcommand{\bx}{\boldsymbol{x}}
\newcommand{\by}{\boldsymbol{y}}

\newcommand{\bpsi}{\boldsymbol{\psi}}
\newcommand{\bphi}{\boldsymbol{\phi}}
\newcommand{\bphib}{\overline{\bphi}}
\newcommand{\bphic}{\widecheck{\bphi}}

\newcommand{\bdelta}{\boldsymbol{\delta}}
\newcommand{\bchi}{\boldsymbol{\chi}}

\newcommand{\bH}{\boldsymbol{H}}

\newcommand{\cA}{\mathcal{A}}

\newcommand{\cJ}{\mathcal{J}}

\newcommand{\cN}{\mathcal{N}}
\newcommand{\cP}{\mathcal{P}}

\newcommand{\cU}{\mathcal{U}}

\newcommand{\cV}{\mathcal{V}}

\newcommand{\cw}{{\scriptstyle\mathcal{W}}}

\newcommand{\ccw}{{\scriptscriptstyle\mathcal{W}}}

\newcommand{\ccu}{{\scriptscriptstyle\mathcal{U}}}

\newcommand{\bcB}{\boldsymbol{\cal{B}}}
\newcommand{\bcH}{\boldsymbol{\cal{H}}}
\newcommand{\bcQ}{\boldsymbol{\cal{Q}}}
\newcommand{\bcD}{\boldsymbol{\cal{D}}}

\newcommand{\bcw}{\boldsymbol{\cw}}

\newcommand{\bwt}{\widetilde \bw}
\newcommand{\bpsit}{\widetilde \bpsi}
\newcommand{\bphit}{\widetilde \bphi}

\newcommand{\bcwt}{\widetilde\bcw}

\newcommand{\expec}{\mathbb{E}}

\newcommand{\col}{\text{col}}

\newcommand{\diag}{\text{diag}}
\newcommand{\sign}{\text{sign}}

\DeclareMathOperator*{\st}{subject~to}

\usepackage{etoolbox}
\AtBeginEnvironment{tabular}{\scriptsize}
\newcolumntype{C}[1]{>{\centering\arraybackslash}m{#1}}

\begin{document}
%===============================
% Title
%===============================
\title{Quantization for decentralized  learning \\ under subspace constraints}
\author{Roula Nassif, Stefan Vlaski, Marco Carpentiero,    Vincenzo Matta,    Marc Antonini,  Ali H. Sayed
\thanks{%This work was supported in part by NSF grant CCF-1524250. 

 \cblue{The work of R. Nassif was supported in part by ANR JCJC grant ANR-22-CE23-0015-01 (CEDRO project)}. 
 
 A short conference version of this work appears in~\cite{nassif2022finite}.
 
 R. Nassif and M. Antonini are with Universit\'e C\^ote d'Azur, I3S Laboratory, CNRS,  France (email: $\{$roula.nassif,marc.antonini$\}$@unice.fr).  S. Vlaski is with Imperial College London, UK (e-mail: s.vlaski@imperial.ac.uk). M. Carpentiero and V.  Matta are with University of Salerno, Italy (e-mail: $\{$mcarpentiero,vmatta$\}$@unisa.it).  A. H. Sayed is with {the Institute of Electrical and Micro Engineering}, EPFL, Switzerland (e-mail: ali.sayed@epfl.ch).}}

%\title{Adaptation and learning over networks \\under subspace constraints}
%\author{Roula Nassif, \IEEEmembership{Member, IEEE}, Stefan Vlaski, \IEEEmembership{Student Member, IEEE}, Ali H. Sayed, \IEEEmembership{Fellow Member, IEEE}\\
%\thanks{This work was supported in part by NSF grant CCF-1524250. A short conference version of this work appears in~\cite{nassif2019distributed}.
%
%The authors are with Institute of Electrical Engineering, EPFL, Switzerland (e-mail: $\{$roula.nassif, stefan.vlaski, ali.sayed$\}$@epfl.ch).
%}}

\maketitle

\begin{abstract}
In this paper, we consider decentralized optimization problems where agents have individual cost functions to minimize subject to subspace constraints that require the minimizers across the network to lie in low-dimensional subspaces. This constrained formulation includes consensus or single-task optimization as special  cases, and allows for more general task relatedness models such as multitask smoothness and  coupled optimization. In order to cope with  communication constraints, we propose and study an adaptive decentralized strategy where the agents employ \emph{differential randomized quantizers} to compress their estimates before communicating with their neighbors. The analysis shows that, under some general conditions on the quantization noise, and for sufficiently small step-sizes~$\mu$, the  strategy is stable both in  terms of  mean-square error and  average bit rate: by reducing  $\mu$, it is possible to keep the \emph{estimation errors small (on the order of $\mu$) without increasing indefinitely the bit rate as $\mu\rightarrow 0$ \cblue{when variable-rate quantizers are used}.} %The analysis also reveals the influence of the step-size $\mu$ and % gradient noise, network topology, and 
%quantization noise on the steady-state performance. 
Simulations illustrate   the theoretical findings and the effectiveness of the proposed approach, revealing that decentralized learning is achievable {at the expense of only a}  few bits.
% The analysis also reveals that, under some conditions on the quantization noise  and for small step-sizes $\mu$, the iterates generated by the quantized  implementation can still lead to small estimation errors on the order of $\mu$.
\end{abstract}
%============================
% Keywords
%============================

%\newpage
%==============================
% Sec: Introduction
%==============================
\section{Introduction}
Mobile phones, wearable devices, {and autonomous vehicles} are   {examples of} modern distributed networks generating massive amounts of data each day. Due to the growing computational power  {in these devices} and the increasing size  {of the} datasets, coupled with concerns over {privacy}, federated and decentralized training of statistical models have become  {desirable and often necessary~\cite{mcmahan2017communication,li2020federated,kairouz2021advances,alistarch2017qsgd,aledhari2020federated,smith2017federated,sayed2014adaptation,sayed2014adaptive,sayed2013diffusion,nassif2020multitask}.} In  {these} approaches, each participating device (which is referred to as \emph{agent} or \emph{node}) has a local training dataset, which is never uploaded to the server. {The training} data is kept locally {on the} users' devices, {which also serve as computational agents acting on the local data}  in order to update global models of interest. In applications where communication {with} a server becomes a bottleneck, \emph{decentralized} topologies (where agents only communicate with their neighbors) are {attractive} alternatives to federated %or star 
topologies (where a server connects with all remote devices). \cblue{Compared to federated approaches\cite{mcmahan2017communication,li2020federated,kairouz2021advances,alistarch2017qsgd,aledhari2020federated,smith2017federated}, decentralized} implementations  reduce  the high communication cost on the central server since, in this case, {the model} updates are exchanged {locally between} agents without relying on a central coordinator~\cite{koloskova2019decentralized,sayed2014adaptation,sayed2014adaptive,sayed2013diffusion,nassif2020multitask}. 

In practice, there are several issues that arise in the implementations of decentralized algorithms due to the use of a communication network. For instance, in modern distributed networks comprising a massive number of devices, e.g., thousands of participating smartphones, communication can be slower than local computation by many orders of {magnitude} (due to limited resources such as energy and bandwidth). Designers are typically limited by an upload bandwidth of 1MB/s or less~\cite{mcmahan2017communication}. While there have been significant works in the literature on solving optimization and inference problems in a \cblue{decentralized manner~\cite{nedic2009distributed,bertsekas1997new,dimakis2010gossip,sayed2014adaptation,sayed2014adaptive,sayed2013diffusion,nassif2020multitask,plata2017heterogeneous,mota2015distributed,nassif2017diffusion,kekatos2013distributed,alghunaim2020distributed,nassif2020learning,sahu2018cirfe,chen2014multitask,nassif2020adaptation,nassif2020adaptation2,dilorenzo2020distributed,nedic2008distributed,zhao2012diffusion,nassif2016diffusion,thanou2013distributed,lee2021finite,carpentiero2022adaptive,carpentiero2022distributed,kovalev2021linearly,koloskova2019decentralized,sparq2023sparq,singh2021squarm,tang2018communication,taheri2020quantized,zhang2019compressed}, with some exceptions~\cite{koloskova2019decentralized,nedic2008distributed,carpentiero2022adaptive,carpentiero2022distributed,thanou2013distributed,lee2021finite,zhao2012diffusion,nassif2016diffusion,kovalev2021linearly,zhang2019compressed,tang2018communication,sparq2023sparq,singh2021squarm,taheri2020quantized}}, the large majority of these works is not tailored to the specific challenge of limited communication capabilities. %(due to tight energy and bandwidth constraints) encountered in decentralized settings.  
%\cblue{Quantization for decentralized inference has received attention in recent years -- see Sec.\ref{sec: related work} for a comprehensive overview of the state-of-the-art. Here we emphasize that existing quantized decentralized approaches are designed only to solve consensus-based optimization problems. Furthermore, existing approaches rely mainly on quantization rules that assume that some quantities are represented with very high precision and that neglect the associated quantization error in the theoretical analyses.}
\cblue{Moreover, the existing works on decentralized approaches are often designed only to solve consensus-based optimization problems. And many of these existing approaches rely mainly on quantization rules that assume that some quantities are represented with very high precision and neglect the associated quantization error in the theoretical analyses.}

In this work, \cblue{we do not make any assumptions about the high-precision representation of specific variables. Moreover, we consider quantization for decentralized learning under subspace constraints by} studying the effects of quantization on the performance of  the following decentralized stochastic gradient approach {from~\cite{nassif2020adaptation,nassif2020adaptation2}}: 
\begin{subequations}
\label{eq: decentralized learning approach}
 \begin{empheq}[left={\empheqlbrace\,}]{align}
\bpsi_{k,i}&=\bw_{k,i-1}-\mu\widehat{\nabla_{w_k}J_k}(\bw_{k,i-1})\label{eq: step1}\\
\bw_{k,i}&=\sum_{\ell\in\cN_k}A_{k\ell}\bpsi_{\ell,i}\label{eq: step2}
 \end{empheq}
\end{subequations}
where $\mu>0$ is a small step-size parameter,  $\cN_k$ is the neighborhood set of agent $k$ (i.e., the set of nodes connected to agent $k$ by a communication link or edge, including  node $k$ itself), $A_{k\ell}$ is an $M_k\times M_\ell$ matrix associated with  link $(k,\ell)$, $w_k\in\mathbb{R}^{M_k}$ is the parameter vector at agent~$k$, and $J_k(w_k):\mathbb{R}^{M_k}\rightarrow\mathbb{R}$ is a differentiable convex cost associated with agent~$k$. {This cost is usually} expressed as the expectation of some loss function $L_k(\cdot)$ and written as $J_k(w_k)=\expec\,L_k(w_k;\by_k)$, where $\by_k$ denotes the random {data at agent $k$}  (throughout the paper, random quantities are denoted in boldface). The expectation is computed {relative to the distribution of the local data}. In the stochastic-optimization framework,  {the statistical} distribution of the data $\by_k$ is usually unknown and, hence, the risks $J_k(\cdot)$ and their gradients $\nabla_{w_k}J_k(\cdot)$ are unknown. In this case, and  instead of using the true gradient, it is common to use approximate gradient vectors  {such as} $\widehat{\nabla_{w_k}J_k}(w_k)=\nabla_{w_k}L_k(w_k;\by_{k,i})${,} where $\by_{k,i}$ represents the data {realization observed} at iteration~$i$~\cite{sayed2014adaptation}. We note that a unique feature of algorithm~\eqref{eq: decentralized learning approach} is the utilization of \emph{matrix} valued combination weights, as opposed to scalar weighting as is commonly employed in  {conventional consensus and diffusion} optimization~\cite{nedic2009distributed,sayed2014adaptation}. As explained in \cblue{the next paragraph}, this generalization allows the network to solve a broader class of multitask optimization problems beyond classical consensus. \cblue{Multitask learning is suitable for network applications where regional differences in the data require more complex models and more flexible algorithms than single-task or consensus implementations. In multitask networks, agents generally need to estimate and track multiple distinct, though related, objectives. For instance, in distributed power  state estimation problems, the local state vectors  at neighboring control centers may overlap partially since the areas in a power system are interconnected~\cite{kekatos2013distributed}. Likewise, in weather forecasting applications, regional differences in the collected data distributions  require agents to exploit the correlation profile in the data for enhanced decision rules~\cite{nassif2020learning}. Other multitask network applications include distributed minimum-cost flow~\cite{mota2015distributed,nassif2017diffusion}, distributed active noise control~\cite{plata2017heterogeneous}, and distributed
sub-optimal beamforming~\cite{nassif2020adaptation}.}

Let $N$ denote the total number of  agents and let $\cw=\col\{w_1,\ldots,w_N\}$ denote {the $M-$th  dimensional vector (where $M=\sum_{k=1}^NM_k$) collecting the} parameter vectors from across the network. Let $\cA$ denote the $N\times N$ block matrix whose $(k,\ell)$-th block is $A_{k\ell}$ if $\ell\in\cN_k$ %\footnote{\cmag{Professor: In the previous works~\cite{nassif2020adaptation,nassif2020adaptation2}, we used $\cA$ instead of $\cA^\top$. $A_{k\ell}$ is the weight that node $k$ assigns to the information coming from its neighboring node $\ell\in\cN_k$.}}
 and $0$ otherwise. It was shown in~\cite[Theorem~1]{nassif2020adaptation} that, for sufficiently small $\mu$ and  for a combination matrix $\cA$ satisfying:
\begin{equation}
\label{eq: condition A}
\cA\,\cU=\cU,\quad \cU^\top\cA=\cU^\top,~~ \text{and~ }\rho(\cA-\cP_{\ccu})<1,
\end{equation}
where $\rho(\cdot)$ denotes the spectral radius of its matrix argument, $\cU$ is {any given} $M\times P$ {full-column} rank matrix ({with} $P\ll M$) that is assumed  to be semi-unitary, i.e., its columns are orthonormal ($\cU^\top\cU=I_P$), and $\cP_{\ccu}=\cU\cU^\top$ is the orthogonal projection matrix onto $\text{Range}(\cU)$, strategy~\eqref{eq: decentralized learning approach} will converge in the mean-square-error sense to the solution of the following subspace constrained optimization problem%\footnote{\cmag{Professor: The formulation allows us to solve the consensus optimization problem written in~\eqref{eq: consensus optimization} by choosing $\cU=\frac{1}{\sqrt{N}}(\mathds{1}_N\otimes I_L)$. In this case, $M=NL$ and $P=L$. The Pareto optimization cannot be formulated using the subspace formulation. As condition~\eqref{eq: doubly stochastic condition} shows, the consensus solution requires a doubly-stochastic matrix $\cA$, not left-stochastic (as in the Pareto case).}}
:
\begin{equation}
\label{eq: network constrained problem}
\begin{split}
\cw^o=&\arg\min_{\ccw}~J^{\text{glob}}(\cw)\triangleq\sum_{k=1}^NJ_k(w_k)\\
&\st ~\cw\in\text{Range}(\cU).
\end{split}
\end{equation}
 In particular,  it was shown that 
%\begin{equation}
%\label{eq: condition A}
$\limsup_{i\rightarrow\infty}\expec\|w^o_k-\bw_{k,i}\|^2=O(\mu)$ for all $k$,
%\end{equation}
where $w^o_k$ is the $k$-th $M_k\times 1$ subvector of $\cw^o$. As \cblue{explained in~\cite{nassif2020multitask},~\cite{dilorenzo2020distributed},~\cite[Sec. II]{nassif2020adaptation},} by properly selecting $\cU$ and $\cA$, strategy~\eqref{eq: decentralized learning approach} can be employed to solve different decentralized  optimization problems such as (i) consensus or single-task optimization (where{, as explained in Remark~1 further ahead,} the {agents'} objective is to reach a consensus on the minimizer of the aggregate cost $\sum_{k=1}^NJ_k(w)$)~\cite{nedic2009distributed,sayed2014adaptation,sayed2014adaptive,sayed2013diffusion}, (ii)~decentralized coupled optimization (where the parameter vectors to be estimated at neighboring agents are partially overlapping)~\cite{plata2017heterogeneous,mota2015distributed,kekatos2013distributed,alghunaim2020distributed,sahu2018cirfe}, and (iii) multitask inference under smoothness (where the network parameter vector  $\cw$ to be estimated is  smooth w.r.t. the underlying network topology)~\cite{nassif2020adaptation,chen2014multitask}. \cblue{For instance, while projecting onto the space spanned by the vector of all ones allows to enforce consensus across the network (see Remark~1),  graph smoothness  can in general be promoted by projecting onto the space spanned by the eigenvectors of the graph Lapacian corresponding to small eigenvalues  (see the simulation section~VI for an illustration). Besides being capable of solving single-task or consensus-based optimization problems,  the constrained formulation~\eqref{eq: network constrained problem} and its decentralized solution~\eqref{eq: decentralized learning approach}  are general enough to apply to a wide variety of network applications, including those listed in the previous paragraph.}

The first step~\eqref{eq: step1} in algorithm~\eqref{eq: decentralized learning approach} is the \emph{self-learning} step corresponding to the stochastic gradient descent step on the individual cost $J_k(\cdot)$. This step is followed by the \emph{social learning} step~\eqref{eq: step2} where agent $k$ receives the intermediate estimates $\{\bpsi_{\ell,i}\}$  from its neighbors $\ell\in\cN_k$ and combines them  through $\{A_{k\ell}\}$ to form $\bw_{k,i}$, {which corresponds to the} estimate of $w^o_k$ at agent $k$ and iteration $i$. To alleviate the communication bottleneck resulting from the exchange of the intermediate estimates among agents over many iterations, quantized communication must be considered. In this paper, we  study the effect of quantization on the convergence properties of the decentralized learning approach~\eqref{eq: decentralized learning approach}. {First, we describe in Sec.~\ref{sec: rendomized quantizer} the  class of randomized quantizers considered in this study. Then,} we propose in Sec.~\ref{sec: Decentralized learning in the presence of quantized communications} a \emph{differential randomized quantization} strategy for solving \cblue{problem~\eqref{eq: network constrained problem} -- see~\eqref{eq: decentralized learning approach with quantization} further ahead. Compared with the unquantized version~\eqref{eq: decentralized learning approach}, the new approach  consists of three steps where a quantization step is added. Interestingly, instead of exchanging compressed versions of the intermediate estimates $\{\bpsi_{\ell,i}\}$, and due to differential quantization, the agents in the proposed solution exchange compressed versions of the differences between subsequent iterates (which  tend to have a reduced range when compared to  the intermediate estimates). At the receiver side, prediction rules are implemented in order to reconstruct the intermediate estimates. This step leads to a new set of intermediate estimates $\{\bphi_{\ell,i}\}$ which are then: (i) stored at the receiver in order to be used as \emph{predictors} in the next iteration, and (ii) % referred to as \emph{predictors}. These predictors  
combined according to a modified version of the combination step. In the modified version, a \emph{mixing} parameter $\gamma$ is introduced in order to control the mixing speed of the algorithm. This allows to control the network stability in situations where quantization can lead to network instability.  We} establish in Sec.~\ref{sec: Stochastic performance analysis} that, under some general conditions on the quantization noise \cblue{and mixing parameter}, and for sufficiently small step-sizes $\mu$, the  decentralized quantized approach is stable in the mean-square error sense. {In addition to investigating the mean-square-error stability, we characterize} the steady-state average bit rate of the proposed approach  \cblue{when  \emph{variable-rate} quantizers are used}. %Our analysis reveals the influence of the quantization noise on the network stability and  performance. 
The analysis  shows that, by properly designing the quantization operators, the iterates generated by the quantized decentralized adaptive  implementation lead to small estimation errors on the order of {$\mu$ (as it happens in the {\em ideal case without quantization}), while concurrently guaranteeing a bounded average bit rate as $\mu\rightarrow 0$. While there exist several useful works in the literature that study decentralized learning approaches in the presence of \cblue{differential quantization~\cite{carpentiero2022distributed,carpentiero2022adaptive,kovalev2021linearly,koloskova2019decentralized,taheri2020quantized,lee2021finite,reisizadeh2019exact,tang2018communication,sparq2023sparq,singh2021squarm,zhang2019compressed}, these} works investigate standard consensus or single-task optimization, and do not consider the subspace constrained {formulation~\eqref{eq: network constrained problem} or multitask variants}. Moreover, \cblue{with some exceptions that consider deterministic optimization~\cite{lee2021finite,reisizadeh2019exact,zhang2019compressed}}, the analyses conducted in these works assume that some quantities (e.g., the norm  or some components  of the vector to be quantized) are represented with very high precision (e.g., machine precision) and the associated quantization error is neglected. On the other hand, the analysis in the current work is not limited to standard consensus optimization and does not assume a machine precision representation of some quantities -- a detailed discussion of related work is provided in Sec.~\ref{sec: related work}. }

\noindent\textbf{Notation:} All vectors are column vectors. Random quantities are denoted in boldface. Matrices are denoted in uppercase letters while vectors and scalars are denoted in lower-case letters. The symbol $(\cdot)^\top$ denotes matrix {transposition}. The operator $\col\{\cdot\}$ stacks the column vector entries on top of each other. The operator $\diag\{\cdot\}$ %\footnote{\cmag{Professor: To be consistent with the previous unquantized versions~\cite{nassif2020adaptation,nassif2020adaptation2} of this work and your NOW book~\cite{sayed2014adaptation}, I used the notation $\diag$.}}
 forms a matrix from block arguments by placing each block immediately below and to the right of its predecessor. The symbol $\otimes $ denotes the Kronecker product. The $M\times M$ identity matrix is denoted by $I_M$. The abbreviation ``w.p.'' is used for ``with probability''.  The Gaussian distribution with mean $m$ and covariance $C$ is denoted by ${\mathcal {N}}(m ,C)$. The notation $\alpha=O(\mu)$ signifies that there exist two positive constants $c$ and $\mu_0$ such that $|\alpha|\leq c\mu$ for all $\mu\leq\mu_0$. A vector of all zeros is denoted by 0. For tables, the header is not counted as a row  (i.e., ``first row'' means the first row after the header). %$|\alpha|\leq c\mu$ for some constant $c>0$~\cite{sayed2014adaptation}. }

%==============================================
% Sec: Randomized quantizer
%==============================================
\section{Randomized quantizers}
\label{sec: rendomized quantizer}

%===============================================
% Subsec: Deterministic input
%==============================================
%\subsection{Randomized quantizers}
%\cblue{To illustrate the difference between deterministic and randomized quantizers, consider a simple $T$-point scalar quantizer. In addition to the partition $\{R_i|i=1,\ldots, T\}$ of the real line $\mathbb{R}$ and the set of output points $\{y_i|i=1,\ldots,T\}$ (with $y_j\in R_j$) describing a deterministic quantizer $\cQ(x)$, a randomized quantizer $\bcQ(x)$ is defined by a set of conditional probabilities $\mathbb{P}[\boldsymbol{Y}=y_j|X\in R_i]$ where the variables $X$ and $\boldsymbol{Y}$ represent the quantizer input and output, respectively~\cite[pp.~198]{gersho1992scalar}}. 
{In this paper, we  study decentralized learning under subspace constraints in the presence of quantized communications. \emph{Randomized quantizers} $\bcQ(\cdot)$\footnote{{Since the output of a randomized quantizer  is random even for deterministic input, we use  the boldface notation $\bcQ(\cdot)$ to refer to randomized quantizers. 
%As we will see, in randomized quantizer, given an input value, the output level is randomly assigned through a probabilistic mechanism that is specified in terms of a probability function.
}}   will be employed instead of deterministic quantizers. As explained in~\cite{carpentiero2022distributed}, deterministic quantizers can lead to severe estimation biases in inference problems. To overcome this issue,  randomized quantizers $\bcQ(\cdot)$ %{(which are described in Sec.~\ref{subsec: rendomized quantizer})} 
are commonly used  to compensate for the bias (on average, over time)~\cite{koloskova2019decentralized,carpentiero2022adaptive,carpentiero2022distributed,kovalev2021linearly,lee2021finite,taheri2020quantized}. This section is devoted to describing the class of randomized quantizers considered throughout the study.}

{For any deterministic input $x\in\mathbb{R}^L$ {with $L$ representing a generic vector length}, the randomized quantizer $\bcQ(\cdot)$ % in~\eqref{eq: stepb} 
is characterized in terms of a probability 
%\begin{equation}
%p_k(q;x)=
{$\mathbb{P}[\bcQ(x)=y]$ }%,\qquad k=1,2,\ldots,N, \cred{\text{ why not }\mathbb{P}[\bcQ_k(x)=q|x]?}
%\label{eq:qpmf}
%\end{equation}
 for any $y$ belonging to the set of output levels of the quantizer.  We {consider randomized quantizers $\bcQ(\cdot)$ % in~\eqref{eq: stepb} 
 satisfying the} %shall analyze algorithm~\eqref{eq: decentralized learning approach with quantization} under the 
following general property, which %  on the quantizers $\{\bcQ_k(\cdot)\}$
as explained in the sequel, relaxes the condition on the mean-square error from~\cite{carpentiero2022adaptive,koloskova2019decentralized,alistarch2017qsgd,kovalev2021linearly,carpentiero2022distributed,taheri2020quantized}.}

%\begin{assumpB}
%This assumption is labelled "Assumption A1".
%\end{assumpB}

\begin{Property}{\emph{\textbf{{(Unbiasedness and variance bound)}}.}}
\label{property: quantization noise}
The randomized quantizer $\bcQ(\cdot)$ satisfies the following two conditions: 
% This probability mass function satisfies the following two conditions: 
%\textcolor{blue}{we could say in the notation section that the symbol $0$ is occasionally used to denote a vector with all zero entries.}
\begin{align}
\expec[x-\bcQ(x)]&=0,\label{eq: expectation of the quantization noise deterministic}\\
{\expec\|x-\bcQ(x)\|^2}&\leq \beta^2_{q}\|x\|^2+\sigma^2_{q},\label{eq: expectation squared of the quantization noise deterministic}
\end{align}
for some $\beta^2_{q}\geq 0$ and $\sigma^2_{q}\geq 0${, and where the expectations are evaluated w.r.t. the randomness of $\bcQ(\cdot)$}. When the quantizers are applied to a random input $\bx$, conditions~\eqref{eq: expectation of the quantization noise deterministic} and~\eqref{eq: expectation squared of the quantization noise deterministic} become:
\begin{align}
\expec[\bx-\bcQ(\bx)|\bx]&=0,\label{eq: expectation of the quantization noise random 1}\\
{\expec[\|\bx-\bcQ(\bx)\|^2|\bx]}&\leq \beta^2_{q}\|\bx\|^2+\sigma^2_{q}.\label{eq: expectation squared of the quantization noise random 1}
\end{align}
\hfill\qed
\end{Property}

{Property~\ref{property: quantization noise} is satisfied by many %\emph{deterministic} and
{randomized} quantization operators of interest in decentralized learning. Table~\ref{table: examples of quantizers} further ahead lists some typical choices (a detailed comparison of the various schemes will be provided in Sec.~\ref{sec: related work}).  {Many existing works focus on studying decentralized learning approaches in the presence of randomized quantizers that satisfy the \emph{unbiasedness} condition~\eqref{eq: expectation of the quantization noise deterministic} and the \emph{variance bound}~\eqref{eq: expectation squared of the quantization noise deterministic} with the \emph{absolute noise} term $\sigma^2_{q}=0$~\cite{carpentiero2022adaptive,koloskova2019decentralized,alistarch2017qsgd,kovalev2021linearly,carpentiero2022distributed,taheri2020quantized}. In contrast, the analysis in the current work is general and does not require $\sigma^2_{q}$ to be zero. As we will explain in Sec.~\ref{sec: related work}, neglecting the effect of $\sigma^2_{q}$ requires that some quantities (e.g., the norm of the vector to be quantized) are represented with no quantization error, in practice at the machine precision. In the following, we describe a  useful framework for designing randomized quantizers that do not require high-precision quantization of specific variables.}}%  precision quantization of some neglecting any quantization error.   For illustration purposes, and before moving to Sec.~\ref{sec: Decentralized learning in the presence of quantized communications}, % where we provide a more detailed {comparison} of the various schemes listed in Table~\ref{table: examples of quantizers}, %In the following section, we provide a more detailed discussion of the various schemes %by focusing on quantifying   the amount of compression achieved by the listed compression schemes 
\subsection{Uniform and non-uniform randomized quantizers}
\label{subsec: Uniform and non-uniform randomized quantizers}
\subsubsection{Quantizers' design}
\label{sec: Non-uniform randomized quantizers}
%\cblue{We start by describing a general way to assign the quantization intervals and \cblue{output} %reproduction
% levels  for non-uniform randomized unbiased quantizers, and then we show how the two specific examples can be recovered. }%These two quantizers, which do not assume machine precision quantization of specific variables, will be investigated in the following sections}.
 {Let $x\in\mathbb{R}^L$ denote the  input vector to be quantized with $x_j$ representing the $j$-th element of~$x$.} We consider a general quantization rule $\bcQ:\mathbb{R}^L\rightarrow\mathbb{R}^L$ of the {following form -- see Fig.~\ref{fig: illustration nonuniform} for an illustration%of the randomized quantization scheme described in~\eqref{eq: non-uniform quantization}--\eqref{eq: reproduction level}
 }:
\begin{equation}
\label{eq: non-uniform quantization}
[\bcQ(x)]_j= y_{\boldsymbol{n}(x_j)},
\end{equation}
where {$[\bcQ(x)]_j$ denotes the $j$-th element of~$\bcQ(x)$ and} $y_{\boldsymbol{n}(x_j)}$ is the {quantization output} level (defined further ahead in~\eqref{eq: reproduction level}) associated with a realization of the random index $\boldsymbol{n}(x_j)\in\{m,m+1\}$. The \emph{probabilistic} rule to choose $\boldsymbol{n}(x_j)$ is {as follows}:
\begin{equation}
\label{eq: random index n}
\boldsymbol{n}(x_j)=\left\lbrace\begin{array}{ll}
m,&\text{w.p. }\frac{y_{m+1}-x_j}{y_{m+1}-y_m},\\
m+1,&\text{w.p. }\frac{x_j-y_m}{y_{m+1}-y_m}.
\end{array}
\right.
\end{equation}
Regarding the integer $m$  in~\eqref{eq: random index n}, and motivated by the so-called companding procedure~\cite{gersho1992scalar}{,} which can be conveniently described in terms of a non-linear function and its inverse, it is found according to (the dependence of $m$ on $x_j$ is left implicit in~\eqref{eq: random index n} for ease of notation):
\begin{equation}
\label{eq: equation for m}
m=\lfloor g(x_j)\rfloor,
\end{equation}
where {$\lfloor \cdot\rfloor$ denotes the floor function and} $g:\mathbb{R}\rightarrow\mathbb{R}$ is {some} {strictly increasing} continuous function.  Given the {inverse function} $h=g^{-1}$, the {output} level associated with an index $m$ {is defined} by:
%\footnote{In a standard companding approach,  a uniform quantization block is placed between the non-linearities $g$ and $h$~\cite{gersho1992scalar}. However, without loss of generality, this intermediate step can be avoided by absorbing the uniform quantization step {into} the definition of the non-linearities.  Moreover, companding is traditionally employed only within a granular region (i.e., within a bounded range) and with a fixed rate (i.e., when the number of bits is fixed and independent of the input $x$).  In this work, we will instead use companding as a general way to devise non-uniform quantization intervals over the entire real axis, and we will consider variable-rate quantizers where the number of bits depends on $x$, as explained in {Sec.~\ref{subsec: Variable-rate coding and bit budget}}. }:
\begin{equation}
\label{eq: reproduction level}
y_m=h(m).
\end{equation}
%{Figure~\ref{fig: illustration nonuniform} provides an illustration of the randomized quantization scheme described in~\eqref{eq: non-uniform quantization}--\eqref{eq: reproduction level}. }
By evaluating the expected value of $[\bcQ(x)]_j$ w.r.t. the quantizer randomness, we obtain:
\begin{equation}
%\expec[\bcQ(x)]_j\overset{\eqref{eq: non-uniform quantization}}=\expec[ y_{\boldsymbol{n}(x_j)}]\overset{\eqref{eq: random index n}}=\left(y_m\frac{y_{m+1}-x_j}{y_{m+1}-y_m}+y_{m+1}\frac{x_j-y_m}{y_{m+1}-y_m}\right)=x_j,
\expec[\bcQ(x)]_j\overset{\eqref{eq: non-uniform quantization},\eqref{eq: random index n}}=\left(y_m\frac{y_{m+1}-x_j}{y_{m+1}-y_m}+y_{m+1}\frac{x_j-y_m}{y_{m+1}-y_m}\right)=x_j,
\end{equation}
which {establishes} the unbiasedness condition~\eqref{eq: expectation of the quantization noise deterministic} {in Property~\ref{property: quantization noise}}.
\begin{figure*}
\begin{center}
\includegraphics[scale=0.18]{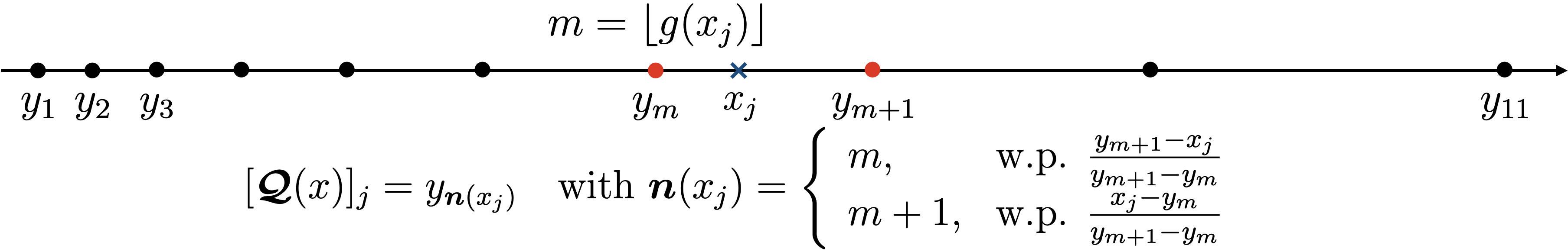}
\caption{{An illustration of the non-uniform randomized quantization scheme defined by~\eqref{eq: non-uniform quantization}--\eqref{eq: reproduction level}. %The real axis partition is constructed by choosing the function  $g(\cdot)$ according to~\eqref{eq:generallogcomp} with $a=4.17$ and $b=10$. %
The real axis partition is constructed by choosing the functions  $g(\cdot)$ and $h(\cdot)$ according to~\eqref{eq:generallogcomp} with $a=4.17$ and $b=10$.%according to~\eqref{eq: non-linearities for ANQ} with $\omega=0.122$ and $\eta=0.1\omega$.%$g(t)=\frac{\ln(1+b|t|)}{10\ln(1+b)}$ whose inverse is given by $h(m)=\frac{1}{b}[(1+b)^{0.1m}-1]$% ~\eqref{eq:generallogcomp} with $a=4.17$ and $b=10$.
}}
\label{fig: illustration nonuniform}
\end{center}
\end{figure*}

\noindent \textbf{Example~1.}   {\textbf{({Randomized uniform or dithered quantizer}~\cite{aysal2008distributed,reisizadeh2019exact})}}. By choosing:
\begin{equation}
\label{eq: non-linearities of the uniform}
g(t)=\frac{t}{\Delta},\qquad h(t)=\Delta\cdot t,
\end{equation}
we obtain the \emph{uniform} quantizer described in Table~\ref{table: examples of quantizers}  {(second row)} with  a quantization step $\Delta>0$. {{Condition~\eqref{eq: expectation squared of the quantization noise deterministic} (with $\beta^2_{q}=0$ and $\sigma^2_{q}=\frac{L\Delta^2}{4}$) can be established by letting $p=\frac{y_{m+1}-x_j}{y_{m+1}-y_m}$ and by evaluating the variance:
\begin{align}
\expec(x_j-[\bcQ(x)]_j)^2&=(x_j- y_m)^2p+(x_j- y_{m+1})^2(1-p)\notag\\
&=\Delta^2p(1-p)\leq \frac{\Delta^2}{4},
\end{align}
where we replaced $y_{m+1}-y_m$ by ${\Delta}$ and we used the fact that $p(1-p)\leq \frac{1}{4}$ for $p\in[0,1]$}.}$\hfill\blacksquare$
\begin{table*}
\caption{Examples of quantizers $\bcQ:\mathbb{R}^{L}\rightarrow\mathbb{R}^{L}$ satisfying Property~\ref{property: quantization noise}. {For each scheme, we report the quantization rule, the parameters $\beta_q^2$ and $\sigma^2_q$ in~\eqref{eq: expectation squared of the quantization noise deterministic}, and the bit-budget. }  %More  details on  the design of dithered quantizer and  probabilistic ANQ are provided in Sec.~\ref{subsec: Uniform and non-uniform randomized quantizers}. 
{$B_{\text{HP}}$ denotes the number of bits required to encode a  scalar with high precision (Typical values for $B_{\text{HP}}$ are $32$ or $64$). } }
\label{table: examples of quantizers}
\begin{center}
\begin{tabular}{||>{\centering\arraybackslash}m{0.905in} ||c ||>{\centering\arraybackslash}m{0.62in} |>{\centering\arraybackslash}m{0.17in} ||>{\centering\arraybackslash}m{1.17in}||} 
 \hline \hline
\cellcolor[gray]{0.8} Quantizer name &\cellcolor[gray]{0.8} Rule &\cellcolor[gray]{0.8}$\beta^2_{q}$ &\cellcolor[gray]{0.8} $\sigma^2_{q}$&\cellcolor[gray]{0.8} Bit-budget \\ [0.5ex] 
 \hline\hline
 \multirow{1}{*}{No compression~\cite{nassif2020adaptation} }& $\bcQ(x)=x$ &\multirow{1}{*}{$0$} &\multirow{1}{*}{$0$}&{$LB_{\text{HP}}$}\\ \hline
 \multirow{3}{*}{ Probabilistic uniform}&  $\left[\bcQ(x)\right]_j=\Delta\cdot\boldsymbol{n}(x_j) $&\multirow{3}{*}{0} &\multirow{3}{*}{$L\frac{\Delta^2}{4}$}&\multirow{3}{*}{%{$\log_2(3)\sum\limits_{j=1}^L(1+\lceil\log_2(|\boldsymbol{n}(x_j)|+1)\rceil)$
$\boldsymbol{r}(x)$ defined in~\eqref{eq: bit rate general formula}}\\ 
 \multirow{2}{*}{or dithered} &  $\boldsymbol{n}(x_j)=\left\lbrace\begin{array}{ll}
m,&\text{w.p. }%1-\frac{|x_i|-n\Delta}{\Delta}
\frac{(m+1)\Delta-x_j}{\Delta},\\
m+1,&\text{w.p. }\frac{x_j-m\Delta}{\Delta},
 \end{array}\right.\qquad m=\left\lfloor
\frac{x_j}{\Delta}
\right\rfloor$ &&&\\ 
{quantizer~\cite{aysal2008distributed,reisizadeh2019exact}} & $\Delta$ is the quantization step  &&&\\ 
% &{$m=\left\lfloor
%\frac{x_i}{\Delta}
%\right\rfloor$, $\Delta$ is the quantization step}%, 
 %&$\ell$ is such that $|x_i|\in[\ell\Delta,(\ell+1)\Delta]$, $\Delta$ is a small step-size%, $b\in\mathbb{N}$ number of bits 
% & &\\ 
 \hline
 \multirow{5}{*}{Probabilistic ANQ~\cite{lee2021finite} }& $\left[\bcQ(x)\right]_j=y_{\boldsymbol{n}(x_j)}$ &\multirow{5}{*}{$2\omega^2$} &\multirow{5}{*}{$2L\eta^2$}&\multirow{5}{*}{$\boldsymbol{r}(x)$ defined in~\eqref{eq: bit rate general formula}}\\ 
 & {$\boldsymbol{n}(x_j)= \left\lbrace\begin{array}{ll}
m,&\text{w.p. } \frac{y_{m+1}-x_j}{y_{m+1}-y_{m}},\\
m+1,&\text{w.p. } \frac{x_j-y_{m}}{y_{m+1}-y_{m}},
 \end{array}\right.\quad m=\left\lfloor
\sign(x_j)\frac{\ln\left(1+\frac{\omega}{\eta}|x_j|\right)}
{2\ln \left( \omega + \sqrt{1+\omega^2} \right)}
\right\rfloor$} & &&\\ 
&{$y_{m}=\sign(m)\frac{\eta}{\omega}\left[
\left(\omega + \sqrt{1+\omega^2}\right)^{2 |m|}
- 1
\right]$}
 & &&\\ 
&$\omega$ and $\eta$ are two non-negative design parameters & &&\\ \hline
 \multirow{2}{*}{Rand-$c$~\cite{kovalev2021linearly} }& $\left[\bcQ(x)\right]_j=\left\lbrace\begin{array}{ll}
\frac{L}{c}\cdot x_j,&\text{if  }x_j\in\Omega_c\\
0,&\text{otherwise }
 \end{array}\right.$ $\qquad\qquad(c\in\{1,\ldots,L\})$&\multirow{2}{*}{$\left(\frac{L}{c}-1\right)$} &\multirow{2}{*}{$0$}&\multirow{2}{*}{{$cB_{\text{HP}}+c\lceil\log_2(L)\rceil$}}\\ 
& $\Omega_c$ is a set of $c$ randomly selected coordinates& &&\\ 
  \hline
\multirow{1}{*}{Randomized Gossip~\cite{koloskova2019decentralized} }& $\bcQ(x)=\left\lbrace\begin{array}{ll}
\frac{x}{q},&\text{w.p. }q\\
0,&\text{w.p. }1-q
 \end{array}\right.$ $\qquad\qquad(q\in(0,1])$&\multirow{1}{*}{$\left(\frac{1}{q}-1\right)$} &\multirow{1}{*}{$0$}&\multirow{1}{*}{{$(LB_{\text{HP}})q\qquad$ (on average)}}\\ 
 \hline
\multirow{2}{*}{Gradient sparsifier~\cite{reisizadeh2019exact}}& $\left[\bcQ(x)\right]_j=\left\lbrace\begin{array}{ll}
\frac{x_j}{q_j},&\text{w.p. }q_j\\
0,&\text{w.p. }1-q_j
 \end{array}\right.$  &\multirow{2}{*}{$\frac{1}{\min\limits_j\{q_j\}}-1$} &\multirow{2}{*}{$0$}&\multirow{1}{*}{{$\sum\limits_{j=1}^Lq_j( B_{\text{HP}}+\lceil\log_2(L)\rceil)$}}\\ 
 & $q_j$ is the probability that coordinate $j$ is selected  & &&(on average)%\cblue{$\bc$: number of randomly selected coordinates}
 \\ 
 \hline
  \multirow{3}{*}{QSGD~\cite{alistarch2017qsgd} }& $\left[\bcQ(x)\right]_j=\|x\|\cdot\sign(x_j)\cdot\frac{\boldsymbol{n}(x_j,x)}{s}$ &\multirow{3}{*}{$\min\left(\frac{L}{s^2},\frac{\sqrt{L}}{s}\right)$} &\multirow{3}{*}{$0$}&\multirow{3}{*}{{$B_{\text{HP}}+L+L\lceil\log_2(s)\rceil$}}\\ 
 & $\boldsymbol{n}(x_j,x)=\left\lbrace\begin{array}{ll}
m,&\text{w.p. }(m+1)-\frac{|x_j|}{\|x\|}s,\\
 m+1,&\text{w.p. }\frac{|x_j|}{\|x\|}s-m,
 \end{array}\right.\qquad m=\left\lfloor s \frac{|x_j|}{\|x\|}\right\rfloor$ & &&\\ 
%& $q(a,s)=as-n$,\qquad \cblue{$\frac{n}{s}=\left\lfloor \frac{|x_i|}{\|x\|}\right\rfloor$} & &\\ 
&$s$ is the number of quantization levels %$r\in[0,s)$ integer such that $\frac{|x_i|}{\|x\|}\in\left[ \frac{\ell}{s}, \frac{\ell+1}{s}\right]$
 & &&\\ 
 
 \hline
 \hline
\end{tabular}
\end{center}
%\end{sidewaystable}
\end{table*}

\noindent \textbf{Example~2.}  {\textbf{({Randomized logarithmic companding})}}. For non-uniform quantizers, one popular choice is   {\em logarithmic companding}, which corresponds to  {the following direct and inverse non-linear} functions~\cite{gersho1992scalar}:
\begin{equation}
g(t)=\sign(t)\,a\ln(1+b|t|),\quad {h(t)=\sign(t)\frac{1}{b}\left(e^{\frac{|t|}{a}}-1\right)},
\label{eq:generallogcomp}
\end{equation}
with $a>0$ and $b>0$. Two useful results were established in~\cite{lee2021finite} regarding this choice. First, it was shown that, by setting the constants $a$ and $b$ according to:
\begin{equation}
\label{eq: choice for a and b}
a=\frac{1}{2\ln (\omega+\sqrt{1+\omega^2})},\qquad b=\frac{\omega}{\eta},
\end{equation}
we obtain the following bound on the variance:
\begin{equation}
\label{eq: error bound scutari}
{\expec\|x-\bcQ(x)\|^2\leq \left(\omega\|x\|+\sqrt{L}\eta\right)^2}. 
\end{equation}
The choice~\eqref{eq: choice for a and b} gives the probabilistic ANQ rule reported in Table~\ref{table: examples of quantizers} {(third row)} with the non-linear functions given by:
\begin{equation}
\label{eq: non-linearities for ANQ}
\left\lbrace
\begin{array}{l}
g(t)=\sign(t)\,\frac{\ln\left(1+\frac{\omega}{\eta}|t|\right)}{2\ln (\omega+\sqrt{1+\omega^2})},\\
h(t)=\sign(t)\frac{\eta}{\omega}\left[(\omega+\sqrt{1+\omega^2})^{2|t|}-1\right].
\end{array}
\right.
\end{equation}
Second, it was shown in~\cite{lee2021finite} that, for a fixed number of {output} levels, the choice in~\eqref{eq: choice for a and b} maximizes the range of the input variable while still satisfying the error bound~\eqref{eq: error bound scutari}. This property justifies the efficiency of the choice~\eqref{eq: choice for a and b} in terms of quantization bit rate.

 Now, by taking the limit of $g(t)$ in~\eqref{eq: non-linearities for ANQ} as $\omega\rightarrow 0$  and by applying l'H\^opital's rule, we get:
\begin{equation}
\lim_{\omega\rightarrow 0}g(t)=\sign(t)\frac{|t|}{2\eta}=\frac{t}{2\eta},
\end{equation}
from which we conclude that the setting where  $\omega=0$ allows us to recover the uniform quantizer~\eqref{eq: non-linearities of the uniform} of Example~1.

 Finally, we can relate the quantizer parameters $\omega$ and $\eta$ to the constants $\beta^2_{q}$ and $\sigma^2_{q}$ appearing in~\eqref{eq: expectation squared of the quantization noise deterministic}. In fact, by applying  Jensen's inequality to~\eqref{eq: error bound scutari}, we can see that the bound  in~\eqref{eq: expectation squared of the quantization noise deterministic} is satisfied with:
\begin{equation}
\label{eq: beta and sigma relation}
\beta^2_{q}=\frac{\omega^2}{\alpha},\qquad\sigma^2_{q}=\frac{L\eta^2}{1-\alpha}, \qquad \text{for any } \alpha\in(0,1).
\end{equation}$\hfill\blacksquare$

%================================
% Subsubsection: Variable-rate coding
%================================
\subsubsection{{Variable-rate coding scheme and bit budget~\cite{lee2021finite}}}
\label{subsec: Variable-rate coding and bit budget}
 {Note that the quantization rule~\eqref{eq: non-uniform quantization}--\eqref{eq: reproduction level} maps a continuous variable to some integer by partitioning the real  line into {an infinite number of} intervals. Thus, a fixed-rate quantizer (i.e., a quantizer that uses the same number of bits for any input) cannot represent all possible intervals of the partition. On the other hand, assuming a {finite support} for the quantizer input would require some boundedness assumptions (e.g., on the iterates, on the gradient) that are usually violated  in  stochastic optimization theory~\cite{nguyen2018sgd}. By following a standard approach in coding theory, we shall  instead investigate the use of \emph{variable-rate} quantizers which are able to adapt the bit rate based on the quantizer input, e.g., assigning more bits to larger inputs and less bits to smaller ones. In this work, we illustrate the main concepts by focusing on the variable-rate coding scheme proposed in~\cite{lee2021finite} and described in the following%\footnote{{While an in-depth investigation of other coding methods can be of great interest, it remains beyond the scope of this work.}}
 . This rule has the advantage of not requiring any knowledge about the distribution of the variables to be quantized. This is particularly relevant since such knowledge is typically unavailable in learning applications.}
% In coding theory,} this issue is usually circumvented by resorting to \emph{variable-rate} quantizers which are able to adaptively choose the number of bits to represent the input, e.g., assigning more bits to larger inputs and less bits to smaller ones. \cblue{In the following, we describe the variable-rate coding rule proposed in~\cite{lee2021finite},  which has the advantage of not requiring any knowledge about the distribution of the variables to be quantized. This is particularly relevant since such knowledge is usually unavailable in learning applications.}}
 
 \begin{figure*}
\begin{center}
\includegraphics[scale=0.33]{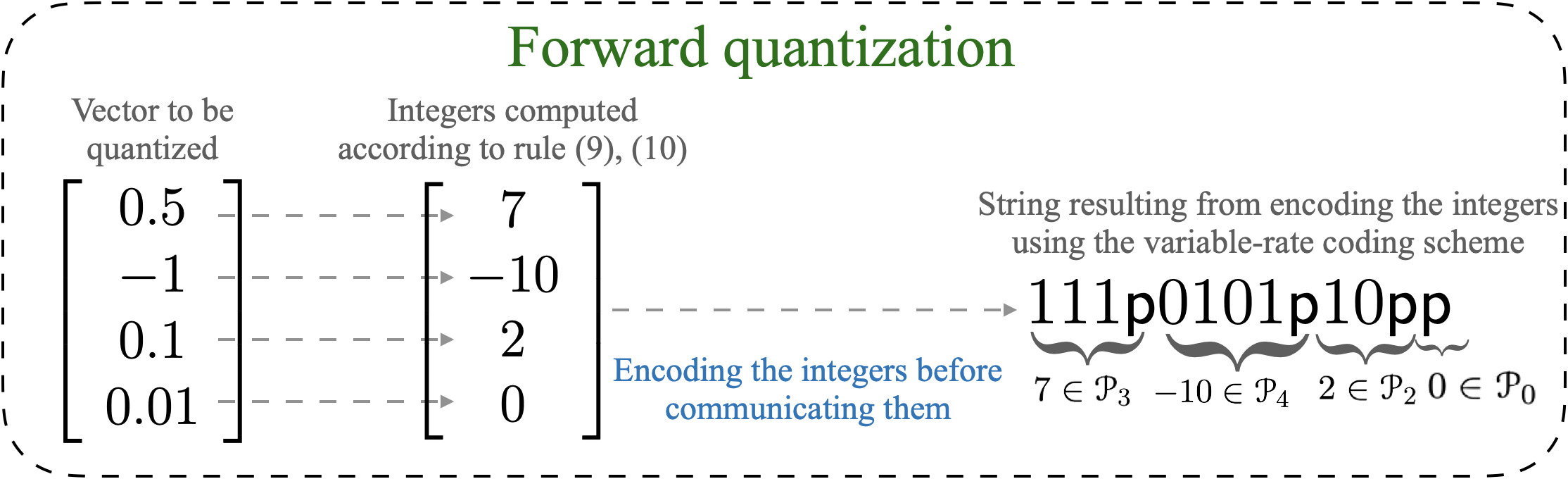}
\caption{{Illustration of the variable-rate coding scheme. The non-uniform randomized quantization scheme defined by~\eqref{eq: non-uniform quantization}--\eqref{eq: reproduction level} is used. The function  $g(\cdot)$ is chosen according to~\eqref{eq:generallogcomp} with $a=4.17$ and $b=10$%~\eqref{eq: non-linearities for ANQ} with $\omega=0.122$ and $\eta=0.1\omega$
. In this example, the codeword associated with the $j$-th element of the ordered set $\mathscr{P}_b$ is given by the $b$-bit binary representation of the integer $j-1$. For instance, the codewords associated with the first element in $\mathscr{P}_3$ (which is $-7$ according to~\eqref{eq: partition}) and the sixth element in $\mathscr{P}_4$ (which is $-10$ according to~\eqref{eq: partition 1}) are $000$ and $0101$, respectively. %it is assumed that encoding the ordered elements of the set $\mathscr{P}_b$ results in the following set of codewords (each of length $b$-bit) $\{ 0\ldots00~;~0\ldots01~;~0\ldots10~;~\ldots~;~1\ldots11\}$ where the first and the last codewords correspond to the integers  $-2^b+1$ and $2^{b}-1$, respectively. 
Subsequent integers are separated by a parsing symbol $\mathsf{p}$.}}
\label{fig: forward_quantization}
\end{center}
\end{figure*}
   %In Sec.~\ref{sec: variable-rate coding and bit rate stability}, we will show that, when applied in the context of decentralized learning under subspace constraints, these   quantizers  are able to guarantee a finite average bit-rate. %It was shown in~\cite{lee2021finite} that this issue can be circumvented by resorting to \emph{variable-rate} quantizers which are able to adaptively choose the number of bits to represent the input. In Sec.~\ref{sec: variable-rate coding and bit rate stability}, we will show that, when applied in the context of decentralized learning under subspace constraints, these   quantizers  are able to guarantee a finite average bit-rate. 
%For clarity of  presentation, we start by describing the variable-rate coding rule proposed in the work~\cite{lee2021finite}.  
Consider the following partition of the set of integers {$\mathbb{Z}$:
\begin{equation}
\label{eq: partition}
\begin{split}
\mathscr{P}_0=\{0\},&\quad\mathscr{P}_1=\{-1,1\},\quad\mathscr{P}_2=\{-3,-2,2,3\},\\
&\mathscr{P}_3=\{-7,-6,-5,-4,4,5,6,7\},\quad\ldots
\end{split}
\end{equation}
which can be written more compactly as:
\begin{equation}
\label{eq: partition 1}
\mathscr{P}_b=\left\lbrace\begin{array}{ll}
\{0\},&\text{if }b=0\\
\{-1,1\},&\text{if }b=1\\
\{-2^b+1,\ldots,-2^{b-1},2^{b-1},\ldots,2^{b}-1\},&\text{if }b\geq 2.\\
\end{array}
\right.
\end{equation}
%\begin{equation}
%\mathscr{P}_0=\{0\},~ \mathscr{P}_1=\{-1,1\},~\mathscr{P}_2=\{-3,-2,2,3\},\ldots,~\mathscr{P}_b=\{-2^b+1,\ldots,-2^{b-1},2^{b-1},\ldots,2^{b}-1\},\ldots
%\end{equation}
Note that} the ensemble of sets $\{\mathscr{P}_b|b=0,1,2,\ldots\}$ forms a partition of $\mathbb{Z}$ {where} each set $\mathscr{P}_b$ has cardinality equal to $2^b$. Thus, {an integer $n\in\mathscr{P}_b$} can be represented with a string of $b$ bits. {Accordingly, from~\eqref{eq: partition 1}, the number of bits to represent an integer  $n\in\mathbb{Z}$ can be determined according to}:%each index $n$ belonging to $\mathscr{P}_b$ can be represented with a string of $b$ bits. In particular, we have:
\begin{equation}
b=\lceil\log_2(|n|+1)\rceil,
\end{equation}
{where $\lceil \cdot\rceil$ denotes the ceiling function. When the rule~\eqref{eq: non-uniform quantization}--\eqref{eq: reproduction level} is implemented to quantize an input vector $x\in\mathbb{R}^L$, a sequence of integers will be generated according~\eqref{eq: non-uniform quantization} and~\eqref{eq: random index n}.  These integers must be encoded before being transmitted over the communication links{--see Fig.~\ref{fig: forward_quantization} for an illustration}.
%\cblue{For example, when the rule described in Sec.~\ref{sec: Non-uniform randomized quantizers} is implemented to quantize the input vectors $\bchi_{k,i}=\bpsi_{k,i}-\bphi_{k,i-1}$ in~\eqref{eq: stepb}, a sequence of integers will be generated according~\eqref{eq: non-uniform quantization} and~\eqref{eq: random index n}.  These integers must be encoded by agent~$k$ before being transmitted over the communication links\cblue{--see Fig.~\ref{fig: forward_quantization} for an illustration}.
 Thus, we}  need to encode sequences of integers $n_1,n_2,\ldots,$ with different integers belonging in general to different partitions $\mathscr{P}_b$. {At the receiver side, and since the transmitted integers are unknown, the decoder does not know the number of bits used for each integer.} % Obviously, the decoder does not know the number of bits used for each index (the indices are in fact the unknown information to be transmitted).} %Consequently,}  the string of bits obtained by concatenating the strings corresponding to each index would lead in general to an ambiguity that would impair correct decoding.
  In order to solve this issue, the work~\cite{lee2021finite} proposes to consider an overall encoder alphabet $\mathscr{S}$ made of two binary digits plus a parsing symbol $\mathsf{p}$, namely, $\mathscr{S}=\{0,1\}\cup\{\mathsf{p}\}$. In this way, we can encode a sequence of integers $n_1,n_2,\ldots,$ by first encoding each of the individual integers using a number of bits determined by the corresponding partition, and then using the parsing symbol to separate subsequent integers. For example, a string of the form {(see also Fig.~\ref{fig: forward_quantization})}:
 %$$0\in\mathscr{P}_0$$
 %$$111\mathsf{p}0101\mathsf{p}10\mathsf{p}\mathsf{p}$$
%\begin{equation}
%\underbrace{\mathsf{p}}_{n_1}\underbrace{01\mathsf{p}}_{n_2}\underbrace{0\mathsf{p}}_{n_3}
%\end{equation}
\begin{equation}
\underbrace{\mathsf{111\mathsf{p}}}_{n_1}\underbrace{0101\mathsf{p}}_{n_2}\underbrace{10\mathsf{p}}_{n_3}\underbrace{\mathsf{p}}_{n_4}
\end{equation}
{corresponds to integer $n_1\in\mathscr{P}_3$, followed by $n_2\in\mathscr{P}_4$, $n_3\in\mathscr{P}_2$, and $n_4=0$} (only the parsing symbol since partition $\mathscr{P}_0=\{0\}$ requires 0 bits).
Accounting for the parsing symbol, {observe that} the total number of {ternary digits (considering the ternary alphabet $\mathscr{S}$)} required to encode an integer~$n$ is: %symbols required to encode an index~$n$ is:
\begin{equation}
1+\lceil\log_2(|n|+1)\rceil,
\end{equation}
which corresponds to a number of bits equal {to:
\begin{equation}
\label{eq: bit budget for variable rate}
\log_2(3)(1+\lceil\log_2(|n|+1)\rceil)
\end{equation}
since one ternary digit is equivalent to $\log_2(3)$ bits of information.} Considering now the {quantization rule~\eqref{eq: non-uniform quantization}--\eqref{eq: reproduction level} which is} applied entrywise to a vector $x\in\mathbb{R}^{L}$, we find from~\eqref{eq: bit budget for variable rate} that the {overall (random)} bit budget is equal to:
\begin{equation}
\label{eq: bit rate general formula}
\boldsymbol{r}(x)=\log_2(3)\sum_{j=1}^L(1+\lceil\log_2(|\boldsymbol{n}(x_j)|+1)\rceil).
\end{equation}

\section{Decentralized learning in the presence of quantized communication}
\label{sec: Decentralized learning in the presence of quantized communications}
%================================
% Subsec: Differential randomized quantization approach
%================================
\subsection{Differential randomized quantization approach}
\begin{figure*}
\begin{center}
\includegraphics[scale=0.4]{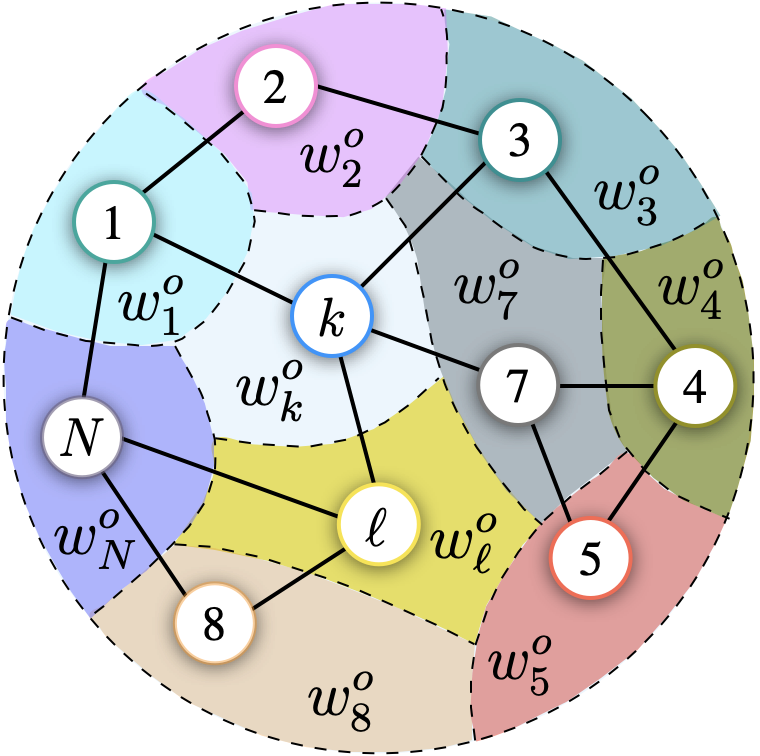}\qquad\qquad
\includegraphics[scale=0.3]{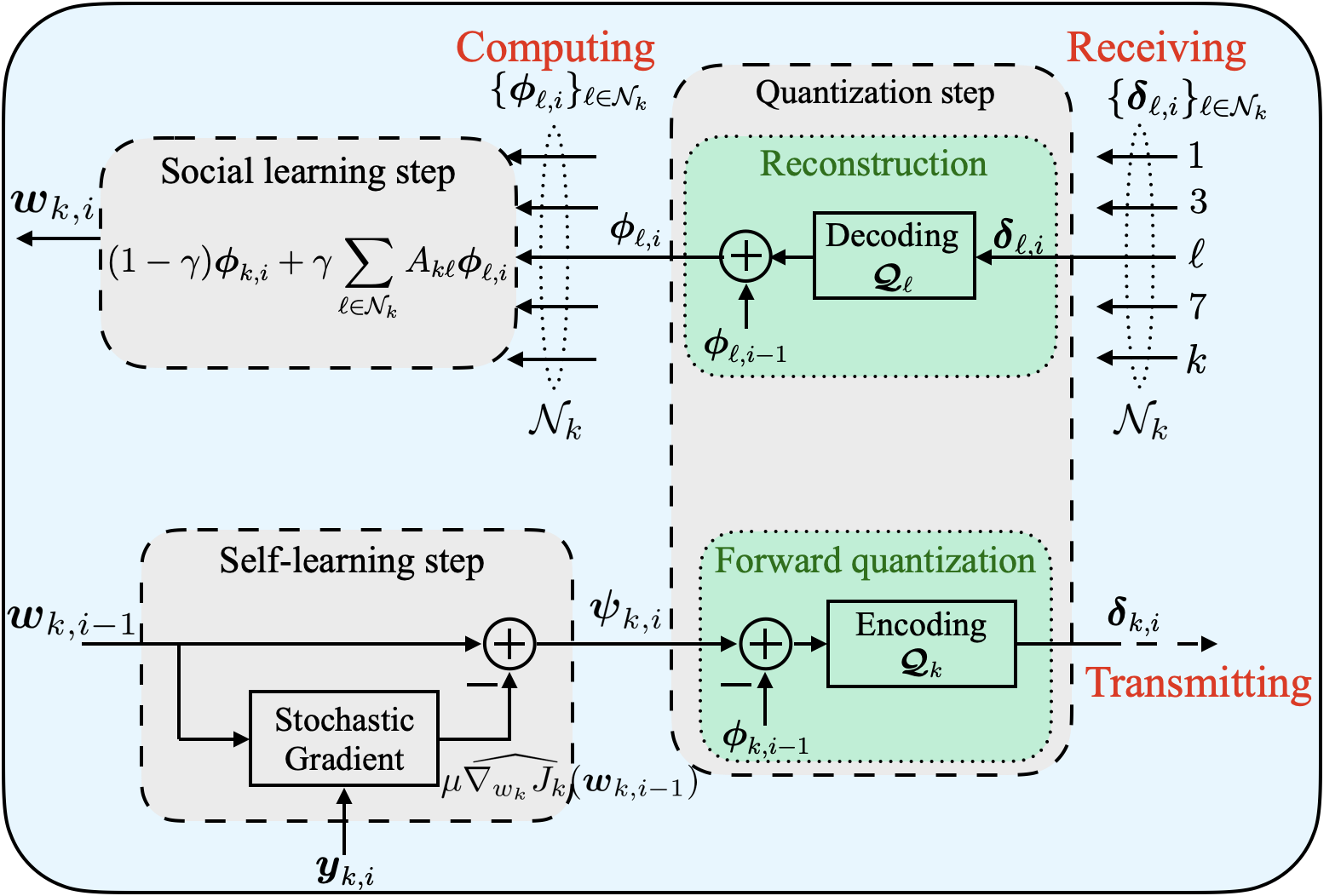}
\caption{\textit{(Left)} An illustration of a multitask network~\cite{nassif2020multitask}. The objective at agent $k$ is to estimate $w^o_k$ (of dimension $M_k\times 1$), the  $k$-th subvector of $\cw^o$ in~\eqref{eq: network constrained problem}. In this example, the neighborhood set of agent $k$ is given by $\cN_k=\{1,k,3,\ell,7\}$. \textit{(Right)} The implementation of the quantized multitask approach~\eqref{eq: decentralized learning approach with quantization} at agent $k$. The quantization step consists of (i) the \emph{forward} step where agent $k$ encodes the difference $\bpsi_{k,i}-\bphi_{k,i-1}$ and sends the resulting vector $\bdelta_{k,i}$ (sequence of symbols or bits) to its neighbors, and (ii) the \emph{reconstruction} step where agent $k$ receives the encoded vectors $\{\bdelta_{\ell,i}\}_{\ell\in\cN_k}$ from its neighbors and decodes them according to~\eqref{eq: quantization function-second stage} to obtain $\{\bphi_{\ell,i}\}_{\ell\in\cN_k}$. The resulting vectors are then used in the social learning step~\eqref{eq: stepc}.
}
\label{fig: illustration figure}
\end{center}
\end{figure*}
Motivated by the approaches proposed \cblue{in~\cite{koloskova2019decentralized,carpentiero2022adaptive,carpentiero2022distributed,kovalev2021linearly,lee2021finite,taheri2020quantized,zhang2019compressed,tang2018communication,sparq2023sparq,singh2021squarm}}, we equip~\eqref{eq: decentralized learning approach} with a quantization mechanism by proposing the following decentralized  multitask learning  approach--see Fig.~\ref{fig: illustration figure}:
\begin{subequations}
\label{eq: decentralized learning approach with quantization}
 \begin{empheq}[left={\empheqlbrace\,}]{align}
\bpsi_{k,i}&=\bw_{k,i-1}-\mu\widehat{\nabla_{w_k}J_k}(\bw_{k,i-1})\label{eq: stepa}\\
\bphi_{k,i}&=\bphi_{k,i-1}+\bcQ_k(\bpsi_{k,i}-\bphi_{k,i-1})\label{eq: stepb}\\
\bw_{k,i}&=(1-\gamma) \bphi_{k,i}+\gamma\sum_{\ell\in\cN_k}A_{k\ell}\bphi_{\ell,i}\label{eq: stepc}
 \end{empheq}
\end{subequations}
where $\gamma\in(0,1]$ is a \emph{mixing} parameter that tunes the degree of cooperation between agents, and $\bcQ_k(\cdot)$ is a  \emph{randomized} quantizer {satisfying Property~\ref{property: quantization noise}}. {Note that, since the quantizer characteristics can vary with $k$, the randomized quantizer becomes $\bcQ_k(\cdot)$ instead of $\bcQ(\cdot)$ with a subscript $k$ added to $\bcQ$. }% quantization operator  (a map from real valued vectors to a finite set of quantized vectors) 
%at agent~$k$
%. First, observe that  randomized quantizers $\bcQ_k(\cdot)$\footnote{\cblue{Since the output of a randomized quantizer  is random even for deterministic input, we use  the boldface notation $\bcQ_k(\cdot)$ to refer to randomized quantizers. A detailed description of the randomized quantization mechanism will be provided in Sec.~\ref{sec: rendomized quantizer}. %As we will see, in randomized quantizer, given an input value, the output level is randomly assigned through a probabilistic mechanism that is specified in terms of a probability function.
%}}  are used in~\eqref{eq: stepb} instead of deterministic quantizers $\cQ_k(\cdot)$. As explained in~\cite{carpentiero2022distributed}, deterministic quantizers can lead to severe estimation biases in inference problems. To overcome this issue,  randomized quantizers $\bcQ_k(\cdot)$ %\cblue{(which are described in Sec.~\ref{subsec: rendomized quantizer})} 
%are commonly used  to compensate for the bias (on average, over time)~\cite{koloskova2019decentralized,carpentiero2022adaptive,carpentiero2022distributed,kovalev2021linearly,lee2021finite,taheri2020quantized}. 
Observe that  \emph{differential quantization} is used in~\eqref{eq: stepb} to leverage possible correlation between subsequent iterates~\cite{carpentiero2022distributed}. In this case, instead of communicating compressed versions of the estimates $\bpsi_{k,i}$,  the \emph{prediction} error $\bpsi_{k,i}-\bphi_{k,i-1}$ is quantized at agent $k$ and then \cblue{transmitted~\cite{koloskova2019decentralized,carpentiero2022distributed,carpentiero2022adaptive,lee2021finite,taheri2020quantized,kovalev2021linearly,zhang2019compressed,sparq2023sparq,singh2021squarm,tang2018communication}}.  
At each iteration~$i$, agent $k$ performs the \emph{forward quantization} (see Fig.~\ref{fig: illustration figure}) by mapping the real-valued vector $\bpsi_{k,i}-\bphi_{k,i-1}$ into a quantized vector $\bdelta_{k,i}$, sends $\bdelta_{k,i}$ to its neighbors through \cblue{{ideal}} communication links \cblue{(i.e., it is assumed that node $k$ can transmit perfectly and reliably $\bdelta_{k,i}$ to its neighbors)}, receives $\{\bdelta_{\ell,i}\}$ from its neighbors $\ell\in\cN_k$, and performs the \emph{reconstruction}  (see Fig.~\ref{fig: illustration figure}) on each received vector $\bdelta_{\ell,i}$ by first decoding it and then computing $\{\bphi_{\ell,i}\}$ according to step~\eqref{eq: stepb}:
\begin{equation}
\label{eq: quantization function-second stage}
\bphi_{\ell,i}=\bphi_{\ell,i-1}+ \bcQ_{\ell}(\bpsi_{\ell,i}-\bphi_{\ell,i-1}),\qquad\ell\in\cN_k.
\end{equation}
Observe that implementing~\eqref{eq: quantization function-second stage} requires storing the \cblue{\emph{predictors}} $\{\bphi_{\ell,i-1}\}_{\ell\in\cN_k}$ by agent $k$. The reconstructed vectors $\{\bphi_{\ell,i}\}$ are then combined according to~\eqref{eq: stepc} to produce the estimate $\bw_{k,i}$. \cblue{As we will see in Secs.~\ref{sec: Stochastic performance analysis} and~\ref{sec: Theorems insights and observations}, in the presence of \emph{relative} quantization noise (i.e., when $\beta^2_{q}\neq0$), the mixing parameter $\gamma$ in~\eqref{eq: stepc} will be used to control the network stability. However, in the absence of \emph{relative} noise (i.e., when $\beta^2_{q}=0$), the quantization does not affect the network stability, and the parameter $\gamma$ can be set  to one.} 

\noindent\textbf{Remark 1 {(Quantized diffusion-type approach)}}: \cblue{Most prior literature on decentralized quantized learning focuses primarily on single-task consensus  optimization} %works  in the decentralized learning  literature For illustration purposes, we explain in this remark how to choose the blocks $\{A_{k\ell}\}$ in~\eqref{eq: decentralized learning approach with quantization} in order to solve consensus-type optimization problems} 
where the objective at each agent is to estimate {the  \emph{same}} $L$-th dimensional vector $w^o$ given by:
\begin{equation}
\label{eq: consensus optimization}
w^o=\arg\min_w\sum_{k=1}^NJ_k(w).
\end{equation} 
\cblue{For this reason, and in order to make the paper self-contained for readers interested in problem~\eqref{eq: consensus optimization}, we explain in this remark how to choose the blocks $\{A_{k\ell}\}$ in~\eqref{eq: decentralized learning approach with quantization} in order to solve problem~\eqref{eq: consensus optimization} as well}. %Approach~\eqref{eq: decentralized learning approach with quantization} can be applied to solve consensus-type optimization where the objective at each agent is to estimate {the  same} $L$-th dimensional vector $w^o$ given by:
%\begin{equation}
%\vspace{-1.5mm}
%\label{eq: consensus optimization}
%w^o=\arg\min_w\sum_{k=1}^NJ_k(w)
%\end{equation} 
%in a decentralized manner.
 In fact, by setting in~\eqref{eq: network constrained problem} $P=L$ and $\cU=\frac{1}{\sqrt{N}}(\mathds{1}_N\otimes I_L)$ where $\mathds{1}_N$ is the $N\times 1$ vector of all ones, then solving problem~\eqref{eq: network constrained problem} will be equivalent to solving the well-studied consensus problem~\eqref{eq: consensus optimization}. Different algorithms for solving~\eqref{eq: consensus optimization} over strongly-connected networks have been proposed~\cite{sayed2014adaptation,sayed2014adaptive,sayed2013diffusion,bertsekas1997new,nedic2009distributed,dimakis2010gossip}. By choosing an $N\times N$ doubly-stochastic matrix $A=[a_{k\ell}]$ satisfying:
\begin{equation}
\label{eq: doubly stochastic condition}
a_{k\ell}\geq 0,\quad A\mathds{1}_N=\mathds{1}_N,\quad\mathds{1}_N^\top A= \mathds{1}_N^\top,\quad a_{k\ell}=0\text{ if }\ell\notin\cN_k,
\end{equation}
the \emph{diffusion} strategy for instance {would take} the form~\cite{sayed2014adaptation,sayed2014adaptive,sayed2013diffusion}:
\begin{subequations}
\label{eq: diffusion strategy}
 \begin{empheq}[left={\empheqlbrace\,}]{align}
\bpsi_{k,i}&=\bw_{k,i-1}-\mu\widehat{\nabla_{w_k}J_k}(\bw_{k,i-1})\label{eq: diffusion strategy step1}\\
\bw_{k,i}&=\sum_{\ell\in\cN_k}a_{k\ell}\bpsi_{\ell,i}\label{eq: diffusion strategy step2}
 \end{empheq}
\end{subequations}
This strategy can be written in the form of~\eqref{eq: decentralized learning approach} with $A_{k\ell}=a_{k\ell}I_L$ and $\cA=A\otimes I_L$. It can be verified that, when $A$ satisfies~\eqref{eq: doubly stochastic condition} over a strongly connected network, the matrix $\cA$ will satisfy~\eqref{eq: condition A}. Consequently, by setting $A_{k\ell}=a_{k\ell}I_L$ in~\eqref{eq: stepc} with a set of combination coefficients $\{a_{k\ell}\}$ satisfying~\eqref{eq: doubly stochastic condition}, we obtain a new \emph{quantized diffusion-type approach} for solving the consensus  optimization problem~\eqref{eq: consensus optimization}. %In this case, it turns out that the works~\cite{carpentiero2022adaptive,carpentiero2022distributed} employing differential quantization in the context of diffusion adaptation strategies can be considered as the closest to the current work. However, as we will point out in Sec.~\ref{subsec: Modeling assumptions}, the resulting  quantized diffusion-type approach~\eqref{eq: decentralized learning approach with quantization} and its analysis differ from the works~\cite{carpentiero2022adaptive,carpentiero2022distributed} in three different ways.
\hfill\qed

%\cblue{In Sec.~\ref{sec: Stochastic performance analysis}, we shall analyze the differential randomized quantization approach~\eqref{eq: decentralized learning approach with quantization} under a general class of randomized quantizers, which is described in the next section.} %We start by describing the class of randomized quantizers for deterministic inputs, and then we show how the randomness of the quantizers interplays with the randomness of the quantizer inputs $\{\bpsi_{k,i}-\bphi_{k,i-1}\}$ generated by algorithm~\eqref{eq: decentralized learning approach with quantization}.}  

%==============================
% Subsec: related work
%==============================
\subsection{Related work}
\label{sec: related work}

Table~\ref{table: examples of quantizers}  provides a list of  typical compression schemes with the corresponding  quantization noise parameters $\beta^2_{q}$ and $\sigma^2_{q}$. %However, Table~\ref{table: examples of quantizers} misses an important information, since whenever we talk about a compression scheme, we need to quantify the amount of compression that such a scheme achieves. }
By comparing the reported schemes, we first observe that the ``rand-c'', ``randomized gossip'', and ``gradient sparsifier'' do not really quantize their input. These methods basically map a full vector into a sparse version thereof. In other words, these methods assume that the non-zero vector components are represented with very high (e.g., machine) precision, and that the overall compression gain lies in representing few entries of the input vector. {For instance, under the  ``rand-c'' scheme, $c$ randomly selected  components of the input vector are encoded with very high precision ($32$ or $64$ bit are typical values for encoding a scalar), and then the resulting bits are communicated over the links in addition to the bits encoding the locations of the selected components. Such} schemes should be more properly referred to as \emph{compression operators}, rather than \cblue{quantizers~\cite{kovalev2021linearly,lee2021finite,koloskova2019decentralized}}. The idea behind sparsification operators is that, when the number of vector components is large, {the gain resulting from encoding  a \textit{few} randomly selected  components compensates for the high precision required to represent them.}   % the high precision required to encode the non-zero components is far more than compensated by the gain resulting from encoding a few components.   
The QSGD scheme in Table~\ref{table: examples of quantizers} uses a different approximation rule. It assumes that the norm of the input vector is represented with high precision. {In addition to encoding the norm, this rule requires $L$-bits to encode the signs of the vector components and $L\lceil \log_2(s)\rceil$ to encode the levels. It} %Bits are then allocated across the vector components assuming that the norm is known. 
tends  to be well-suited {for} high-dimensional settings where the number of components is large. %Consequently, this rule 
%requires $B_\text{FP}$-bits to encode the norm, plus $L$-bits to encode the signs, plus $L\lceil \log_2(s)\rceil$ to encode the levels, and tends also
{As it can be observed from Table~\ref{table: examples of quantizers}, the sparsity-based schemes and the QSGD scheme  have an \emph{absolute noise} term $\sigma^2_{q}=0$. %\cblue{It is shown in~\cite[Corollary~9]{lee2021finite} that the design of unbiased quantizers satisfying condition~\eqref{eq: expectation squared of the quantization noise random 1} with $\sigma^2_q=0$ using a finite number of quantization points is not feasible. In general, neglecting the effect of $\sigma^2_{q}$ requires that some quantities (e.g., the norm of the input vector  in QSGD) are represented with no quantization error.}%
As explained in~\cite{lee2021finite}, neglecting the effect of $\sigma^2_{q}$ requires that some quantities (e.g., the norm of the input vector  in QSGD) are represented with no quantization error\cblue{\footnote{\cblue{In particular, it is shown in~\cite[corollary~9]{lee2021finite} that the design of unbiased quantizers satisfying condition~\eqref{eq: expectation squared of the quantization noise random 1} with $\sigma^2_q=0$ using a finite number of quantization points is not feasible.}}}.} 
{In comparison, the probabilistic uniform and ANQ schemes \emph{do not make any assumptions on the high-precision quantization of specific variables} and   have a \emph{non-zero absolute noise} term~$\sigma^2_{q}$.}

{A detailed analysis of the decentralized strategy~\eqref{eq: decentralized learning approach with quantization} in the presence of both the \emph{relative} (captured through $\beta^2_{q}$) and  the \emph{absolute} (captured through $\sigma^2_{q}$) quantization noise terms will be conducted in Sec.~\ref{sec: Stochastic performance analysis}. Consequently, our  results apply to a general class of quantizers satisfying Property~\ref{property: quantization noise}, including all those listed in Table~\ref{table: examples of quantizers}. }
Compared to prior works, we provide the following  advances. First, the analysis results  are novel compared to the \cblue{works~\cite{koloskova2019decentralized,carpentiero2022adaptive,carpentiero2022distributed,alistarch2017qsgd,kovalev2021linearly,sparq2023sparq,singh2021squarm}} since we do not make any assumptions about the high-precision quantization of specific variables. \cblue{Although no such assumptions are made  in the work~\cite{tang2018communication}, the analysis there is performed in the absence of the relative quantization and gradient noise terms. }%Regarding the work~\cite{lee2021finite} (which also avoids the high-precision approximations),  we use \cblue{in Sec.~\ref{sec: variable-rate coding and bit rate stability}} the variable-rate quantization scheme proposed therein. However, our results are novel as compared to~\cite{lee2021finite}\footnote{\cmag{Professor: We are studying here a different algorithm than the one studied in~\cite{lee2021finite}.
{Our results are also novel in comparison to~\cite{lee2021finite}} %\footnote{\cmag{Professor: We are studying here a different algorithm than the one studied in~\cite{lee2021finite}.  In fact, our settings are different. First, they are studying consensus-based optimization. Second, they are considering deterministic optimization, so they don't have the gradient approximation or the mean-square error stability. They focus instead on establishing some convergence rate properties. Their analyses results cannot be applied in our settings.}} 
 since we consider: i)  learning under \emph{subspace constraints}; ii)  a \emph{stochastic} optimization setting with non-diminishing step-size; iii) combination policies that can lead to \emph{non-symmetric} combination matrices; and iv) a network \emph{global}, as opposed to local, strong convexity condition on the individual costs -- see condition~\eqref{eq: hessian} further ahead.  {Moreover, the work~\cite{lee2021finite} considers only deterministic {optimization} and {does} not deal with gradient approximation or mean-square error stability. In} addition, it must be noted that even the \cblue{works~\cite{koloskova2019decentralized,carpentiero2022adaptive,carpentiero2022distributed,alistarch2017qsgd,kovalev2021linearly,sparq2023sparq,singh2021squarm,tang2018communication}}  do not consider our more general setting that simultaneously addresses i), ii), iii),~and~iv). 
 
 \cblue{As the derivations in the next section will reveal, characterizing the behavior of the decentralized learning system under the aforementioned conditions is challenging  due to at least two factors. First, the quantization noise that  interferes with the operation of the algorithm. It is therefore important to assess its impact on the algorithm's performance and stability by extending the mean-square analyses of the works~\cite{sayed2014adaptation,sayed2014adaptive,nassif2020learning,nassif2020adaptation}, which consider decentralized learning  in the absence of quantization. Second, as we will see later, when we remove the common assumption about the high precision quantization of specific variables (in which case we use variable-rate quantization as suggested in Sec.~\ref{subsec: Variable-rate coding and bit budget}), then the quantizer resolution is required to increase as $\mu\rightarrow 0$ to guarantee small mean-square errors. %$i)$ the quantizer resolution is increased as $\mu\rightarrow 0$; and $ii)$ we use variable-rate quantization. 
 It therefore becomes important to establish the bit rate stability of the network, namely, to show that the average bit rate remains finite as $\mu\rightarrow 0$. To the best of our knowledge, this analysis has not been carried out before. Existing works investigate either fixed-rate quantizers~\cite{carpentiero2022distributed,carpentiero2022adaptive,koloskova2019decentralized,kovalev2021linearly} or variable-rate quantizers %in the presence of diminishing step-size sequences or
  without streaming data~\cite{lee2021finite}.
 }
  %the bit budget which, for consistency, is required to remain  finite when variable-rate quantizers are used. It therefore becomes important to establish the bit rate stability of the network.  Such analyses have never been carried out in the context of learning with constant step-sizes. Existing works  investigate either variable-rate quantizers in the presence of diminishing step-size sequences or without streaming data~\cite{lee2021finite}, or rely on the use of fixed-rate quantizers.}

%%==============================
%% Subsec: Main contributions
%%==============================
%\subsection{Main contributions}
 In summary, we provide the following main contributions.
\begin{itemize}
\item We propose a decentralized strategy for learning and adaptation over networks under subspace constraints. This strategy is able to operate under finite communication rate by employing \emph{differential randomized quantization}.
\item We provide a detailed characterization of the {proposed approach} for a {general} class of quantizers and compression operators satisfying {Property~\ref{property: quantization noise}, both in terms of mean-square stability and communication resources}. 
\item {The analysis reveals the following useful conclusions%\footnote{\cmag{Professor: The analysis covers general probabilistic quantizers satisfying Property 1. If the network communication constraints allow for sending certain variables with high precision, then the designer can use the class of quantizers with $\sigma^2_{q}=0$. The bit rate is controllable in this case and does not depend on the input of the quantizer as it can be observed in Table 1 (column 5, rows 5--8). An illustration that uses QSGD was also added to the simulation section to clarify the bit rate in the QSGD case for instance. Otherwise, if the designer cannot afford sending high-precision quantized variables, we can employ the variable-rate quantizers described in Sec.~\ref{subsec: Uniform and non-uniform randomized quantizers}. Interestingly, we showed that, for these quantizers, the bit rate does not diverge even in the small adaptation regime, since the quantizer automatically adapts the number of bits according to the effective range of the input.}}
. First,   the quantization error does not impair the network mean-square stability.  Second, in the absence of the absolute quantization noise term, and at the expense of communicating some quantities with high precision, the iterates generated by the quantized approach~\eqref{eq: decentralized learning approach with quantization} lead to small estimation errors on the order of $\mu$, as it happens in the unquantized case~\eqref{eq: decentralized learning approach}. In the presence of the absolute noise term, the situation becomes more challenging. The analysis reveals that, to guarantee the $O(\mu)$ mean-square-error behavior, the absolute noise term must converge to zero as $\mu\rightarrow 0$. We prove that this result can be achieved with a bit rate that remains {\em bounded} as $\mu\rightarrow 0$, despite the fact that we are requiring an increasing precision as the step-size decreases. In particular, we illustrate one useful strategy (the variable-rate quantization scheme described in Sec.~\ref{subsec: Uniform and non-uniform randomized quantizers}) that achieves the aforementioned goals. }
\end{itemize}

%==============================
% Sec: Stochastic performance analysis
%==============================
\section{Mean-square and bit rate stability analysis}
\label{sec: Stochastic performance analysis}

%==============================
% Subsec: Modeling assumptions
%==============================
\subsection{Modeling assumptions}
\label{subsec: Modeling assumptions}
In this section, we analyze strategy~\eqref{eq: decentralized learning approach with quantization} with a matrix $\cA$ satisfying~\eqref{eq: condition A}  by examining the average squared distance between $\bw_{k,i}$ and $w^o_k$, namely, $\expec\|w^o_k-\bw_{k,i}\|^2$, under the following %Assumption~\ref{assump: quantization noise} and the following 
assumptions on the risks $\{J_k(\cdot)\}$, the gradient noise processes $\{\bs_{k,i}(\cdot)\}$ {defined by~\cite{sayed2014adaptation}}:
\begin{equation}
\label{eq: gradient noise process}
\bs_{k,i}(w_k)\triangleq \nabla_{w_k}J_k(w_k)-\widehat{\nabla_{w_k}J_k}(w_k),
\end{equation}
and the randomized quantizers $\{\bcQ_k(\cdot)\}$.

\begin{assumption}{{\emph{\textbf{(Conditions on individual and aggregate costs)}}.}}
\label{assump: assumption of the individual costs}
The individual costs $J_k(w_k)$ are assumed to be twice differentiable {and convex} such that:
\begin{equation}
\label{eq: individual costs}
\lambda_{k,\min}I_{M_k}\leq \nabla^2_{w_k}J_k(w_k)\leq \lambda_{k,\max}I_{M_k},
\end{equation}
{where $\lambda_{k,\min}\geq 0$} for $k=1,\ldots,N$. It is further assumed that, for any $\{w_k\in\mathbb{R}^{M_k}\}$, the  individual costs  satisfy:
\begin{equation}
\label{eq: hessian}
0<\lambda_{\min} I_{P}\leq \cU^\top\emph{\diag}\left\{\nabla^2_{w_k}J_k(w_k)\right\}_{k=1}^N\cU\leq \lambda_{\max} I_{P},
\end{equation}
for some positive parameters $\lambda_{\min}\leq \lambda_{\max}$.\hfill\qed
\end{assumption}
\noindent As explained in~\cite{nassif2020adaptation}, condition~\eqref{eq: hessian} ensures that problem~\eqref{eq: network constrained problem} has a unique minimizer $\cw^o$.

\begin{assumption}{{\emph{\textbf{(Conditions on gradient noise)}}.}}
\label{assump: gradient noise}
The gradient noise process defined in~\eqref{eq: gradient noise process} satisfies for $k=1,\ldots,N$%\footnote{\cmag{Professor: Conditioning on $\bphi$ is necessary for (85) since $\bx_{i-1}$ depends on $\bphi$. Also, if we know  $\bphi$, we can compute $\bw$ since $\bw$ is a deterministic (linear) function of $\bphi$ according to (28c). }}
:
\begingroup
\allowdisplaybreaks\begin{align}
\expec[\bs_{k,i}(\bw_{k,i-1})|\{{\bphi_{\ell,i-1}}\}_{\ell=1}^N]&=0,\label{eq: expectation of the gradient noise}\\
\expec[\|\bs_{k,i}(\bw_{k,i-1})\|^2|\{{\bphi_{\ell,i-1}}\}_{\ell=1}^N]&\leq\beta^2_{s,k}\|\bwt_{k,i-1}\|^2+\sigma^2_{s,k},\label{eq: expectation squared of the gradient noise}
\end{align}
\endgroup
for some $\beta^2_{s,k}\geq 0$ and $\sigma^2_{s,k}\geq 0$.\hfill\qed%, and where $\bcF_{i-1}$ denotes the filtration generated by the random processes $\{{\bphi_{\ell,j}}\}$ for all $\ell=1,\ldots,N$ and $j\leq i-1$. %$\bcF_{i-1}$ denotes the filtration generated by the random processes $\{\bw_{\ell,j},{\bphi_{\ell,j}}\}$ for all $\ell=1,\ldots,N$ and $j\leq i-1$.
\end{assumption}
\noindent As explained in~\cite{sayed2014adaptation,sayed2014adaptive,sayed2013diffusion}, these conditions are satisfied  by many  cost functions of interest in learning and adaptation such as quadratic and  regularized logistic risks.  Condition~\eqref{eq: expectation of the gradient noise} states that the gradient  approximation should be unbiased conditioned on the \cblue{predictors $\{\bphi_{\ell,i-1}\}_{\ell=1}^N$}. Condition~\eqref{eq: expectation squared of the gradient noise} states that the second-order moment of the gradient noise should get smaller for better estimates, since it is bounded by the squared norm of the iterate.

\begin{assumption}{\emph{\textbf{{(Conditions on quantizers)}}.}}
\label{assump: quantization noise}
%{For any deterministic input $x\in\mathbb{R}^L$ \cblue{with $L$ representing a generic vector length}, the randomized quantizers $\{\bcQ_k(\cdot)\}$ in~\eqref{eq: stepb} are unbiased and satisfy:
%%\begin{equation}
%%p_k(q;x)=\mathbb{P}[\bcQ_k(x)=q],\qquad k=1,2,\ldots,N,
%%\label{eq:qpmf}
%%\end{equation}
%%for any $q$ belonging to the set of reproduction levels of the quantizer. These probability mass functions satisfy the following two conditions: 
%%\textcolor{blue}{we could say in the notation section that the symbol $0$ is occasionally used to denote a vector with all zero entries.}
%\begin{align}
%\expec[x-\bcQ_k(x)]&=0,\label{eq: expectation of the quantization noise deterministic}\\
%\cblue{\expec\|x-\bcQ_k(x)\|^2}&\leq \beta^2_{z,k}\|x\|^2+\sigma^2_{z,k},\label{eq: expectation squared of the quantization noise deterministic}
%\end{align}
%for some $\beta^2_{z,k}\geq 0$ and $\sigma^2_{z,k}\geq 0$, and where the expectations are evaluated w.r.t. the randomness of $\bcQ_k(\cdot)$.
{In step~\eqref{eq: stepb} of the learning approach, each agent $k$ at time  $i$ applies to the difference $\bchi_{k,i}=\bpsi_{k,i}-\bphi_{k,i-1}$ a randomized quantizer $\bcQ_k(\cdot)$ satisfying Property~\ref{property: quantization noise} with quantization noise parameters $\beta^2_{q,k}$ and  $\sigma^2_{q,k}$. It is assumed that given the past history,   the randomized quantization mechanism depends only on the quantizer input $\bchi_{k,i}$. %Formally, 
%Each agent $k$ at time $i$ applies a randomized quantizer $\bcQ_k(\cdot)$ to the input $\bchi_{k,i}=\bpsi_{k,i}-\bphi_{k,i-1}$ in~\eqref{eq: stepb}, 
 %When the quantizers are applied to random inputs as in~\eqref{eq: decentralized learning approach with quantization}, the randomized quantization mechanism is characterized as follows. Consider the quantizer input 
%\begin{equation}
%$\bchi_{k,i}=\bpsi_{k,i}-\bphi_{k,i-1}$ in~\eqref{eq: stepb}.
%\label{eq:qinputdef}
%\end{equation}
Consequently, from~\eqref{eq: expectation of the quantization noise random 1} and~\eqref{eq: expectation squared of the quantization noise random 1}, we~get:%we find that the the conditional first and second-order moments of the quantization error are characterized as follows:
\begin{align}
\expec[\bchi_{k,i}-\bcQ_k(\bchi_{k,i})|\boldh_i]&=\expec[\bchi_{k,i}-\bcQ_k(\bchi_{k,i})|\bchi_{k,i}]=0,\label{eq: expectation of the quantization noise random}\\
\expec\left[\|\bchi_{k,i}-\bcQ_k(\bchi_{k,i})\|^2|\boldh_i\right]&=\expec\left[\|\bchi_{k,i}-\bcQ_k(\bchi_{k,i})\|^2|\bchi_{k,i}\right]\notag\\
&\leq \beta^2_{q,k}\|\bchi_{k,i}\|^2+\sigma^2_{q,k},\label{eq: expectation squared of the quantization noise random}
\end{align}
where  $\boldh_i$ is the  vector collecting all iterates generated by~\eqref{eq: decentralized learning approach with quantization} before the quantizer is applied to $\bchi_{k,i}$, namely, 
%\begin{equation}
$\Big\{
\{\bphi_{\ell,j}\}_{j=0}^{i-1},\{\bpsi_{\ell,j}\}_{j=1}^{i}
\Big\}_{\ell=1}^N$
%\end{equation}
.}\hfill\qed
\end{assumption}
\noindent\cblue{Assumption~\ref{assump: quantization noise} formalizes the quantization mechanism (which is in fact a standard mechanism in decentralized learning with randomized quantization) and highlights two fundamental relations~\eqref{eq: expectation of the quantization noise random} and~\eqref{eq: expectation squared of the quantization noise random} that will be useful in the technical derivations.}

\subsection{Network error vector recursion}
 In order to examine the evolution of the iterates $\bw_{k,i}$ generated by algorithm~\eqref{eq: decentralized learning approach with quantization} with respect to the minimizer $\cw^o=\col\{w^o_k\}_{k=1}^N$ defined in~\eqref{eq: network constrained problem}, we start by deriving the network error vector recusion. Let $\bwt_{k,i}=w^o_k-\bw_{k,i}$, $\bpsit_{k,i}=w^o_k-\bpsi_{k,i}$, and $\bphit_{k,i}=w^o_k-\bphi_{k,i}$. Using~\eqref{eq: gradient noise process}  and the mean-value theorem~\cite[pp.~24]{polyak1987introduction},~\cite[Appendix~D]{sayed2014adaptation}, we can express the stochastic gradient vector appearing in~\eqref{eq: stepa} as follows:
\begin{equation}
\label{eq: stochastic gradient vector-mean-value}
\widehat{\nabla_{w_k}J_k}(\bw_{k,i-1})=-\bH_{k,i-1}\bwt_{k,i-1}+b_k-\bs_{k,i}(\bw_{k,i-1})
\end{equation}
where  $\bH_{k,i-1}\triangleq \int_{0}^1\nabla^2_{w_k}J_k(w^o_k-t\bwt_{k,i-1})dt$ and $b_k\triangleq\nabla_{w_k}J_k(w^o_k)$.  By subtracting $w^o_k$ from both sides of~\eqref{eq: stepa}  and by introducing the following network quantities:
\begingroup
\allowdisplaybreaks\begin{align}
b&\triangleq\col\left\{b_1,\ldots,b_N\right\},\label{eq: equation for b}\\
\bs_i&\triangleq\col\left\{\bs_{1,i}(\bw_{1,i-1}),\ldots,\bs_{N,i}(\bw_{N,i-1})\right\},\label{eq: equation for s}\\
\bcH_{i-1}&\triangleq\diag\left\{\bH_{1,i-1},\ldots,\bH_{N,i-1}\right\},\label{eq: equation for H}\\
\bcwt_{i-1}&\triangleq\col\{\bwt_{1,i-1},\ldots,\bwt_{N,i-1}\},\label{eq: collection of the network error vector}
\end{align}
\endgroup
we can show that the network error vector $\bpsit_i=\col\{\bpsit_{k,i}\}_{k=1}^N$ evolves according to:
\begin{equation}
\bpsit_i=\left(I_M-\mu\bcH_{i-1}\right)\bcwt_{i-1}-\mu\bs_i+\mu b.\label{eq: network error vector psi}
\end{equation}
By subtracting $w^o_k$  from both sides of~\eqref{eq: stepc}, by replacing $w^o_k$ by $(1-\gamma) w^o_k+\gamma w^o_k$, and by using  $w^o_k=\sum_{\ell\in\cN_k}A_{k\ell}w^o_{\ell}$~\cite[Sec. III-B]{nassif2020adaptation}, we obtain:
\begin{equation}
\label{eq: evolution of wtk i}
\begin{split}
\bwt_{k,i}&=(1-\gamma) \bphit_{k,i}+\gamma\sum_{\ell\in\cN_k}A_{k\ell}\bphit_{\ell,i}.
\end{split}
\end{equation}
From~\eqref{eq: evolution of wtk i}, we can show {that the} network error vector $\bcwt_{i-1}$ in~\eqref{eq: collection of the network error vector} evolves according to:
\begin{equation}
\bcwt_{i-1}=(1-\gamma)\bphit_{i-1}+\gamma\cA\bphit_{i-1}=\cA'\bphit_{i-1},\label{eq: wt in terms of phi}
\end{equation} 
where 
\begin{eqnarray}
\bphit_{i}&\triangleq\col\{\bphit_{1,i},\ldots,\bphit_{N,i}\}.\label{eq: collection of phi tilde}\\
\cA'&\triangleq(1-\gamma) I_M+\gamma\cA,\label{eq: equation for A'}
\end{eqnarray}  
By  subtracting $w^o_k$ from both sides of~\eqref{eq: stepb} and by adding and subtracting $w^o_k$ to the difference $\bpsi_{k,i}-\bphi_{k,i-1}$, we can write:
\begin{equation}
\label{eq: bphit_k}
\bphit_{k,i}=\bphit_{k,i-1}- \bcQ_k(\bphit_{k,i-1}-\bpsit_{k,i}).
\end{equation}
Let 
\begin{equation}
\label{eq: equation for bchi}
\bchi_{k,i}=\bphit_{k,i-1}-\bpsit_{k,i}.
\end{equation}
 By introducing the quantization error vector:
\begin{equation}
\label{eq: quantization error vector}
 \bz_{k,i}\triangleq\bchi_{k,i}- \bcQ_k(\bchi_{k,i}),
\end{equation}
we can write:
\begin{equation}
\label{eq: bphit_k in terms of the quantization noise}
\bphit_{k,i}=\bpsit_{k,i}+ \bz_{k,i}.
\end{equation}
By {combining}~\eqref{eq: network error vector psi},~\eqref{eq: wt in terms of phi}, and~\eqref{eq: bphit_k in terms of the quantization noise}, we {conclude} that the network  error vector $\bphit_i$ in~\eqref{eq: collection of phi tilde} evolves according to the following dynamics%\footnote{\cmag{Professor: Writing the recursion in terms of $\bphit_i$ turns to be convenient for the analysis. It simplifies some expressions when bounding the quantization noise. If we study $\bphit_i$, it is simple to conclude regarding $\bcwt_i$ since the two variables are related by a linear relation as done in~\eqref{eq: equation relating w tilde to phi tilde}. }}
:
\begin{equation}
\label{eq: network weight error vector}
\boxed{\bphit_i=\bcB_{i-1}\bphit_{i-1}-\mu\bs_i+\mu b+\bz_i}
\end{equation} 
where
\begin{align}
\bcB_{i-1}&\triangleq \left(I_M-\mu\bcH_{i-1}\right)\cA',\label{eq: equation for B}\\
\bz_i&\triangleq\col\left\{ \bz_{k,i}\right\}_{k=1}^N.\label{eq: equation for z}
\end{align}

{In Sec.~\ref{subsec: Mean-square-error stability}, we will first establish the boundedness of $\limsup_{i\rightarrow\infty}\expec\|\bphit_i\|^2$ and then we will use relation~\eqref{eq: wt in terms of phi} to deduce boundedness of $\limsup_{i\rightarrow\infty}\expec\|\bcwt_i\|^2$.} The analysis of recursion~\eqref{eq: network weight error vector} is facilitated by transforming it to a convenient basis using the Jordan canonical decomposition of the matrix $\cA'$ defined in~\eqref{eq: equation for A'}. Now, to exploit the eigen-structure of  $\cA'$, we first recall that a matrix $\cA$ satisfying the conditions in~\eqref{eq: condition A} (for a full-column rank semi-unitary matrix $\cU$) has a Jordan decomposition of the form $\cA=\cV_{\epsilon}\Lambda_{\epsilon}\cV_{\epsilon}^{-1}$ with~\cite[Lemma~2]{nassif2020adaptation}:
\begin{equation}
\label{eq: jordan decomposition of A}
\cV_{\epsilon}=\left[\begin{array}{c|c}
\cU&\cV_{R,\epsilon}
\end{array}\right],~
\Lambda_{\epsilon}=\left[\begin{array}{c|c}
I_P&0\\
\hline
0&\cJ_{\epsilon}
\end{array}\right],~
\cV_{\epsilon}^{-1}=\left[\begin{array}{c}
\cU^\top\\
\hline
\cV_{L,\epsilon}^{\top}
\end{array}\right],
\end{equation} 
where $\cJ_{\epsilon}$ is a Jordan matrix with  eigenvalues (which may be complex but have magnitude less than one) on the diagonal and $\epsilon>0$  on the super-diagonal~\cite[Lemma~2]{nassif2020adaptation},\cite[pp.~510]{sayed2014adaptation}. The parameter $\epsilon$ is chosen small enough to ensure  $\rho(\cJ_{\epsilon})+\epsilon\in(0,1)$~\cite{nassif2020adaptation}. Consequently, the matrix $\cA'$ in~\eqref{eq: equation for A'} has a Jordan decomposition of the form $\cA'=\cV_{\epsilon}\Lambda'_{\epsilon}\cV_{\epsilon}^{-1}$ where:
\begin{equation}
\label{eq: jordan decomposition of A'}
\Lambda'_{\epsilon}=\left[\begin{array}{c|c}
I_P&0\\
\hline
0&\cJ'_{\epsilon}
\end{array}\right], \quad{\text{with }}\cJ'_{\epsilon}\triangleq (1-\gamma) I_{M-P}+\gamma\cJ_{\epsilon}.
\end{equation} 

By multiplying both sides of~\eqref{eq: network weight error vector} from the left by $\cV_{\epsilon}^{-1}$ in~\eqref{eq: jordan decomposition of A},  we obtain the transformed iterates and variables:
\begin{align}
\cV_{\epsilon}^{-1}\bphit_i&=\left[
\begin{array}{c}
\cU^\top\bphit_i\\
\cV_{L,\epsilon}^{\top}\bphit_i
\end{array}
\right]\triangleq\left[
\begin{array}{c}
\bphib_i\\
\bphic_i
\end{array}
\right],\label{eq: transformed variable phi definition}\\
\cV_{\epsilon}^{-1}\bs_i&=\left[
\begin{array}{c}
\cU^\top\bs_i\\
\cV_{L,\epsilon}^{\top}\bs_i
\end{array}
\right]\triangleq\left[
\begin{array}{c}
\bsb_i\\
\bsc_i
\end{array}
\right],\label{eq: transformed variable s definition}\\
\cV_{\epsilon}^{-1}b&=\left[
\begin{array}{c}
\cU^\top b\\
\cV_{L,\epsilon}^{\top}b
\end{array}
\right]\triangleq\left[
\begin{array}{c}
0\\
\widecheck{b}
\end{array}
\right],\label{eq: bias transformed vector}\\
\cV_{\epsilon}^{-1}\bz_i&=\left[
\begin{array}{c}
\cU^\top \bz_i\\
\cV_{L,\epsilon}^{\top}\bz_i
\end{array}
\right]\triangleq\left[
\begin{array}{c}
\bzb_i\\
\bzc_i
\end{array}
\right],
\end{align}
where  in~\eqref{eq: bias transformed vector}  we used the fact that $\cU^\top b=0$ as shown in~\cite[Sec. III-B]{nassif2020adaptation}.  In particular, the transformed components $\bphib_{i}$ and ${\bphic}_{i}$ {evolve according to the recursions}:
\begin{align}
\bphib_{i}&=(I_P-\mu\bcD_{11,i-1})\bphib_{i-1}-\mu\bcD_{12,i-1}\bphic_{i-1}+\bzb_{i}-\mu\bsb_{i}\label{eq: evolution of the centralized recursion}\\
{\bphic}_{i}&=(\cJ'_{\epsilon}-\mu\bcD_{22,i-1})\bphic_{i-1}-\mu\bcD_{21,i-1}\bphib_{i-1}+\bzc_{i}+\mu\widecheck{b}-\mu\bsc_{i}\label{eq: evolution of the centralized recursion 2}
\end{align}
where 
\begin{align}
\bcD_{11,i-1}&\triangleq\cU^{\top}\bcH_{i-1}\cU,\label{eq: definition of D11}\\
 \bcD_{12,i-1}&\triangleq\cU^{\top}\bcH_{i-1}\cV_{R,\epsilon}\cJ_{\epsilon}',\\
\bcD_{21,i-1}&\triangleq \cV_{L,\epsilon}^{\top}\bcH_{i-1}\cU,\\
\bcD_{22,i-1}&\triangleq\cV_{L,\epsilon}^{\top}\bcH_{i-1}\cV_{R,\epsilon}\cJ_{\epsilon}'.\label{eq: definition of D22}
\end{align}

In the following, we shall establish the mean-square-error stability of algorithm~\eqref{eq: decentralized learning approach with quantization}. The analysis will reveal the influence of the step-size $\mu$, the mixing parameter $\gamma$, and the quantization noise  (through  $\{\beta^2_{q,k},\sigma^2_{q,k}\}$) on the network mean-square-error stability and performance, and will provide insights into the design of effective quantizers for decentralized learning under subspace constraints.

%==============================
% Subsec: Mean-square-error stability
%==============================
\subsection{Mean-square-error stability}
\label{subsec: Mean-square-error stability}
\begin{theorem}{{\emph{\textbf{(Mean-square-error stability)}.}}} 
\label{theorem: Network mean-square-error stability}Consider a network of $N$ agents running the quantized decentralized strategy~\eqref{eq: decentralized learning approach with quantization}  under Assumptions~\ref{assump: assumption of the individual costs},~\ref{assump: gradient noise}, and~\ref{assump: quantization noise}, with a matrix $\cA$ satisfying~\eqref{eq: condition A} for an $M\times P$ (with $P\ll M$)  full-column rank semi-unitary matrix~$\cU$. Let $\gamma\in(0,1]$ be such that:%\footnote{\cmag{Professor: the case where $\gamma=0$ is not interesting since it gives us the non-cooperative solution. No subspace constraints in this case.}}:
\begin{equation}
\label{eq: mixing parameter condition theorem}
0<\gamma<\min\left\{1,\frac{1-(\rho(\cJ_{\epsilon})+\epsilon)}{4v_1^2v_2^2\beta^2_{q,\max}(\rho(I-\cJ_{\epsilon})+\epsilon)^2}\right\}.
\end{equation}
{where $v_1=\|\cV_{\epsilon}^{-1}\|$, $v_2=\|\cV_{\epsilon}\|$%$\cV_{\epsilon}$ and $\cJ_{\epsilon}$ are given by~\eqref{eq: jordan decomposition of A}
, and
%\begin{equation}
${\beta}_{q,\max}^2\triangleq\max_{1\leq k\leq N}\{{\beta}^2_{q,k}\}$.
%\end{equation} 
Then,} the network is mean-square-error stable for sufficiently small step-size~$\mu$, namely, it holds that:
\begin{equation}
%&\limsup_{i\rightarrow\infty}\expec\|w^o_k-\bw_{k,i}\|^2=\sigma^2_s O(\mu)+\beta^2_{z,\max}\sigma^2_s O(\mu)+\beta^2_{z,\max}\|\widecheck{b}\|^2 O(\mu)+\|\widecheck{b}\|^2 O(\mu)+\sigma^2_{z} O(\mu^{-1})\label{eq: steady state mean square error},
\limsup_{i\rightarrow\infty}\expec\|w^o_k-\bw_{k,i}\|^2=\delta O(\mu)+\overline{\sigma}^2_{q} O(\mu^{-1})\label{eq: steady state mean square error},
\end{equation}
for $k=1,\ldots,N$ {and where $\delta$ is a constant given by\footnote{{Using relation~\eqref{eq: bound on norm of J' epsilon} in Appendix~\ref{app: Mean-square-error analysis}, and since $\gamma>0$ and $\rho(\cJ_{\epsilon})+\epsilon\in(0,1)$, it can be verified  that $1-\|\cJ'_{\epsilon}\|\geq\gamma(1-\rho(\cJ_{\epsilon})-\epsilon)>0$.} }:
\begin{equation}
\delta\triangleq 2v_1^2\overline{\sigma}^2_s+\frac{3\|\widecheck{b}\|^2}{1-\|\cJ'_{\epsilon}\|}+\beta^2_{q,\max}v_1^2v_2^2(2v_1^2\overline{\sigma}^2_s+12\|\widecheck{b}\|^2),
%{\sigma}^2_z&\triangleq\sum_{k=1}^N{\sigma}^2_{z,k},\\
 %\overline{\sigma}^2_s&\triangleq\sum_{k=1}^N\sigma_{s,k}^2.
\end{equation}
$\overline{\sigma}^2_s=\sum_{k=1}^N\sigma_{s,k}^2$, $\widecheck{b}=O(1)$ is given by~\eqref{eq: bias transformed vector}%$\cJ'_{\epsilon}$ is given by~\eqref{eq: jordan decomposition of A'}
, and  $\overline{\sigma}^2_q=\sum_{k=1}^N{\sigma}^2_{q,k}$. \cblue{Convergence to the steady-state value~\eqref{eq: steady state mean square error} is linear at a rate given by:
\begin{equation}
\begin{split}
    \rho\left( \Gamma \right) \leq \max\Big\{& 1 - \mu \sigma_{11} + O(\mu^2), \\
    & \|\mathcal{J}_{\epsilon}'\| + \kappa\| I - \mathcal{J}_{\epsilon}'\|^2 + O(\mu) \Big\} < 1,
    \end{split}
\end{equation}
for some positive constants $\sigma_{11}$ and $\kappa$.}}
\end{theorem}
\begin{proof} See Appendix~\ref{app: Mean-square-error analysis}.
\end{proof}

While expression~\eqref{eq: steady state mean square error} in Theorem~\ref{theorem: Network mean-square-error stability} reveals the influence of the \emph{step-size} $\mu$, the \emph{quantization noise} (captured by $\{\overline{\sigma}^2_{q},\beta^2_{q,\max}\}$), and the \emph{gradient noise} (captured by $\overline{\sigma}^2_s$) on the steady-state mean-square error, expression~\eqref{eq: mixing parameter condition theorem} reveals the influence of the \emph{relative quantization noise} term (captured by $\beta^2_{q,\max}$) on the network stability.  {One main conclusion stemming \cblue{from expression~\eqref{eq: steady state mean square error} in Theorem 1} is that the mean-square-error contains an $O(\mu)$ term, which is classically encountered in the unquantized case, plus an $O(\mu^{-1})$ term, which can be problematic  in the small step-size regime \cblue{(if the quantity $\overline{\sigma}^2_{q}=\sum_{k=1}^N{\sigma}^2_{q,k}$ is not coupled with $\mu$). 
Since $\overline{\sigma}^2_{q}$} depends on the quantizers' absolute noise components $\{\sigma^2_{q,k}\}$,   this issue can be addressed by ensuring that $\sigma^2_{q,k}\propto \mu^2$. In this way, we would recover  the small estimation error result $\limsup_{i\rightarrow\infty}\expec\|w^o_k-\bw_{k,i}\|^2=O(\mu)$ observed in the unquantized case. }
%While expression~\eqref{eq: steady state mean square error} in Theorem~\ref{theorem: Network mean-square-error stability} reveals the influence of the \emph{step-size} $\mu$, the \emph{quantization noise} (captured by $\{\sigma^2_{q},\beta^2_{q,\max}\}$), and the \emph{gradient noise} (captured by $\sigma^2_{s}$) on the steady-state mean-square error, expression~\eqref{eq: mixing parameter condition theorem} reveals the influence of the \emph{relative quantization noise} term (captured by $\beta^2_{q,\max}$) on the network stability.  One main conclusion stemming from Theorem 1 is that the mean-square-error contains an $O(\mu)$ term, which is classically encountered in the unquantized version of the algorithm, plus an $O(\mu^{-1})$ term, which could be instead problematic in the small step-size regime. 
%Since the term $\sigma^2_{q}$ in~\eqref{eq: steady state mean square error} depends on the quantizers absolute noise component, the small estimation error result $\limsup_{i\rightarrow\infty}\expec\|w^o_k-\bw_{k,i}\|^2=O(\mu)$ can be obtained by % it is tempting to solve this issue by simply 
%choosing $\sigma^2_{q}\propto \mu^2$. In this way, we would recover the traditional  behavior observed in the unquantized case. 

{However, this setup requires a careful inspection since small values of $\sigma^2_{q,k}$  imply small quantization errors, which might in principle require large bit rates. Consequently, in the small step-size regime ($\mu\rightarrow 0$), the bit rate might increase without bound when $\sigma^2_{q,k}\propto \mu^2$. Our goal then becomes to find a quantization scheme that achieves the classical small estimation error result $\limsup_{i\rightarrow\infty}\expec\|w^o_k-\bw_{k,i}\|^2=O(\mu)$ while guaranteeing that the bit rate stays bounded as $\mu\rightarrow 0$. In the next theorem, we will be able to show that the variable-rate scheme illustrated in Sec.~\ref{subsec: Uniform and non-uniform randomized quantizers} achieves both objectives.} \cblue{This theoretical finding will be further illustrated in the simulation section~\ref{subsec: Effect of the small  step-size parameter}.}
\subsection{Bit rate stability}
 Before establishing the bit rate stability, we recall that, at each iteration $i$ and agent $k$, the quantizer input  is given by $\bpsi_{k,i}-\bphi_{k,i-1}$,  which is equal to the vector $\bchi_{k,i}$ in~\eqref{eq: equation for bchi}. Consequently, from~\eqref{eq: bit rate general formula}, the bit rate at agent $k$ and iteration $i$ is given by:
\begin{equation}
\label{eq: bit rate general at agent k}
{{r}_{k,i}}=\log_2(3)\sum_{j=1}^{M_k}\left(1+\expec\left [\left\lceil\log_2(|\boldsymbol{n}([\bchi_{k,i}]_j)|+1)\right\rceil\right ]\right),
\end{equation}
where $[\bchi_{k,i}]_j$ denotes the $j$-th entry of the vector $\bchi_{k,i}$.
\begin{theorem}{{\emph{\textbf{(Bit rate stability)}.}}} 
\label{theorem: rate stability} Assume that each agent $k$ employs the randomized quantizer with the non-linearities given by~\eqref{eq: non-linearities for ANQ}, and with the parameters $\omega_k$ and $\eta_k$ scaling according to:
\begin{equation}
\label{eq: condition of the theorem 2}
\omega_k=c, \qquad \eta_k\propto \mu,
\end{equation}
where $c$ is a constant independent of $\mu$, and where the symbol $\propto$ hides a proportionality constant independent of~$\mu$. First, under conditions~\eqref{eq: condition of the theorem 2}, we have:
\begin{equation}
\label{eq: sigma square}
\sigma^2_{q,k}\propto \mu^2.
\end{equation}
Second, {in steady-state,} the average number of bits at agent $k$ {stays bounded} as $\mu\rightarrow 0$, namely%\footnote{\cmag{Professor: Here we are using differential quantization. Consequently, and according to Fig.~\ref{fig: illustration figure} or equation~\eqref{eq: stepb}, the vectors to be encoded are the difference vectors $\bchi_{k,i}=\bpsi_{k,i}-\bphi_{k,i-1}$. These are the  vectors to be exchanged, and consequently, they require communication resources. We know from the analysis that these vectors are $O(\mu)$ in steady-state. In the simulation section, we represented the bit rates by reporting the \emph{average} number of bits per node (we are averaging over the nodes), per component (we are averaging over the components), required to encode the difference vector. For this reason, for each value of $\mu$ in Fig.~\ref{fig: variable step-size}, we have only one curve reporting the bit rate (we are representing the average number of bits required by a node to encode one component of $\bchi_{k,i}$). }}
, 
\begin{equation}
\label{eq: rate stability result}
\limsup_{i\rightarrow\infty}\,{r}_{k,i}=O(1).
\end{equation}
\end{theorem}
\begin{proof}
See Appendix~\ref{app: rate stability}.
\end{proof}

\cblue{To be concrete, we deemed it useful to focus in Theorem~\ref{theorem: rate stability} on the logarithmic companding rule~\eqref{eq: non-linearities for ANQ}. However, it should be noted that the proof in Appendix~\ref{app: rate stability} can be extended to handle more general non-linear functions.}

%==============================
% Subsec: Theorems insights and observations
%==============================
\section{Theorems~\ref{theorem: Network mean-square-error stability} {and~\ref{theorem: rate stability}:} insights and observations}
\label{sec: Theorems insights and observations}
{The} fundamental conclusion emerging from Theorems~\ref{theorem: Network mean-square-error stability} and~\ref{theorem: rate stability} is that, when compared with~\eqref{eq: decentralized learning approach}, {and when the quantizers are properly designed}, the quantized approach~\eqref{eq: decentralized learning approach with quantization} can guarantee a finite average bit rate whilst still ensuring small estimation errors on the order of $\mu$.   As detailed in the following, there are further  insights that can be gained from the theorems.
 
First, note that in the \emph{absence of quantization}, the analysis allows us to recover the mean-square-error stability result established in~\cite[Theorem~1]{nassif2020adaptation} where it was shown that $\limsup_{i\rightarrow\infty}\expec\|w^o_k-\bw_{k,i}\|^2=O(\mu)$ for all $k$ and for $\gamma=1$.  To see this, {we just set to $0$} the quantization noise parameters $\{\beta_{q,\max}^2,\overline{\sigma}^2_{q}\}$ in Theorem~\ref{theorem: Network mean-square-error stability}.

%Consider, for instance, the case that the dithered uniform quantizer from Table I is used. Then, assuming to set $\sigma^2_{z,k}\propto\mu^2$ reduces to requiring that the quantization step-size $\Delta\propto \mu$. Notably, our analysis reveals that, when $\Delta\propto\mu$, the {\em differential} input $\bm{\chi}_{k,i}=\bm{\psi}_{k,i}-\bm{\phi}_{k,i-1}$ in (4b) is on the order of $\mu$ at the steady-state\footnote{…}. 
%In other words, in this setting, the inputs used in the {\em differential} quantizer scheme have range $O(\mu)$, and, consequently, the quantizer resolution $\Delta\propto\mu$ scales proportionally to the {\em effective} range of the quantizer input. 
%

{Second, from our analysis, it is possible to explain the  bit rate stability  result. Consider, for simplicity, the dithered uniform quantizer from Table~\ref{table: examples of quantizers}. Then, setting $\sigma^2_{q,k}\propto\mu^2$ is equivalent to requiring that the quantization step-size $\Delta$ is proportional to $\mu$. Our analysis reveals that, when $\Delta\propto\mu$, the {\em differential} input $\bchi_{k,i}=\bpsi_{k,i}-\bphi_{k,i-1}$ in~\eqref{eq: stepb}  is on the order of $\mu$ at the steady-state\footnote{This can be seen by taking the limits on both sides of relation~\eqref{eq: centralized error vector recursion inequality on sum of  delta 1} and by using~\eqref{eq: single inequality recursion steady state 3 new}.}. In this setting, i) the \emph{effective} range of the  inputs used in the {\em differential} quantizer scheme vanishes as $O(\mu)$ when $\mu\rightarrow 0$, and ii) the quantizer resolution $\Delta\propto\mu$ scales proportionally to the {\em effective} range of the quantizer input. Notably, Theorem~\ref{theorem: rate stability} reveals that the variable-rate scheme is able to adapt the number of bits to this effective range, in such a way that the expected bit rate stays bounded as $\mu\rightarrow 0$.}

 Third, we conclude from Theorem 1 that, while a value $\gamma=1$ could lead to network instability when $\beta^2_{q,\max}\neq 0$, this conclusion does not hold when $\beta^2_{q,\max}= 0$. In other words, the \emph{absolute quantization noise component does not affect the network stability}, and the network remains stable\footnote{In fact, from~\eqref{eq: general form for gamma}--\eqref{eq: constant d definition new} and~\eqref{eq: condition ensuring stability of gamma new},  when $\beta^2_{q,\max}= 0$ and $\gamma=1$, one can observe that a sufficiently small step-size $\mu$ can ensure the stability of the matrix $\Gamma$.} for sufficiently small step-size $\mu$ when $\gamma=1$.

Fourth, by noting that step~\eqref{eq: stepb} can be written alternatively~as:
\begin{equation}
\bphi_{k,i}=\bpsi_{k,i}-\bz_{k,i},
\end{equation}
with $\bz_{k,i}$ given by~\eqref{eq: quantization error vector}, it can be observed that the impact of quantization in~\eqref{eq: stepc} is similar to the impact of exchanging information over the communication links in the presence of  additive noise processes since step~\eqref{eq: stepc} can be written alternatively as (for $\gamma=1$):
\begin{equation}
\bw_{k,i}=\sum_{\ell\in\cN_k}A_{k\ell}(\bpsi_{\ell,i}-\bz_{\ell,i}).\label{eq: stepc additive noise}
\end{equation}
Therefore, in the absence of the relative quantization noise term (i.e., when $\beta^2_{q,\max}=0$), %, and when the additive noises $\{\bv_{\ell,i}\}$ can be modeled as temporally white and spatially independent random variables with zero mean and covariances   $\{R_{v,\ell}\}$, 
we obtain a setting similar to the one previously considered in the context of  single-task~\cite{zhao2012diffusion} and multititask~\cite{nassif2016diffusion} estimation over mean-square-error networks   in the presence of noisy links.  However, notice that the analysis in~\cite{zhao2012diffusion} and~\cite{nassif2016diffusion} is limited to LMS diffusion algorithms, does not exploit the randomized quantizers design, and  does not investigate the average bit rate stability.
 
\cblue{Fifth, observe} that in the \emph{absence of absolute quantization} noise term, i.e., when $\overline{\sigma}^2_{q}=0$, we obtain:
\begin{equation}
\limsup_{i\rightarrow\infty}\expec\|w^o_k-\bw_{k,i}\|^2=O(\mu).\label{eq: steady state mean square error in the absence of the absolute noise component}
\end{equation}
Thus, the relative quantization noise component does not affect the $O(\mu)$ small  estimation error result, but affects the network stability through condition~\eqref{eq: mixing parameter condition theorem}. {Regarding the bit rate, as it can be observed from Table~\ref{table: examples of quantizers} (column 5, rows 4--7), the reported quantization schemes are characterized by a fixed bit-budget that is independent of the quantizer~input, but depends on the high-precision quantization parameter $B_{\text{HP}}$.} \cblue{The theoretical result~\eqref{eq: steady state mean square error in the absence of the absolute noise component} will be  illustrated in Sec.~\ref{subsec: Effect of the small  step-size parameter} by considering the QSGD quantizer of Table~\ref{table: examples of quantizers} (row 7).}

\cblue{Finally, and before discussing the stability condition~\eqref{eq: mixing parameter condition theorem}, we note that the parameter $\gamma$ allows to control $\rho(\cA'-\cP_{\ccu})$ (i.e., the spectral radius of the matrix $\cA'-\cP_{\ccu}$), and consequently, the speed of convergence of $(\cA')^{i}$ to $\cP_{\ccu}$~\cite[Lemma~2]{nassif2020adaptation}.  To see this, note that (by using the Jordan decomposition of the matrix~$\cA'$):
\begin{equation}
\cA'-\cP_{\ccu}=\cV_{\epsilon}\left[\begin{array}{c|c}
0&0\\
\hline
0&\cJ'_{\epsilon}
\end{array}\right]\cV_{\epsilon}^{-1},
\end{equation}
from which we obtain  $\rho(\cA'-\cP_{\ccu})=\rho(\cJ'_{\epsilon})\overset{\eqref{eq: equation to bound rho J prime}}\leq  (1-\gamma)+\gamma\rho(\cJ_{\epsilon})$ where $\rho(\cJ_{\epsilon})\in(0,1)$ and $\gamma\in(0,1]$. Thus, the larger the \emph{mixing} parameter $\gamma$ is, the smaller  $\rho(\cA'-\cP_{\ccu})$ tends to be, and consequently, the faster the convergence of $(\cA')^{i}$ to $\cP_{\ccu}$ will be.  Now, returning to the stability condition~\eqref{eq: mixing parameter condition theorem}, besides requiring $\gamma\in(0,1]$, the {mixing} parameter $\gamma$ must be chosen smaller than a value that is inversely proportional to $\beta^2_{q,\max}$. Thus, the larger $\beta^2_{q,\max}$ is, the tighter the upper bound in~\eqref{eq: mixing parameter condition theorem} is, and the smaller $\gamma$ should be. If we consider for instance the QSGD quantizer from Table~\ref{table: examples of quantizers}, it can be observed that the smaller the number of levels $s$ is, i.e., the smaller the number of used bits is,  the larger $\beta^2_{q,k}$ is. Consequently, the larger $\beta^2_{q,k}$ is, the farther the agents (seeking to converge to $\text{Range}(\cU)$) can get from the subspace due to the quantization noise, and thus, the mixing parameter $\gamma$ must be chosen small enough to ensure that the  quantization noise  perturbations will not affect the  network stability. }

\section{Simulation results}
\label{sec: simulation results}
\begin{figure*}
\begin{center}
\includegraphics[scale=0.28]{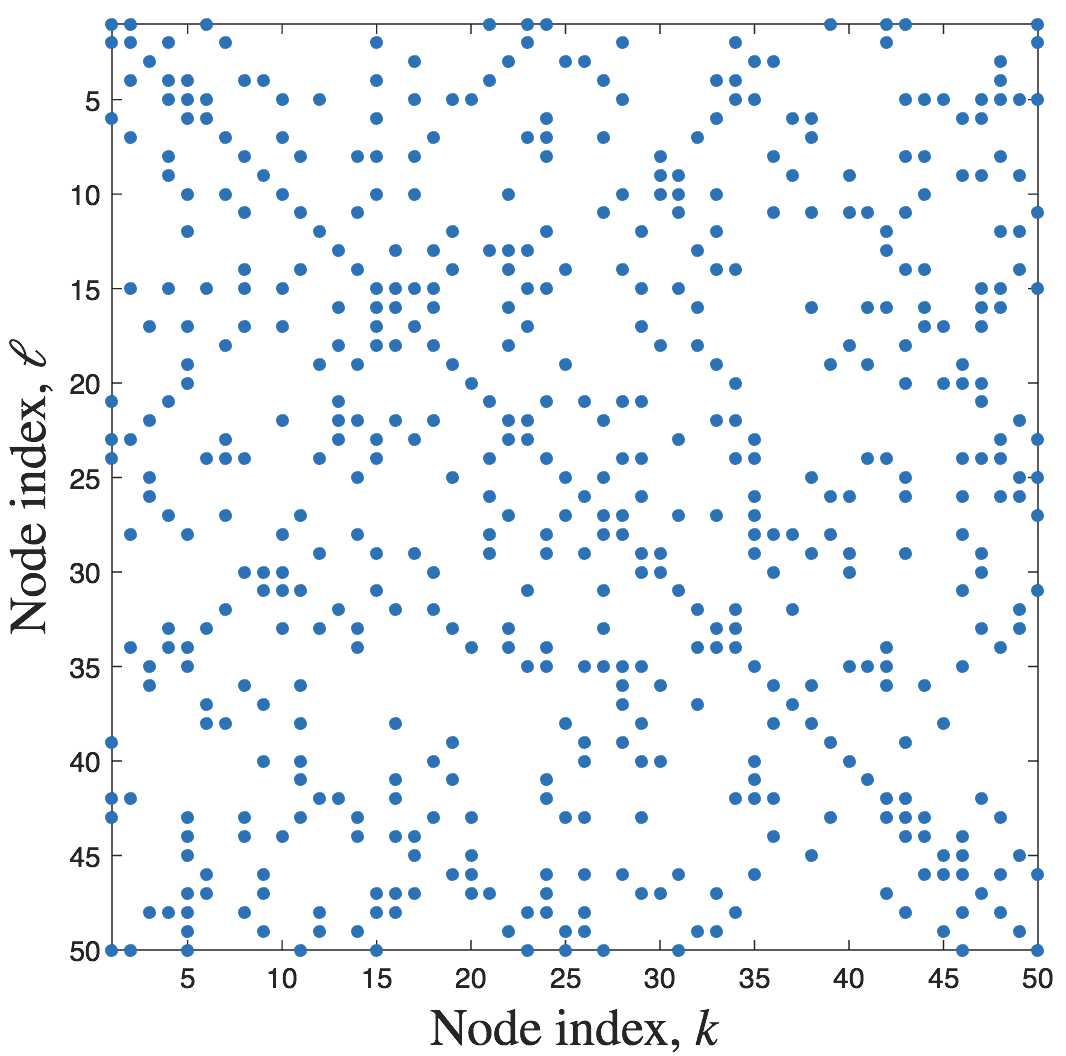}\qquad
\includegraphics[scale=0.37]{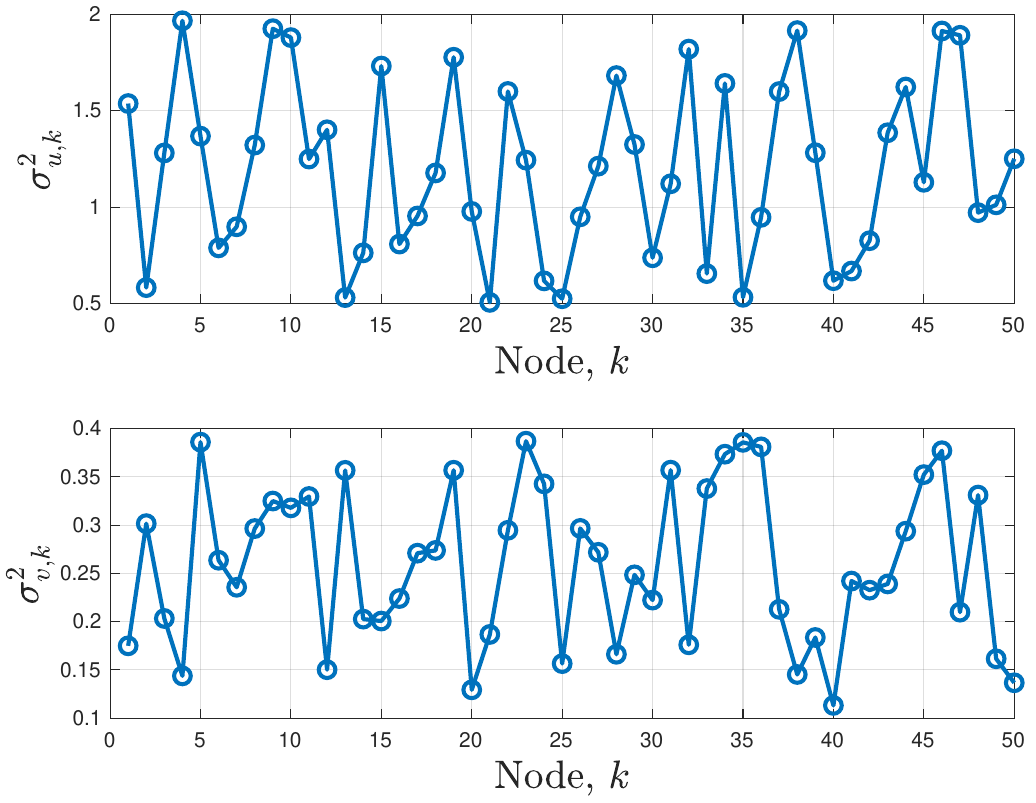}
\caption{Experimental setup. \textit{(Left)} Link matrix ($N=50$ nodes). \textit{(Right)} Regressor and noise variances.
}
\label{fig: data settings}
\end{center}
\end{figure*}

We apply strategy~\eqref{eq: decentralized learning approach with quantization} to a  network of $N=50$ nodes with the link matrix shown in Fig.~\ref{fig: data settings} {\emph{(left)}}.  Each agent is subjected to streaming data $\{\bd_k(i),\bu_{k,i}\}$ assumed to satisfy a linear regression model of the form~\cite{sayed2014adaptation}:
\begin{equation}
\label{eq: linear data regression model}
\bd_k(i)=\bu_{k,i}^\top w^\star_k+\bv_k(i),
\end{equation}
for some unknown $L\times 1$ vector $w^\star_k$ with $\bv_k(i)$ denoting a zero-mean measurement noise and $L=5$. A mean-square-error cost of the form $J_k(w_k)=\frac{1}{2}\expec|\bd_k(i)-\bu_{k,i}^\top w_k|^2$ is associated with each agent $k$. The regressor and noise processes $\{\bu_{k,i},\bv_k(i)\}$ are assumed to be zero-mean Gaussian with: i) $\expec\bu_{k,i}\bu_{\ell,i}^\top=R_{u,k}=\sigma^2_{u,k}I_L$ if $k=\ell$ and zero otherwise; ii) $\expec\bv_{k}(i)\bv_{\ell}(i)=\sigma^2_{v,k}$ if $k=\ell$ and zero otherwise; and iii) $\bu_{k,i}$ and $\bv_{k}(i)$ are independent of each other. The variances $\sigma^2_{u,k}$ and $\sigma^2_{v,k}$ are illustrated in Fig.~\ref{fig: data settings} {\textit{(right)}}. The signal $\cw^\star=\col\{w^\star_1,\ldots,w^\star_N\}$ is generated  by smoothing a signal $\cw_o$, which is randomly generated from the Gaussian distribution  $\cN(0.4\times\mathds{1}_{NL},I_{NL})$,  by a graph diffusion kernel with $\tau=3$ -- see~\cite[Sec. IV]{nassif2020adaptation2}  for more details  on the smoothing process. The matrix $\cU$ is generated according to $\cU=U\otimes I_L$   where $U=[u_1~u_2]$, and $u_1$ and $u_2$ are  the first two eigenvectors of the graph Laplacian  $\mathscr{L}$. The Laplacian matrix is generated according to $\mathscr{L}=\diag\{C\mathds{1}_N\}-C$ where $C$ is the $N\times N$  weighted adjacency matrix chosen such that the $(k,\ell)-$th entry $[C]_{k\ell}=0.1$ if $\ell\in\cN_k$ and $0$ otherwise. The combination matrix $\cA$ satisfying the conditions in~\eqref{eq: condition A} and having the same structure as the graph   is found by following the same approach as in~\cite{nassif2019distributed}.

%===============================
% Subsec: Rate-distortion curves
%===============================
\subsection{Effect of the small  step-size parameter}
\label{subsec: Effect of the small  step-size parameter}
\begin{figure*}
\begin{center}
\includegraphics[scale=0.32]{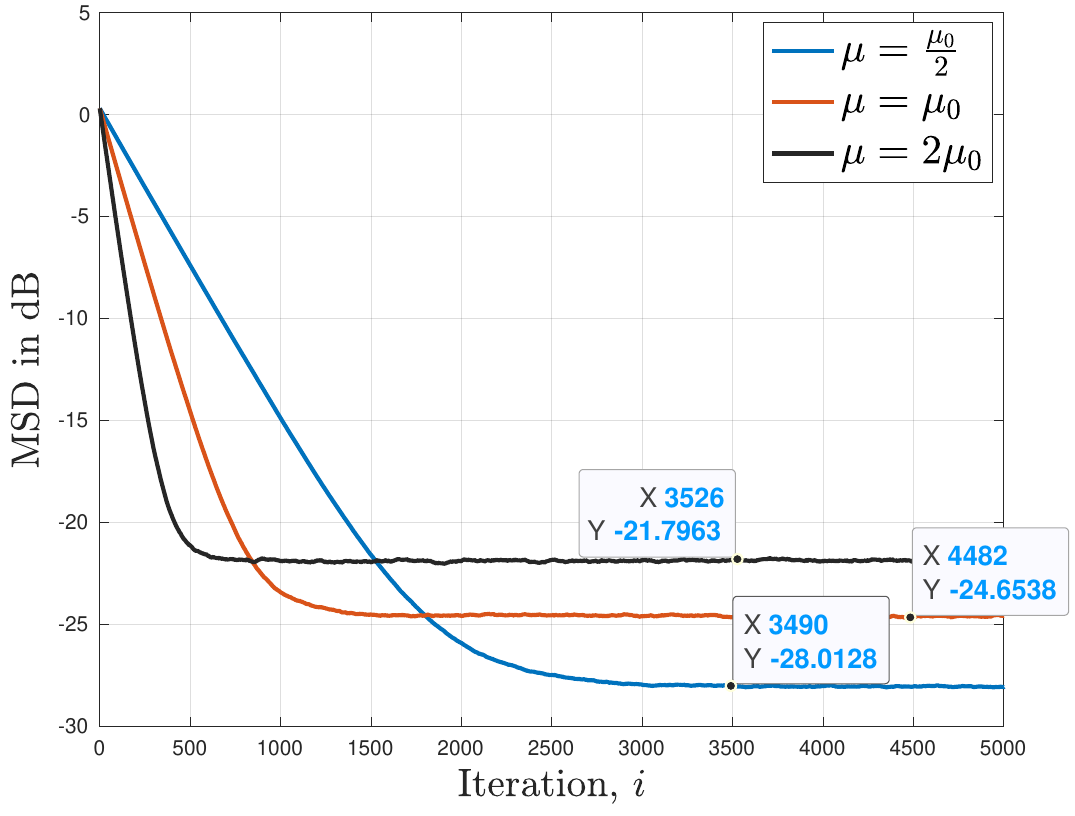}\quad
\includegraphics[scale=0.32]{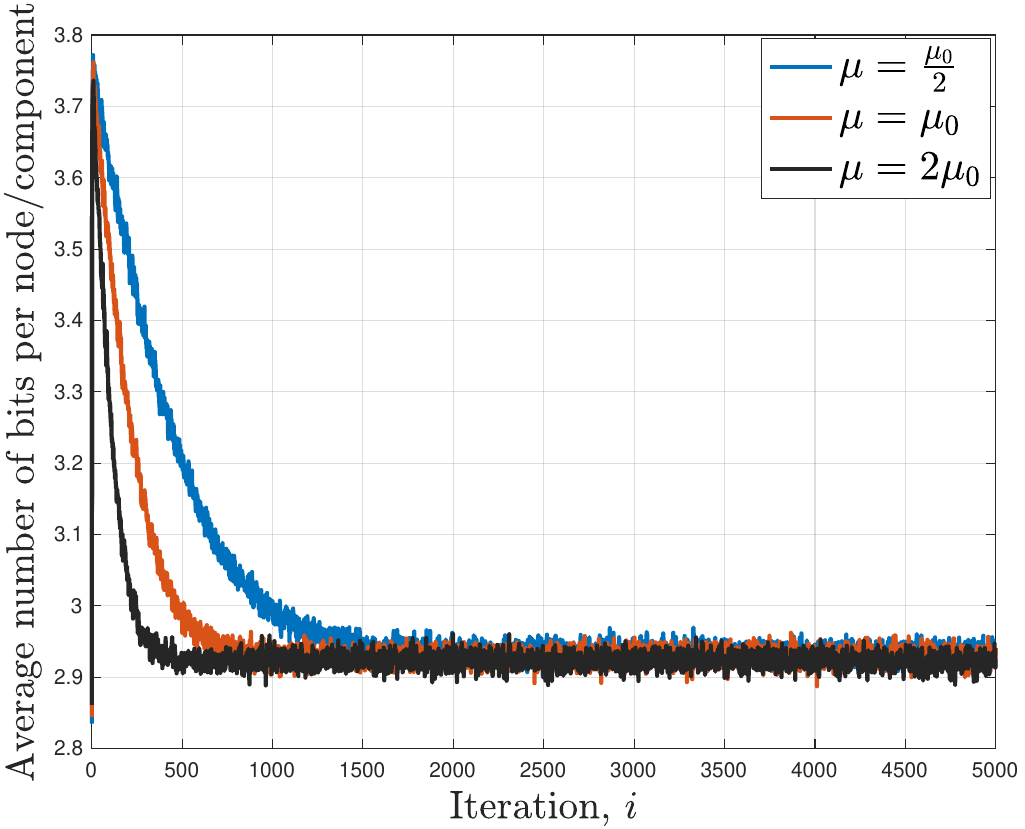}
\includegraphics[scale=0.32]{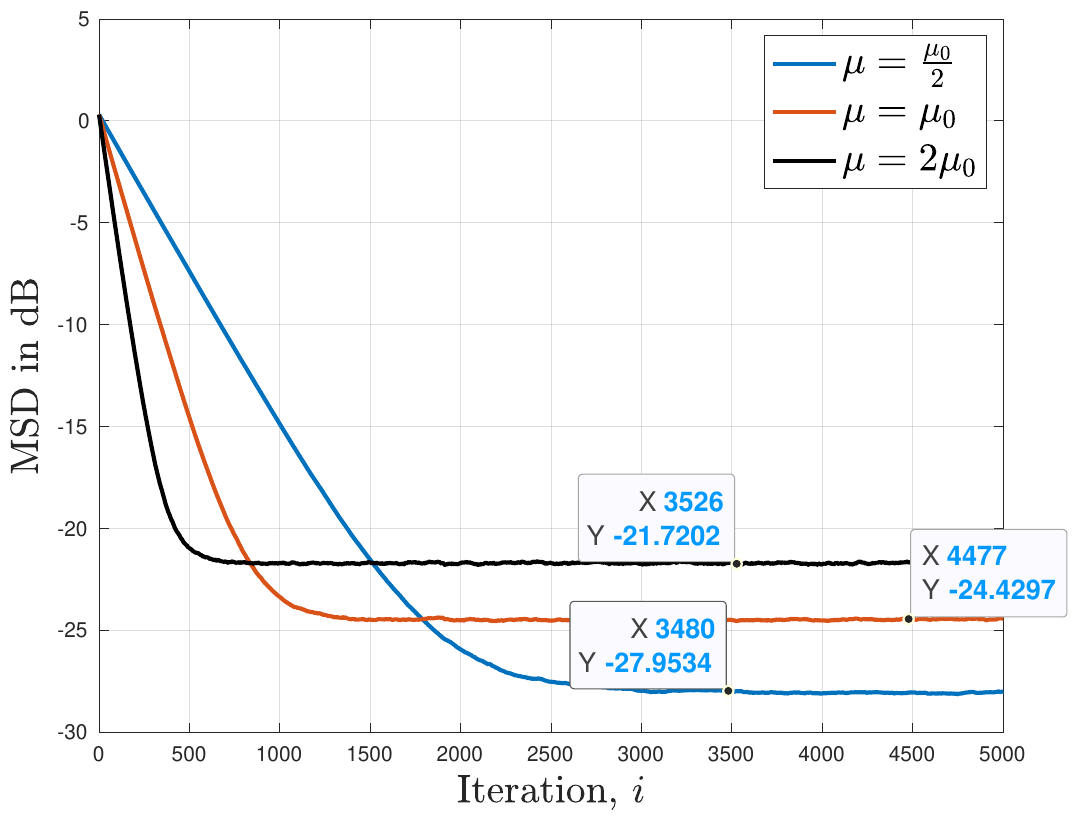}
\caption{Network performance w.r.t. $\cw^o$ in~\eqref{eq: network constrained problem} for three different values of the step-size ($\mu_0=0.003$). {In the {left} and {middle} plots, the probabilistic ANQ quantizer of Example 2 is employed.} \textit{(Left)}  {Evolution of the MSD learning curves. \textit{(Middle)}} Evolution of the average number of bits per node, per component, when  the variable-rate scheme described in {Sec.~\ref{subsec: Uniform and non-uniform randomized quantizers}} is used {to encode the difference $\bchi_{k,i}=\bpsi_{k,i}-\bphi_{k,i-1}$ in~\eqref{eq: stepb}}. {\textit{(Right)}  Evolution of the MSD learning curves when the QSGD quantizer (with $s=2$) of Table~\ref{table: examples of quantizers} is used.}
}
\label{fig: variable step-size}
\end{center}
\end{figure*}
 In Fig.~\ref{fig: variable step-size} {\emph{(left)}}, we report the network mean-square-deviation (MSD) learning {curves: 
\begin{equation}
 \label{eq: MSD learning curve}
 \text{MSD}(i)=\frac{1}{N}\sum_{k=1}^N\expec\|w^o_k-\bw_{k,i}\|^2
 \end{equation}
 for~3} different values of the step-size $\mu$. The results were averaged over 100 Monte-Carlo runs. We used the probabilistic ANQ quantizer of Example 2 with $\omega_k=0.25$ and $\eta_k=\frac{\mu}{\sqrt{2L}}$. This choice ensures that the quantizer settings of Theorem~\ref{theorem: rate stability} are satisfied. % {{\em absolute noise} parameter $\sigma_{z,k}^2=\mu^2$.  We set the  {{\em absolute noise} parameter $\sigma_{z,k}^2=\mu^2$} $\forall k$ by setting 
%\begin{equation}
%\label{eq: formula for delta}
%\Delta=\log_{1+\nu}\left(1+\sqrt{\frac{2\nu^2\mu^2}{LV^2}}\right).
%\end{equation} 
%Using~\eqref{eq: formula for delta} into~\eqref{eq: setting beta and sigma compandor}, we obtain $\beta^2_{z,k}=\frac{\nu^2\mu^2}{LV^2}$. 
We set $\gamma=0.88$. We observe that, in steady-state, the network MSD increases by approximately $3$dB when $\mu$ goes from $\mu_0$ to $2\mu_0$. This means that the performance is on the order of $\mu$, as expected from the discussion in Sec.~\ref{sec: Theorems insights and observations}
since in the simulations  the  absolute noise component is  such that  $\sigma^2_{q,k}=\mu^{2}$. In Fig.~\ref{fig: variable step-size}  {\emph{(middle)}}, we report the average number of bits per node, per component,  computed {according to:
\begin{equation}
\label{eq: average number of bits per node}
R(i)=\frac{1}{N}\sum_{k=1}^N\frac{1}{L}r_{k,i},
\end{equation} 
where} $r_{k,i}$ is {the bit rate given by~\eqref{eq: bit rate general at agent k}, which is associated with the encoding of the difference vector $\bchi_{k,i}=\bpsi_{k,i}-\bphi_{k,i-1}$  transmitted by agent $k$ at iteration $i$ according to~\eqref{eq: stepb}}.
%\begin{equation}
%R(i)=\frac{1}{N}\sum_{k=1}^N\frac{1}{L}\expec \left[\boldsymbol{r}(\bw_{k,i})\right]
%\end{equation}
%where $r(\bw_{k,i})$ is the total number of required bits to encode the $L\times 1$ vector $\bw_{k,i}$. %Throughout the simulation section, the adaptive  encoding scheme proposed in~\cite{lee2021finite} was adopted. 
As it can be observed, and thanks to the variable-rate quantization, a finite average number of bits is guaranteed (approximately 2.9 bits/component/iteration are required {on average} in steady-state).  Moreover, it can be observed that the average number of bits scales coherently with the distortion scaling law induced by $\mu$. That is, a smaller MSD in steady-state would require using a larger number of bits, and vice versa.

{For comparison purposes, we report in Fig.~\ref{fig: variable step-size} \emph{(right)} the MSD learning curves when the QSGD quantizer of Table~\ref{table: examples of quantizers} (row 7) is employed instead of the probabilistic ANQ. Apart from the quantizer scheme, the same settings as above were assumed. For the QSGD scheme, we set the number of quantization levels $s$ to $2$. As it can be observed, this choice allows us to compare the average number of bits for the ANQ and QSGD quantizers for similar values of steady-state MSD. From Table~\ref{table: examples of quantizers} (row 7, column 5), the bit-budget required to encode a $5\times 1$ vector using the QSGD scheme ($s=2$) is given by $B_{\text{HP}}+10$. Now, by replacing $B_{\text{HP}}$ by 32 (since we are performing the experiments on MATLAB 2022a which uses 32 bits to represent a floating number in single-precision), we find that the QSGD quantizer requires, at each iteration $i$, an average number of bits per node, per component, equal to $\frac{42}{5}=8.4$, which is almost three times higher than the one obtained in steady-state when the probabilistic ANQ is used. This is expected since the QSGD scheme requires encoding the norm of the input vector with very high precision.} 

\cblue{Note that, since the results of Theorems~\ref{theorem: Network mean-square-error stability} and~\ref{theorem: rate stability}  hold for any $M\times P$ full-column rank semi-unitary matrix $\cU$, similar observations will hold true when applying the quantized approach~\eqref{eq: decentralized learning approach with quantization} to solve general constrained optimization problems of the form~\eqref{eq: network constrained problem}. In the supplementary material, we illustrate this fact by considering a simulation similar to the one considered in the current section~\ref{subsec: Effect of the small  step-size parameter}, but instead we consider solving consensus optimization~\eqref{eq: consensus optimization} by choosing $\cU=\frac{1}{\sqrt{N}}(\mathds{1}_N\otimes I_L)$.} % illustrate this, we consider  applying approach~\eqref{eq: decentralized learning approach with quantization} to solve consensus optimization problems of the form~\eqref{eq: consensus optimization}. However, due to space limitations, we include the curves resulting from repeating the above experiments  in the supplementary material.}

%===============================
% Subsec: Rate-distortion curves
%===============================
\subsection{Rate-distortion curves}
\begin{figure}
\begin{center}
\includegraphics[scale=0.2]{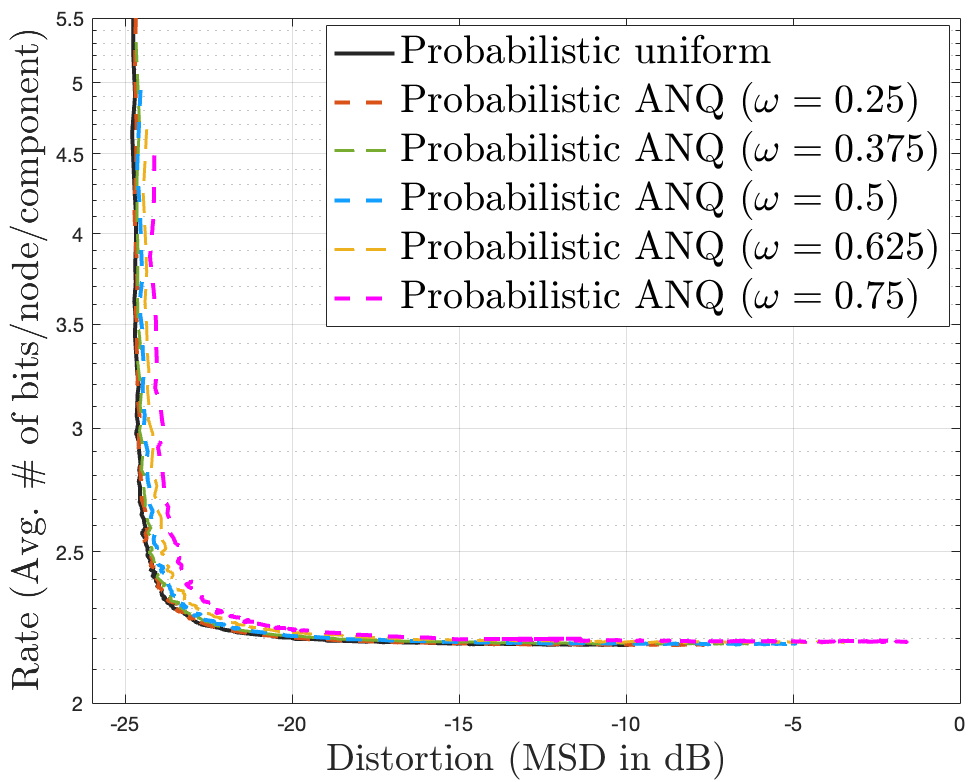}
\caption{\emph{(Left)} Rate-distortion curves for the probabilistic uniform and ANQ.
}
\label{fig: rate-distortion curves}
\end{center}
\end{figure}
In order to examine the performance of the proposed learning strategy, it is necessary to consider both the \emph{attained learning error} (MSD) and the associated \emph{bit expense}. This gives rise to a \emph{rate-distortion} curve, where the rate quantifies the bit budget and the MSD quantifies the distortion. In Fig.~\ref{fig: rate-distortion curves}, we illustrate the \emph{rate-distortion} curves for probabilistic uniform  and ANQ (for different values of the parameter $\omega$). 
We set $\mu=0.003$ and $\gamma=0.88$.  For the probabilistic uniform quantizer, each point {of} the rate/distortion {curve} corresponds to {one} value of the quantization step $\Delta$. In the example, {we selected} $200$ values of $\Delta$, uniformly sampled in the interval $[0,0.1]$. % The ANQ rate-distortion curves were obtained by setting the non-linearity parameter $\omega$ to different values and by setting  the parameter $\eta=\frac{\Delta}{2}$ with  $\Delta$ uniformly sampled over $[0,0.1]$. 
For each value of $\Delta$ {(i.e., each point of the curve)}, the resulting MSD (distortion) and average number of bits/node/component (rate)  were {obtained by averaging the instantaneous mean-square-deviation MSD$(i)$ in~\eqref{eq: MSD learning curve} and averaging the  number of bits $R(i)$ in~\eqref{eq: average number of bits per node} over 500 samples after convergence of the algorithm (the expectations in~\eqref{eq: MSD learning curve} and~\eqref{eq: bit rate general at agent k} are estimated empirically over $50$ Monte Carlo runs).} %{The values of MSD$(i)$ and $R(i)$ were generated by considering 50 Monte-Carlo runs}.
 For the probabilistic ANQ scheme, we draw three curves corresponding to different values of the parameter $\omega$. For each curve, the different rate/distortion pairs correspond to $200$ values of the parameter $\eta$, uniformly sampled in the interval $[0,0.5]$. The trade-off between rate and distortion can be observed from Fig.~\ref{fig: rate-distortion curves}, namely, as the rate decreases, the distortion increases, and vice versa. For the ANQ, we further observe that the curves move away from the uniform case as $\omega$ increases, indicating a superior performance of the uniform probabilistic quantizer. In other words, the simulations show that in the considered example no advantage is obtained by employing non-uniform quantization. %For the probabilistic compandor, we set $V=10$ and $\nu=150$, and $\Delta^{\text{comp}}=\Delta$ with $\Delta$  uniformly sampled  over $[0,0.1]$. For the dithered-quantizer, we uniformly sampled $\Delta$ over $[0,0.1]$ and we set $\Delta^{\text{unif}}=\sqrt{2}\frac{V}{\nu}c_{\nu,\Delta}$. This setting allows us to compare the probabilistic uniform quantizer and the compandor for the same values of the  \cblue{absolute} noise term $\sigma_{z,k}^2$--see Table~\ref{table: examples of quantizers} and~\eqref{eq: setting beta and sigma compandor}. For the probabilistic ANQ, we set $\eta=\frac{V}{2\nu}c_{\nu,\Delta}$ and $\omega=\frac{c_{\nu,\Delta}}{2}$ where  $V=10$, $\nu=150$, and $\Delta$ uniformly sampled over $[0,0.1]$. This setting allows us to compare the probabilistic ANQ and the compandor for the same values of $\{\beta^2_{z,k},\sigma_{z,k}^2\}$--see Table~\ref{table: examples of quantizers} and~\eqref{eq: setting beta and sigma compandor}. For each value of $\Delta$, the resulting MSD (distortion) and average number of bits/node/component (rate)  were averaged over $50$ Monte-Carlo runs and over 500 samples after convergence of the algorithm. The trade-off between rate and distortion can be observed from Fig.~\ref{fig: rate-distortion curves} (left), namely, as the rate decreases, the distortion increases, and vice versa. The ANQ rate-distortion curves in Fig.~\ref{fig: rate-distortion curves} (middle) were obtained by setting the non-linearity parameter $\omega$ to different values and by uniformly sampling the parameter $\eta$ in the interval $[10^{-3}, 10^{-1}]$, while the compandor rate-distortion curves in Fig.~\ref{fig: rate-distortion curves} (left)  were obtained by choosing different values of the design parameters $\{V,\nu\}$ and by uniformly sampling the  parameter $\Delta$ in the interval $[10^{-3}, 10^{-1}]$. For the ANQ, we observe that the curves move away from the uniform case as $\omega$ increases, suggesting a superior performance of the uniform probabilistic quantizer. The curves corresponding to the uniform quantizer, highly non-linear compandor ($\nu=150$), and ``almost'' linear compandor ($\nu=1$) are almost superimposed indicating that no important advantage can be obtained by employing non-linear quantization in the adaptive encoding context considered in the simulation section. % In Figs.~\ref{fig: rate-distortion curves} (middle) and (right), we illustrate the steady-state average number of bits/node/component (rate) and MSD (distortion), respectively, when the parameter $\Delta$ varies over $[10^{-4},0.1]$. As it can be seen, for the three quantizers, the larger $\Delta$ is, the lower the resolution is (i.e., the lower the rate is), and the larger the distortion is. 
{One useful interpretation for this behavior is as follows. In the theory of quantization, non-uniform quantizers are typically employed in a fixed-rate context, where the number of bits is determined in advance and does not depend on the input value. Under this setting, allocating high-resolution quantization intervals where the random variable is more likely to be observed provides some advantage in terms of distortion, given the available \emph{fixed} rate. However, in our case, we are considering a variable-rate quantizer that changes the number of bits depending on the input value.  {For example, when the quantizer input $x$ belongs to a narrower quantization interval (i.e., smaller distortion), the variable-rate scheme allocates a higher number of bits. This  somehow nullifies  the distortion gain  since we are allocating more bits.} Therefore, allocating non-uniform intervals and then adapting the \emph{variable} rate does not seem rewarding when compared to a uniform quantizer. % not rewarding, and a uniform quantizer performs better in the considered examples.  
%In this case, allocating non-uniform intervals and then adapting the (variable) rate does not seem rewarding when compared to a uniform quantizer. % not rewarding, and a uniform quantizer performs better in the considered examples.
{This bears some similarity to what happens in the theory of quantization when one employs a fixed-rate quantizer followed by an entropy encoder. In this case, the possible advantages of non-uniform quantizers are nullified by the entropy encoder, and it is well known that (in the high resolution regime) the uniform quantizer is the best choice~\cite{gish1968asymptotically}.}} 

%================================
% Subsec: tracking ability
%================================
\subsection{\cblue{Tracking ability}}
\label{subsec: tracking ability}
\begin{figure*}
\begin{center}
\includegraphics[scale=0.33]{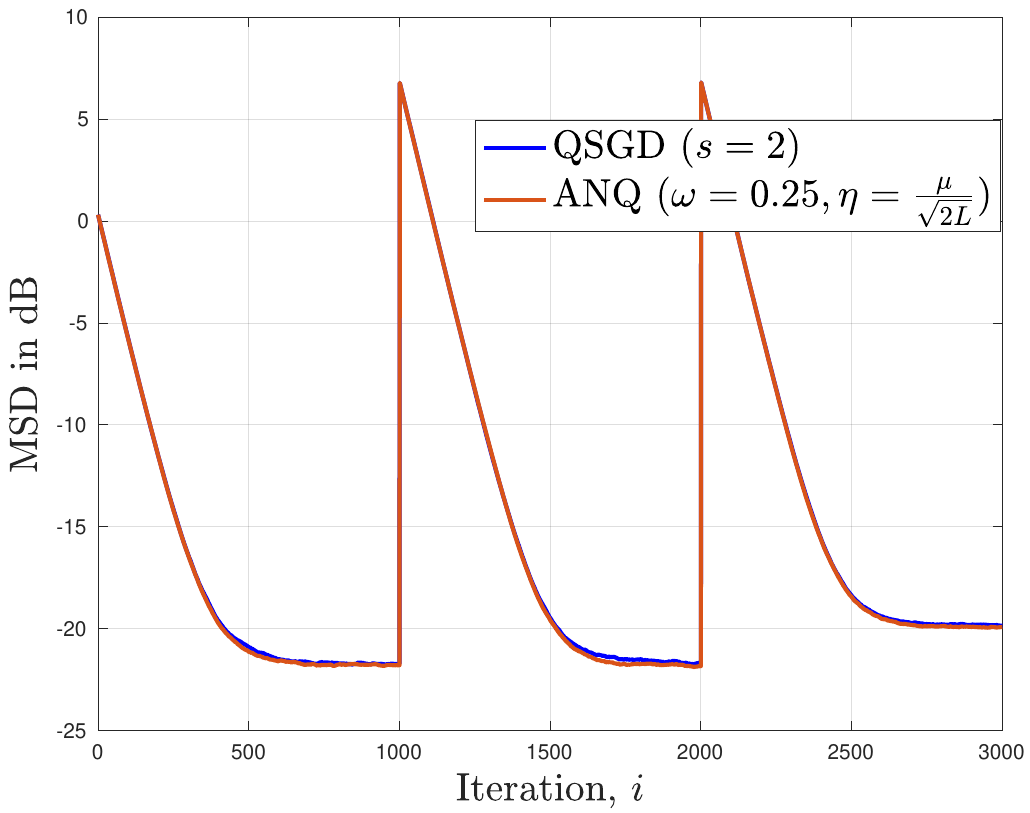}\qquad
\includegraphics[scale=0.33]{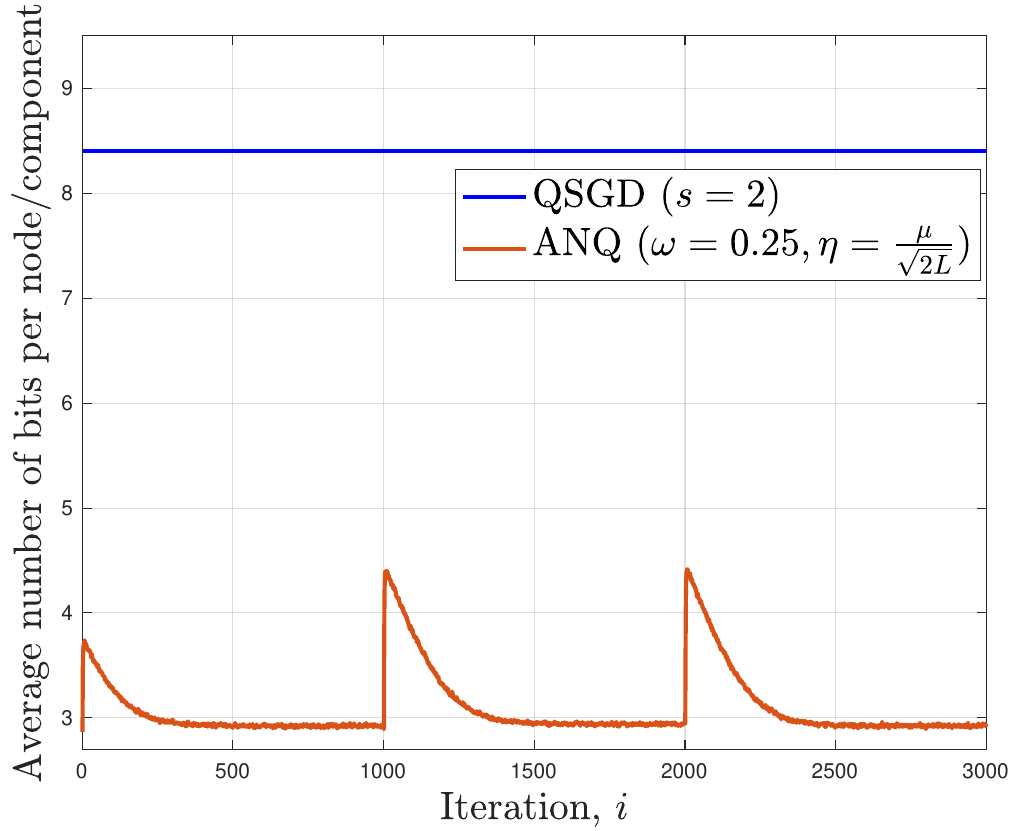}
\caption{\cblue{Network performance in non-stationary environment (i.e., the solution $\cw^o$ in~\eqref{eq: network constrained problem} is changing over time due to changes in the individual costs $\{J_k(w_k)\}$) when the probabilistic ANQ quantizer of Example 2 and  the QSGD quantizer of Table~\ref{table: examples of quantizers} are used. \textit{(Left)}  {Evolution of the MSD learning curves. \textit{(Right)}} Evolution of the average number of bits per node, per component. The variable-rate scheme  described in {Sec.~\ref{subsec: Uniform and non-uniform randomized quantizers}}  is used in the probabilistic ANQ  case.}
}
\label{fig: tracking ability}
\end{center}
\end{figure*}

\cblue{To illustrate the tracking ability of the quantized approach~\eqref{eq: decentralized learning approach with quantization}, we  modify $w^\star_k$ in~\eqref{eq: linear data regression model} at time instants $1000$ and $2000$ by modifying $\cw_o$ such that $\cw_o$ is randomly generated from $\cN(1.4\times\mathds{1}_{NL},I_{NL})$ at $i=1000$ and from $\cN(2.4\times\mathds{1}_{NL},I_{NL})$ at $i=2000$. Apart from varying $\cw_o$ (and, consequently, $\cw^o$ in~\eqref{eq: network constrained problem}) and from fixing the  step-size  $\mu$ to $0.006$, the same settings as in Sec.~\ref{subsec: Effect of the small  step-size parameter} were assumed. The resulting MSD learning curves and average number of bits/node/component curves are reported in Fig.~\ref{fig: tracking ability} \emph{(left)} and  \emph{(right)}, respectively, for both the probabilistic ANQ and the QSGD quantizers. As it can be observed,  approach~\eqref{eq: decentralized learning approach with quantization} is able to track  changes in the solution of the constrained optimization problem~\eqref{eq: network constrained problem} despite quantization. }
\section{Conclusion}
In this work, we considered inference problems over networks where agents have individual parameter vectors to estimate subject to subspace constraints that require the parameters across the network to lie in low-dimensional subspaces. The constrained subspace problem includes standard consensus optimization  as a special case, and allows for more general task relatedness models such as multitask smoothness. To alleviate the communication bottleneck resulting from the exchange of the intermediate estimates among agents over many iterations, we proposed a new decentralized strategy relying on \emph{differential randomized quantizers}. We studied its mean-square-error stability in the presence of both  \emph{relative} and \emph{absolute} quantization noise terms.  The analysis framework is general enough to cover many typical examples of probabilistic quantizers such as QSGD quantizer, gradient sparsifier, uniform or dithered quantizer, and non-uniform compandor. We showed that, for small step-sizes, and under some conditions on the quantizers (captured by the terms $\beta^2_{q,k},\sigma^2_{q,k}$)\cblue{, on the network topology (captured by the eigendecomposition of the combination matrix $\cA$), and on the mixing parameter~$\gamma$}, the decentralized quantized approach is able to converge in the mean-square-error sense within $O(\mu)$ from the solution of the constrained optimization problem, despite the gradient and quantization~noises.

%=====================================
% Sec: Appendices
%=====================================
\begin{appendices}
\section{Mean-square-error analysis}
\label{app: Mean-square-error analysis}
We consider the transformed iterates $\bphib_{i}$ and $\bphic_{i}$ in~\eqref{eq: evolution of the centralized recursion} and~\eqref{eq: evolution of the centralized recursion 2}, respectively. % Conditioning {both sides of~\eqref{eq: evolution of the centralized recursion} on $\bcF_{i-1}$}, computing the conditional second-order moments, using Assumptions~\ref{assump: quantization noise}  and~\ref{assump: gradient noise} on the quantization and gradient noise processes, and computing the expectations again, we get:
Computing the second-order moment {of} both sides of~\eqref{eq: evolution of the centralized recursion}, we get:
\begin{equation}
\begin{split}
\expec\|\bphib_i\|^2=&\expec\|(I_P-\mu\bcD_{11,i-1})\bphib_{i-1}-\mu\bcD_{12,i-1}\bphic_{i-1}\|^2\\
&+\mu^2\expec\|\bsb_{i}\|^2+\expec\|\bzb_{i}\|^2,\label{eq: centralized error recursion}
\end{split}
\end{equation}
where,  from Assumptions~\ref{assump: quantization noise}  and~\ref{assump: gradient noise} on the quantization and gradient noise processes, we used the fact that:
\begin{align}
\expec[\bx_{i-1}^{\top}\bzb_{i}]&=\expec\left[\expec[\bx_{i-1}^{\top}\bzb_{i}|\boldh_i]\right]=\expec\left[\bx_{i-1}^{\top}\expec[\bzb_{i}|\boldh_i]\right]=0\\
\expec[\bx_{i-1}^{\top}\bsb_{i}]&=\expec\left[\expec\left[\bx_{i-1}^{\top}\bsb_{i}\Big|\{{\bphi_{\ell,i-1}}\}_{\ell=1}^N\right]\right]\notag\\
&=\expec\left[\bx_{i-1}^{\top}\expec\left[\bsb_{i}\Big|\{{\bphi_{\ell,i-1}}\}_{\ell=1}^N\right]\right]=0\\
\expec[\bsb_{i}^\top\bzb_{i}]&=\expec\left[\expec[\bsb_{i}^\top\bzb_{i}|\boldh_i]\right]=\expec\left[\bsb_{i}^\top\expec[\bzb_{i}|\boldh_i]\right]=0
\end{align} 
with $\bx_{i-1}=(I_P-\mu\bcD_{11,i-1})\bphib_{i-1}-\mu\bcD_{12,i-1}\bphic_{i-1}$. Using similar arguments, we can also show that:
\begin{equation}
\begin{split}
\expec\|\bphic_i\|^2=&\expec\|(\cJ'_{\epsilon}-\mu\bcD_{22,i-1})\bphic_{i-1}-\mu\bcD_{21,i-1}\bphib_{i-1}+\mu\widecheck{b}\|^2\\
&+\mu^2\expec\|\bsc_{i}\|^2+\expec\|\bzc_{i}\|^2.\label{eq: non centralized error  recursion}
\end{split}
\end{equation}

Before proceeding, we show that, for small enough $\epsilon$, the $2-$induced matrix norm of $\cJ'_{\epsilon}$ in~\eqref{eq: jordan decomposition of A'} satisfies $\|\cJ'_{\epsilon}\|\in(0,1)$. This property will be used in the subsequent analysis. To establish the property, and by following similar arguments as in~\cite[pp.~516--517]{sayed2014adaptation}, we can first show that  the block diagonal matrix $\cJ'_{\epsilon}$, which is given by~\eqref{eq: jordan decomposition of A'}, satisfies:
\begin{equation}
\label{eq: bounding square of j prime}
\|\cJ_{\epsilon}'\|^2\leq(\rho(\cJ'_{\epsilon})+\gamma\epsilon)^2.
\end{equation}
From~\eqref{eq: jordan decomposition of A'}, we can also show that: 
\begin{equation}
\label{eq: equation to bound rho J prime}
\rho(\cJ'_{\epsilon})\leq(1-\gamma)+\gamma\rho(\cJ_{\epsilon}).
\end{equation} Using the fact that $\rho(\cJ_{\epsilon})\in(0,1)$ from~\cite[Lemma~2]{nassif2020adaptation} and the fact that $\gamma\in(0,1]$, we obtain $\rho(\cJ'_{\epsilon})\in(0,1)$.  Since $\rho(\cJ'_{\epsilon})+\gamma\epsilon$ is non-negative, by replacing~\eqref{eq: equation to bound rho J prime} into~\eqref{eq: bounding square of j prime},  we obtain:
\begin{equation}
\label{eq: bound on norm of J' epsilon}
\begin{split}
\|\cJ_{\epsilon}'\|&\leq(1-\gamma)+\gamma\rho(\cJ_{\epsilon})+\gamma\epsilon\\
&=1-\gamma(1-\rho(\cJ_{\epsilon})-\epsilon).
\end{split}
\end{equation}
This identity will also be used in the  subsequent analysis. Regarding the block diagonal matrix $I-\cJ_{\epsilon}'$, which will  appear in the following analysis, we can establish (by following similar arguments as in~\cite[pp.~516--517]{sayed2014adaptation}):
\begin{equation}
\label{eq: bound on I-J epsilon'}
\|I-\cJ_{\epsilon}'\|^2\overset{\eqref{eq: jordan decomposition of A'}}=\|\gamma(I-\cJ_{\epsilon})\|^2 =\gamma^2\|I-\cJ_{\epsilon}\|^2 \leq\gamma^2(\rho(I-\cJ_{\epsilon})+\epsilon)^2,
\end{equation}
where $\rho(I-\cJ_{\epsilon})\in(0,2)$ since $\rho(\cJ_{\epsilon})\in(0,1)$.

Now, returning to the error recursions~\eqref{eq: centralized error recursion} and~\eqref{eq: non centralized error  recursion}, and using similar arguments as those used to  establish inequalities (119) and (124) in~\cite[Appendix D]{nassif2020adaptation}, we can show that:
\begin{equation}
\label{eq: centralized error vector recursion inequality}
\begin{split}
\expec\|\bphib_i\|^2\leq&(1-\mu\sigma_{11})\expec\|\bphib_{i-1}\|^2+\frac{\mu\sigma_{12}^2}{\sigma_{11}}\expec\|\bphic_{i-1}\|^2+\\
&\expec\|\bzb_i\|^2+\mu^2\expec\|\bsb_{i}\|^2,
\end{split}
\end{equation}
and 
\begin{equation}
\label{eq: centralized error vector recursion inequality 2}
\begin{split}
&\expec\|\bphic_i\|^2\leq\left(\|\cJ'_{\epsilon}\|+\frac{3\mu^2\sigma_{22}^2}{1-\|\cJ'_{\epsilon}\|}\right)\expec\|\bphic_{i-1}\|^2+\\
&\left(\hspace{-0.5mm}\frac{3\mu^2\sigma_{21}^2}{1-\|\cJ'_{\epsilon}\|}\hspace{-0.5mm}\right)\hspace{-1mm}\expec\|\bphib_{i-1}\|^2\hspace{-0.5mm}+\hspace{-0.5mm}\left(\hspace{-0.5mm}\frac{3\mu^2}{1-\|\cJ'_{\epsilon}\|}\hspace{-0.5mm}\right)\hspace{-0.5mm}\|\widecheck{b}\|^2\hspace{-1mm}+\hspace{-0.5mm}\expec\|\bzc_{i}\|^2\hspace{-1mm}+\hspace{-0.5mm}\mu^2\expec\|\bsc_{i}\|^2
\end{split}
\end{equation}
for some positive constant $\sigma_{11}$ and non-negative constants $\sigma_{12},\sigma_{21}$, and $\sigma_{22}$ independent of $\mu$. As it can be seen from~\eqref{eq: bias transformed vector}, $\widecheck{b}=\cV_{L,\epsilon}^\top b$ depends on $b$ in~\eqref{eq: equation for b}, which is defined in terms of the gradients $\{\nabla_{w_k}J_k(w^o_k)\}$.  Since the costs $J_k(w_k)$ are twice differentiable, then $\|b\|^2$ is bounded and we obtain $\|\widecheck{b}\|^2=O(1)$.

For the gradient noise terms, by following similar arguments as in~\cite[Chapter 9]{sayed2014adaptation},~\cite[Appendix D]{nassif2020adaptation} and by using Assumption~\ref{assump: gradient noise}, we can show that:
\begin{equation}
\label{eq: gradient noise term bound 0}
\expec\|\bsb_i\|^2+\expec\|\bsc_i\|^2=\expec\|\cV_{\epsilon}^{-1}\bs_i\|^2\leq v_1^2\beta_{s,\max}^2 \expec\|\bcwt_{i-1}\|^2+ v_1^2\overline{\sigma}^2_s,
\end{equation}
where $v_1=\|\cV_{\epsilon}^{-1}\|$, $\beta_{s,\max}^2=\max_{1\leq k\leq N}\beta_{s,k}^2$, and $\overline{\sigma}^2_s=\sum_{k=1}^N\sigma^2_{s,k}$. %, and $\bar{\sigma}^2_{s,k}=2\beta_{s,k}^2\|w^o_k\|^2+\sigma_{s,k}^2$.
 Using expression~\eqref{eq: wt in terms of phi} and the Jordan decomposition of the matrix $\cA'$ in~\eqref{eq: equation for A'}, we obtain:
\begin{equation}
\begin{split}
&\expec\|\bsb_i\|^2+\expec\|\bsc_i\|^2\\&\leq v_1^2\beta_{s,\max}^2 \expec\|\cA'\bphit_{i-1}\|^2+ v_1^2\overline{\sigma}^2_s\\
&\leq v_1^2\beta_{s,\max}^2 \expec\|\cV_{\epsilon}\Lambda'(\cV_{\epsilon}^{-1}\bphit_{i-1})\|^2+ v_1^2\overline{\sigma}^2_s\\
&\leq v_1^2\beta_{s,\max}^2 v_2^2(\expec\|\bphib_{i-1}\|^2+\expec\|\bphic_{i-1}\|^2)+ v_1^2\overline{\sigma}^2_s,\label{eq: gradient noise term bound}
\end{split}
\end{equation}
where $v_2=\|\cV_{\epsilon}\|$. 

Using the bound~\eqref{eq: gradient noise term bound} into~\eqref{eq: centralized error vector recursion inequality} and~\eqref{eq: centralized error vector recursion inequality 2}, we obtain:
\begin{equation}
\label{eq: centralized error vector recursion inequality with bounding gradient noise}
\begin{split}
&\expec\|\bphib_i\|^2\leq(1-\mu\sigma_{11}+\mu^2 v_1^2\beta_{s,\max}^2 v_2^2)\expec\|\bphib_{i-1}\|^2+\\
&\left(\frac{\mu\sigma_{12}^2}{\sigma_{11}}+\mu^2 v_1^2\beta_{s,\max}^2 v_2^2\right)\expec\|\bphic_{i-1}\|^2+\expec\|\bzb_i\|^2+\mu^2 v_1^2\overline{\sigma}^2_s,
\end{split}
\end{equation}
and
\begin{equation}
\label{eq: centralized error vector recursion inequality 2 with bounding gradient noise}
\begin{split}
&\expec\|\bphic_i\|^2\leq\left(\|\cJ'_{\epsilon}\|+\frac{3\mu^2\sigma_{22}^2}{1-\|\cJ'_{\epsilon}\|}+\mu^2 v_1^2\beta_{s,\max}^2 v_2^2\right)\expec\|\bphic_{i-1}\|^2+\\
&\left(\frac{3\mu^2\sigma_{21}^2}{1-\|\cJ'_{\epsilon}\|}+\mu^2 v_1^2\beta_{s,\max}^2 v_2^2\right)\expec\|\bphib_{i-1}\|^2+\left(\frac{3\mu^2}{1-\|\cJ'_{\epsilon}\|}\right)\|\widecheck{b}\|^2\\
&\quad+\expec\|\bzc_{i}\|^2+\mu^2 v_1^2\overline{\sigma}^2_s.
\end{split}
\end{equation}

Now, for the quantization noise, we have:
\begin{equation}
\label{eq: expression 1}
\begin{split}
\expec\|\bzb_i\|^2+\expec\|\bzc_i\|^2=\expec\|\cV_{\epsilon}^{-1}\bz_i\|^2\overset{\eqref{eq: equation for z}}\leq v_1^2\left(\sum_{k=1}^N\expec\|\bz_{k,i}\|^2\right).
\end{split}
\end{equation}
From~\eqref{eq: quantization error vector} and Assumption~\ref{assump: quantization noise}, and since $\bchi_{k,i}=\bphit_{k,i-1}-\bpsit_{k,i}$, we can write:
\begin{equation}
\label{eq: expression 3}
\begin{split}
\expec\|\bz_{k,i}\|^2\leq{\beta}^2_{q,k}\expec\|\bphit_{k,i-1}-\bpsit_{k,i}\|^2+{\sigma}^2_{q,k},
\end{split}
\end{equation}
and, therefore,
\begin{equation}
\label{eq: expression 4}
\begin{split}
\expec\|\bz_i\|^2\leq {\beta}_{q,\max}^2\expec\|\bphit_{i-1}-\bpsit_{i}\|^2+\overline{\sigma}^2_q,
\end{split}
\end{equation}
where ${\beta}_{q,\max}^2=\max_{1\leq k\leq N}\{{\beta}^2_{q,k}\}$ and $\overline{\sigma}^2_q=\sum_{k=1}^N{\sigma}^2_{q,k}$. Since the analysis is facilitated by transforming the network vectors into the Jordan decomposition basis of the matrix $\cA'$, we proceed by noting that the term $\expec\|\bz_i\|^2$ can be bounded as follows:
\begin{equation}
\label{eq: expression 6}
\begin{split}
\expec\|\bz_i\|^2&\overset{\eqref{eq: expression 4}}\leq  {\beta}_{q,\max}^2\expec\|\cV_{\epsilon}\cV_{\epsilon}^{-1}(\bphit_{i-1}-\bpsit_{i})\|^2+ \overline{\sigma}^2_q\\
&~\leq {\beta}_{q,\max}^2\|\cV_{\epsilon}\|^2\expec\|\cV_{\epsilon}^{-1}(\bphit_{i-1}-\bpsit_{i})\|^2+ \overline{\sigma}^2_q\\
&~\leq v_2^2 {\beta}_{q,\max}^2[\expec\|\overline{\bchi}_{i}\|^2+\expec\|\widecheck{\bchi}_i\|^2]+\overline{\sigma}^2_q,
\end{split}
\end{equation}
where
\begin{align}
\overline{\bchi}_{i}&\triangleq\cU^{\top}(\bphit_{i-1}-\bpsit_{i}),\\
\widecheck{\bchi}_i&\triangleq\cV_{L,\epsilon}^{\top}(\bphit_{i-1}-\bpsit_{i}).
\end{align}
Therefore, by combining~\eqref{eq: expression 1} and~\eqref{eq: expression 6}, we obtain:
\begin{equation}
\label{eq: expression 10}
\begin{split}
\expec\|\bzb_i\|^2+\expec\|\bzc_i\|^2\leq v_1^2v_2^2{\beta}_{q,\max}^2[\expec\|\overline{\bchi}_{i}\|^2+\expec\|\widecheck{\bchi}_i\|^2]+v_1^2\overline{\sigma}^2_q.
\end{split}
\end{equation}
We focus now on deriving the network vector transformed recursions $\overline{\bchi}_{i}$ and $\widecheck{\bchi}_i$. Subtracting $\bphit_{i-1}$ from both sides of~\eqref{eq: network error vector psi} and using~\eqref{eq: wt in terms of phi}, we obtain:
\begin{equation}
\label{eq: expression 5}
\begin{split}
\bphit_{i-1}-\bpsit_{i}=(I_{M}-\cA'+\mu\bcH_{i-1}\cA')\bphit_{i-1}+\mu\bs_i-\mu b.
\end{split}
\end{equation}
By multiplying both sides of~\eqref{eq: expression 5} by $\cV_{\epsilon}^{-1}$ and by using~\eqref{eq: transformed variable phi definition}--\eqref{eq: bias transformed vector},~\eqref{eq: definition of D11}--\eqref{eq: definition of D22}, and the Jordan decomposition of the matrix $\cA'$, we obtain:
\begin{equation}
\label{eq: expression 8}
\begin{split}
\left[\begin{array}{c}
\overline{\bchi}_{i}\\
\widecheck{\bchi}_i
\end{array}
\right]=&\left[\begin{array}{cc}
\mu\bcD_{11,i-1}&\mu\bcD_{12,i-1}\\
\mu\bcD_{21,i-1}&I_{M-P}-\cJ'_{\epsilon}+\mu\bcD_{22,i-1}
\end{array}\right]\left[\begin{array}{c}
\bphib_{i-1}\\
\bphic_{i-1}
\end{array}
\right]\\
&+\mu\left[\begin{array}{c}
\bsb_i\\
\bsc_{i}
\end{array}
\right]-\mu\left[\begin{array}{c}
0\\
\widecheck{b}
\end{array}
\right].
\end{split}
\end{equation}
Again, by using similar arguments as those used to  establish inequalities (119) and (124) in~\cite[Appendix D]{nassif2020adaptation} and by using   {Jensen's} inequalities $\|x+y\|^2\leq 2\|x\|^2+2\|y\|^2$ and $\|x+y+z\|^2\leq 3\|x\|^2+3\|y\|^2+3\|z\|^2$, we can {verify} that:
\begin{equation}
\label{eq: centralized error vector recursion inequality on delta}
\begin{split}
%\expec\|\overline{\bchi}_{i}\|^2\leq&\frac{\mu^2\sigma^2_{11}}{\alpha}\expec\|\bphib_{i-1}\|^2+\frac{\mu^2\sigma_{12}^2}{1-\alpha}\expec\|\bphic_{i-1}\|^2+\mu^2\expec\|\bsb_{i}\|^2,
\expec\|\overline{\bchi}_{i}\|^2\leq&2\mu^2\sigma^2_{11}\expec\|\bphib_{i-1}\|^2+2\mu^2\sigma_{12}^2\expec\|\bphic_{i-1}\|^2+\mu^2\expec\|\bsb_{i}\|^2,
\end{split}
\end{equation}
and
\begin{equation}
\label{eq: centralized error vector recursion inequality on delta check}
\begin{split}
&\expec\|\widecheck{\bchi}_{i}\|^2\leq6\mu^2\sigma^2_{21}\expec\|\bphib_{i-1}\|^2+\\
&2\left(\|I-\cJ_{\epsilon}'\|^2+3\mu^2\sigma^2_{22}\right)\expec\|\bphic_{i-1}\|^2+6\mu^2\|\widecheck{b}\|^2+\mu^2\expec\|\bsc_{i}\|^2.
\end{split}
\end{equation}
By combining expressions~\eqref{eq: centralized error vector recursion inequality on delta} and~\eqref{eq: centralized error vector recursion inequality on delta check}, we obtain:
\begin{equation}
\label{eq: centralized error vector recursion inequality on sum of  delta}
\begin{split}
\expec\|\overline{\bchi}_{i}\|^2&+\expec\|\widecheck{\bchi}_{i}\|^2\leq\left(2\mu^2\sigma^2_{11}+6\mu^2\sigma^2_{21}\right)\expec\|\bphib_{i-1}\|^2+\\
&\left(2\mu^2\sigma_{12}^2+2\|I-\cJ_{\epsilon}'\|^2+6\mu^2\sigma^2_{22}\right)\expec\|\bphic_{i-1}\|^2+\\
&~6\mu^2\|\widecheck{b}\|^2+\mu^2\left(\expec\|\bsb_{i}\|^2+\expec\|\bsc_{i}\|^2\right).
\end{split}
\end{equation}
Now, by using the bound~\eqref{eq: gradient noise term bound} {in~\eqref{eq: centralized error vector recursion inequality on sum of  delta}}, we  obtain:
\begin{equation}
\label{eq: centralized error vector recursion inequality on sum of  delta 1}
\begin{split}
&\expec\|\overline{\bchi}_{i}\|^2+\expec\|\widecheck{\bchi}_{i}\|^2\leq6\mu^2\|\widecheck{b}\|^2+\mu^2 v_1^2\overline{\sigma}^2_s+\\
&\left(2\mu^2\sigma^2_{11}+6\mu^2\sigma^2_{21}+\mu^2 v_1^2\beta_{s,\max}^2 v_2^2\right)\expec\|\bphib_{i-1}\|^2+\\
&\left(2\mu^2\sigma_{12}^2+2\|I-\cJ_{\epsilon}'\|^2+6\mu^2\sigma^2_{22}+\mu^2 v_1^2\beta_{s,\max}^2 v_2^2\right)\expec\|\bphic_{i-1}\|^2.
\end{split}
\end{equation}

 Using~\eqref{eq: centralized error vector recursion inequality on sum of  delta 1} and~\eqref{eq: expression 10} {in~\eqref{eq: centralized error vector recursion inequality with bounding gradient noise}} and~\eqref{eq: centralized error vector recursion inequality 2 with bounding gradient noise}, we finally find that  the variances of $\bphib_i$ and $\bphic_i$ are coupled and recursively bounded as:
\begin{equation}
\label{eq: single inequality recursion text new}
\left[\begin{array}{c}
\expec\|\bphib_i\|^2\\
\expec\|\bphic_i\|^2
\end{array}\right]\preceq\Gamma\left[\begin{array}{c}
\expec\|\bphib_{i-1}\|^2\\
\expec\|\bphic_{i-1}\|^2
\end{array}\right]+\left[\begin{array}{c}
e+v_1^2\overline{\sigma}^2_{q}\\
f+v_1^2\overline{\sigma}^2_{q}
\end{array}\right],
\end{equation}
where $\Gamma$ is {the} $2\times 2$ matrix given by:
\begin{equation}
\Gamma=\left[\begin{array}{cc}
a&b\\
c&d
\end{array}\right],\label{eq: general form for gamma}
\end{equation}
with
\begingroup
\allowdisplaybreaks
\begin{align}
a%&\triangleq 1-\mu\sigma_{11}+\mu^2 v_1^2\beta_{s,\max}^2 v_2^2+v_1^2v_2^2{\beta}_{q,\max}^2\left({2\mu^2\sigma^2_{11}}+6\mu^2\sigma^2_{21}+\mu^2 v_1^2\beta_{s,\max}^2 v_2^2\right)\notag\\
&=1-\mu\sigma_{11}+O(\mu^2)=\cblue{1-\Theta(\mu)},\label{eq: constant a definition new}\\
b%&\triangleq \mu\frac{\sigma_{12}^2}{\sigma_{11}}+\mu^2 v_1^2\beta_{s,\max}^2 v_2^2+v_1^2v_2^2{\beta}_{q,\max}^2\left({2\mu^2\sigma_{12}^2}+2\|I-\cJ_{\epsilon}'\|^2+6\mu^2\sigma^2_{22}+\mu^2 v_1^2\beta_{s,\max}^2 v_2^2\right)\notag\\
%&= \mu\frac{\sigma_{12}^2}{\sigma_{11}}+2v_1^2v_2^2{\beta}_{q,\max}^2\|I-\cJ_{\epsilon}'\|^2+O(\mu^2)
&=2v_1^2v_2^2{\beta}_{q,\max}^2\|I-\cJ_{\epsilon}'\|^2+O(\mu),\\
c%&\triangleq\frac{3\mu^2\sigma_{21}^2}{1-\|\cJ'_{\epsilon}\|}+\mu^2 v_1^2\beta_{s,\max}^2 v_2^2+v_1^2v_2^2{\beta}_{q,\max}^2\left({2\mu^2\sigma^2_{11}}+6\mu^2\sigma^2_{21}+\mu^2 v_1^2\beta_{s,\max}^2 v_2^2\right)\notag\\
&=O(\mu^2),\\
d%&\triangleq\|\cJ'_{\epsilon}\|+\frac{3\mu^2\sigma_{22}^2}{1-\|\cJ'_{\epsilon}\|}+\mu^2 v_1^2\beta_{s,\max}^2 v_2^2+v_1^2v_2^2{\beta}_{q,\max}^2\left({2\mu^2\sigma_{12}^2}+2\|I-\cJ_{\epsilon}'\|^2+6\mu^2\sigma^2_{22}+\mu^2 v_1^2\beta_{s,\max}^2 v_2^2\right)\notag\\
&=\|\cJ'_{\epsilon}\|+2v_1^2v_2^2{\beta}_{q,\max}^2\|I-\cJ_{\epsilon}'\|^2+O(\mu^2),\label{eq: constant d definition new}\\
e&\triangleq \mu^2\left(6\|\widecheck{b}\|^2v_1^2v_2^2{\beta}_{q,\max}^2+ v_1^2\overline{\sigma}^2_sv_1^2v_2^2{\beta}_{q,\max}^2+ v_1^2\overline{\sigma}^2_s\right)%=O(\mu^2)
,\label{eq: constant e definition new}\\
f&\triangleq\mu^2\hspace{-1mm}\left(\hspace{-1mm}\frac{3\|\widecheck{b}\|^2}{1-\|\cJ'_{\epsilon}\|}+6v_1^2v_2^2{\beta}_{q,\max}^2\|\widecheck{b}\|^2+ \hspace{-0.5mm}v_1^4\overline{\sigma}^2_sv_2^2{\beta}_{q,\max}^2\hspace{-0.5mm}+ v_1^2\overline{\sigma}^2_s\hspace{-1mm}\right)%=O(\mu^2)
\label{eq: constant f definition new}
\end{align}
\endgroup
\cblue{where the notation $g(\mu)=\Theta(\mu)$ signifies that there exists a positive constant $c$ such that $\lim_{\mu\rightarrow 0}\frac{g(\mu)}{\mu}=c$.}

If the matrix $\Gamma$ is stable, i.e., $\rho(\Gamma)<1$, then by iterating~\eqref{eq: single inequality recursion text new}, we arrive at:
\begin{align}
\label{eq: single inequality recursion steady state 2 new}
\limsup_{i\rightarrow\infty}\left[\begin{array}{c}
\expec\|\bphib_i\|^2\\
\expec\|\bphic_i\|^2
\end{array}\right]\preceq(I-\Gamma)^{-1}\left[\begin{array}{c}
e+v_1^2\overline{\sigma}^2_{q}\\
f+v_1^2\overline{\sigma}^2_{q}
\end{array}\right].
\end{align}
As we will see in the following, for some given data settings (captured by $\{\sigma_{11}^2,\sigma_{12}^2,\sigma_{21}^2,\sigma_{22}^2,\beta^2_{s,\max},\overline{\sigma}^2_s\}$), small step-size parameter $\mu$, network topology (captured by the matrix $\cA$ and its eigendecomposition  $\{v_1^2,v_2^2,\cJ_{\epsilon}\}$), and quantizer settings (captured by $\{\beta_{q,\max}^2,\overline{\sigma}^2_{q}\}$), the stability of $\Gamma$ can be controlled by the mixing parameter $\gamma$ used in step~\eqref{eq: stepc}. As it can be observed from~\eqref{eq: jordan decomposition of A'},~\eqref{eq: bound on norm of J' epsilon}, and~\eqref{eq: bound on I-J epsilon'}, this parameter influences $\|\cJ'_{\epsilon}\|^2$ and $\|I-\cJ'_{\epsilon}\|^2$.  Generally speaking, and since the spectral radius of a matrix is upper bounded by its $1-$norm, the matrix $\Gamma$ is stable if:
\begin{equation}
\label{eq: condition ensuring stability of gamma new}
\rho(\Gamma)\leq\max\{|a|+|c|,|b|+|d|\}<1.
\end{equation}
Since $\sigma_{11}>0$ and $\gamma \neq 0$, a sufficiently small $\mu$ can make $|a|+|c|$ strictly smaller than 1. For $|b|+|d|$, observe that if the mixing parameter $\gamma$ is chosen such that:
%\begin{equation}
%\label{eq: condition on the mixing parameter new}
%\|\cJ'_{\epsilon}\|+4v_1^2v_2^2\beta^2_{z,\max}\|I-\cJ'_{\epsilon}\|^2\leq 1-\kappa<1
%\end{equation}
\begin{equation}
\label{eq: condition on the mixing parameter new}
\|\cJ'_{\epsilon}\|+4v_1^2v_2^2\beta^2_{q,\max}\|I-\cJ'_{\epsilon}\|^2<1,
\end{equation}
%where $0<\kappa\ll 1$ is a small positive parameter independent of $\mu$, 
then $|b|+|d|$ can be made strictly smaller than one for sufficiently small $\mu$. It is therefore clear that the RHS of~\eqref{eq: condition ensuring stability of gamma new} can be made strictly smaller than one for sufficiently small $\mu$ and for a mixing parameter $\gamma\neq 0$ satisfying condition~\eqref{eq: condition on the mixing parameter new}. Now, to derive condition~\eqref{eq: mixing parameter condition theorem} on the mixing parameter $\gamma$, we start by noting that the LHS of~\eqref{eq: condition on the mixing parameter new} can be upper bounded by:
\begin{equation}
\begin{split}
&\|\cJ'_{\epsilon}\|+4v_1^2v_2^2\beta^2_{q,\max}\|I-\cJ'_{\epsilon}\|^2\\
&\overset{\eqref{eq: bound on norm of J' epsilon},\eqref{eq: bound on I-J epsilon'}}\leq \hspace{-1.5mm}1-\gamma(1-\rho(\cJ_{\epsilon})-\epsilon)+4v_1^2v_2^2\beta^2_{q,\max}\gamma^2(\rho(I-\cJ_{\epsilon})+\epsilon)^2.
\end{split}
\end{equation}
The upper bound in the above inequality is guaranteed to be strictly smaller than 1 if:
\begin{equation}
\label{eq: condition on 1-zeta to solve}
4v_1^2v_2^2\beta^2_{q,\max}\gamma^2(\rho(I-\cJ_{\epsilon})+\epsilon)^2-\gamma(1-\rho(\cJ_{\epsilon})-\epsilon)<0.
\end{equation}
Now, by using the above condition and the fact that $\gamma$ must be in $(0,1]$, we  obtain condition~\eqref{eq: mixing parameter condition theorem}. 

% for a small positive constant $\kappa$.
Under condition~\eqref{eq: mixing parameter condition theorem}, $\rho(\Gamma)<1$, and consequently, the matrix $\Gamma$ is stable. Moreover, it holds that:
\begin{equation}
\label{eq: I-Gamma new}
(I-\Gamma)=\left[\begin{array}{cc}
\cblue{\Theta(\mu)}&O(1)\\
O(\mu^2)&\cblue{\Theta(1)}
\end{array}\right],
\end{equation}
and:
\begin{equation}
\label{eq: inverse of I-Gamma new}
(I-\Gamma)^{-1}=\left[\begin{array}{cc}
\cblue{\Theta(\mu^{-1})}&O(\mu^{-1})\\
O(\mu)&\cblue{\Theta(1)}
\end{array}\right].
\end{equation}
Now, using~\eqref{eq: constant e definition new},~\eqref{eq: constant f definition new}, and~\eqref{eq: inverse of I-Gamma new} into~\eqref{eq: single inequality recursion steady state 2 new}, we arrive at:
\begin{align}
\label{eq: single inequality recursion steady state 3 new}
&\limsup_{i\rightarrow\infty}\left[\begin{array}{c}
\expec\|\bphib_i\|^2\\
\expec\|\bphic_i\|^2
\end{array}\right]\preceq\notag\\&\hspace{-1.5mm}\left[\begin{array}{c}
\hspace{-2mm}\beta^2_{q,\max}\overline{\sigma}^2_s O(\mu)\hspace{-1mm}+\hspace{-1mm}\beta^2_{q,\max}\|\widecheck{b}\|^2 O(\mu)\hspace{-1mm}+\hspace{-1mm}\overline{\sigma}^2_s O(\mu)\hspace{-1mm}+\hspace{-1mm}\|\widecheck{b}\|^2 O(\mu)\hspace{-1mm}+\hspace{-1mm}\overline{\sigma}^2_{q}O(\mu^{-1})\hspace*{-2mm}\\
O(\mu^2)+\overline{\sigma}^2_{q} O(1)\hspace{-5.5mm}
\end{array}\right]
\end{align}
By noting that:
\begin{align}
\limsup_{i\rightarrow\infty}\expec\|\bcwt_i\|^2&\overset{\eqref{eq: wt in terms of phi}}=\limsup_{i\rightarrow\infty}\expec\|\cA'\bphit_i\|^2\notag\\
&=\limsup_{i\rightarrow\infty}\expec\|\cV_{\epsilon}\Lambda_{\epsilon}'\cV_{\epsilon}^{-1}\bphit_i\|^2\notag\\
&\leq\limsup_{i\rightarrow\infty}\|\cV_{\epsilon}\|^2\|\Lambda_{\epsilon}'\|^2[\expec\|\bphib_i\|^2+\expec\|\bphic_i\|^2],\label{eq: equation relating w tilde to phi tilde}
\end{align}
and by using~\eqref{eq: single inequality recursion steady state 3 new}, we can finally conclude~\eqref{eq: steady state mean square error}.

%===============================
% Appendix: rate stability
%===============================
\section{Bit Rate stability}
\label{app: rate stability}
Equation~\eqref{eq: sigma square} follows straightforwardly from~\eqref{eq: beta and sigma relation}. Regarding~\eqref{eq: rate stability result}, we first note from~\eqref{eq: non-linearities for ANQ} that $g(t)\geq 0$ for non-negative arguments, and that:
\begin{equation}
\label{eq: bound on n(t) 1}
|\boldsymbol{n}(t)|=\boldsymbol{n}(t)\leq\lfloor g(t)\rfloor + 1\leq g(t) + 1, \qquad \text{for }t\geq 0
\end{equation}
where the first inequality follows from~\eqref{eq: random index n} and~\eqref{eq: equation for m}, and  the fact that the random index $\boldsymbol{n}(t)$ can be either equal to $\lfloor g(t)\rfloor $ or to the integer $\lfloor g(t)+1\rfloor $. The second inequality follows from the definition of the floor operator. Likewise, for $t<0$, we obtain:
\begin{equation}
\label{eq: bound on n(t) 2}
\boldsymbol{n}(t)\geq\lfloor g(t)\rfloor =\lfloor- g(-t)\rfloor =-\lceil g(-t)\rceil\geq- g(-t)-1.
\end{equation}
Combining~\eqref{eq: bound on n(t) 1} and~\eqref{eq: bound on n(t) 2}, we obtain:
\begin{equation}
|\boldsymbol{n}(t)|\leq g(|t|)+1,
\end{equation}
and, consequently, we can write:
\begin{equation}
\lceil\log_2 (|\boldsymbol{n}(t)|+1)\rceil\leq\log_2 (|\boldsymbol{n}(t)|+1)+1\leq\log_2 (g(|t|) + 2)+1.
\label{eq:totalineq2}
\end{equation}
Now, from~\eqref{eq: non-linearities for ANQ}, we can write:
\begin{equation}
\log_2 (g(|t|) + 2)=
\log_2\left(\frac{\ln\left(1+\displaystyle{\frac{\omega}{\eta}} \sqrt{t^2}\right)}{2\ln\left(\omega+\sqrt{1+\omega^2}\right)}
+2\right),
\end{equation}
which can be verified to be a concave function w.r.t. to the argument $t^2$.  Therefore, if we consider a random input~$\boldsymbol{t}$ in~\eqref{eq:totalineq2}, take the expectation, and use  {Jensen's} inequality, we get:
\begin{equation}
\label{eq: intermediate equation 127}
\expec\left[
\lceil\log_2 (|\boldsymbol{n}(\boldsymbol{t})|+1)\rceil\right]\leq\log_2\left(\frac{\ln\left(1+\displaystyle{\frac{\omega}{\eta}} \sqrt{\expec[\boldsymbol{t}^2]}\right)}{2\ln\left(\omega+\sqrt{1+\omega^2}\right)}+2\right)+1.
\end{equation}
By choosing $\boldsymbol{t}=[\bchi_{k,i}]_j$ in~\eqref{eq: intermediate equation 127} and by using the fact that $([\bchi_{k,i}]_j)^2\leq \|\bchi_{k,i}\|^2$, we can upper bound the individual summand in~\eqref{eq: bit rate general at agent k} by the quantity:
\begin{equation}
2+\log_2\left(\frac{\ln\left(1+\displaystyle{\frac{\omega}{\eta}}\sqrt{\expec\|\bchi_{k,i}\|^2}\right)}{2\ln\left(\omega+\sqrt{1+\omega^2}\right)}+2\right).
\label{eq:finalrateformula}
\end{equation}
By taking the limit superior of~\eqref{eq: centralized error vector recursion inequality on sum of  delta 1} as $i\rightarrow\infty$  and by using~\eqref{eq: single inequality recursion steady state 3 new}, we obtain $l_k=\limsup_{i\rightarrow\infty}\expec\|\bchi_{k,i}\|^2=O(\mu^2)$. Consequently, we find that the limit superior of~\eqref{eq:finalrateformula} as $i\rightarrow\infty$ is:
\begin{equation}
2+\log_2\left(\frac{\ln\left(1+\displaystyle{\frac{\omega}{\eta}}\sqrt{l_k}\right)}{2\ln\left(\omega+\sqrt{1+\omega^2}\right)}+2\right).
\label{eq:finalrateformula 2}
\end{equation}
where we have used the fact that the function in~\eqref{eq:finalrateformula} is continuous and increasing in the argument $\expec\|\bchi_{k,i}\|^2$ and, hence,  the limit superior is preserved.
 Under the conditions in~\eqref{eq: condition of the theorem 2}, the above quantity is $O(1)$, and consequently, the rate in~\eqref{eq: bit rate general at agent k} satisfies~\eqref{eq: rate stability result}.

\end{appendices}
%=====================================
% Sec: References
%=====================================
\bibliographystyle{IEEEbib}
{\balance{
\bibliography{reference}}}

\newpage
 \thispagestyle{empty}
\section*{Supplementary material}
\subsection{Decentralized consensus optimization}
In this section, we illustrate the performance of the quantized decentralized approach~\eqref{eq: decentralized learning approach with quantization} when applied to solve the consensus optimization problem~\eqref{eq: consensus optimization}. The network of $N=50$ nodes with the link matrix shown in Fig.~\ref{fig: data settings} {\emph{(left)}} is considered.  Each agent is subjected to streaming data $\{\bd_k(i),\bu_{k,i}\}$ satisfying the linear regression model~\eqref{eq: linear data regression model} for some unknown $L\times 1$ vector $w^\star_k$ with $\bv_k(i)$ denoting a zero-mean measurement noise and $L=5$. A mean-square-error cost of the form $J_k(w_k)=\frac{1}{2}\expec|\bd_k(i)-\bu_{k,i}^\top w_k|^2$ is associated with each agent $k$. Similarly to the settings in Sec.~\ref{sec: simulation results}, the processes $\{\bu_{k,i},\bv_k(i)\}$ are assumed to be zero-mean Gaussian with: i) $\expec\bu_{k,i}\bu_{\ell,i}^\top=R_{u,k}=\sigma^2_{u,k}I_L$ if $k=\ell$ and zero otherwise; ii) $\expec\bv_{k}(i)\bv_{\ell}(i)=\sigma^2_{v,k}$ if $k=\ell$ and zero otherwise; and iii) $\bu_{k,i}$ and $\bv_{k}(i)$ are independent of each other. The variances $\sigma^2_{u,k}$ and $\sigma^2_{v,k}$ are illustrated in Fig.~\ref{fig: data settings} {\textit{(right)}}. The signal $\cw^\star=\col\{w^\star_1,\ldots,w^\star_N\}$ is also generated  by smoothing a signal $\cw_o$, which is randomly generated from the Gaussian distribution  $\cN(0.4\times\mathds{1}_{NL},I_{NL})$,  by a graph diffusion kernel with $\tau=3$. The matrix $\cU$ is generated according to $\cU=U\otimes I_L$   where $U=\frac{1}{\sqrt{N}}\mathds{1}_N$. The combination matrix $\cA$ satisfying the conditions in~\eqref{eq: condition A} and having the same structure as the graph   is found by following the same approach as in~\cite{nassif2019distributed}.

 In Fig.~\ref{fig: variable step-size supplementary} {\emph{(left)}}, we report the network mean-square-deviation (MSD) learning curves for~3 different values of the step-size $\mu$. The results were averaged over 100 Monte-Carlo runs.  We used the probabilistic ANQ quantizer of Example 2 with $\omega_k=0.25$ and $\eta_k=\frac{\mu}{\sqrt{2L}}$. We set $\gamma=0.88$. We observe that, in steady-state, the network MSD increases by approximately $3$dB when $\mu$ goes from $\mu_0$ to $2\mu_0$. This means that the performance is on the order of $\mu$. In Fig.~\ref{fig: variable step-size supplementary}  {\emph{(middle)}}, we report the average number of bits per node, per component,  computed according to~\eqref{eq: average number of bits per node}.
%\begin{equation}
%R(i)=\frac{1}{N}\sum_{k=1}^N\frac{1}{L}\expec \left[\boldsymbol{r}(\bw_{k,i})\right]
%\end{equation}
%where $r(\bw_{k,i})$ is the total number of required bits to encode the $L\times 1$ vector $\bw_{k,i}$. %Throughout the simulation section, the adaptive  encoding scheme proposed in~\cite{lee2021finite} was adopted. 
As it can be observed, and thanks to the variable-rate quantization, a finite average number of bits is guaranteed (approximately 2.8 bits/component/iteration are required {on average} in steady-state).  

We report in Fig.~\ref{fig: variable step-size supplementary} \emph{(right)} the MSD learning curves when the QSGD quantizer of Table~\ref{table: examples of quantizers} (row 7) is employed instead of the probabilistic ANQ. Apart from the quantizer scheme, the same settings as above were assumed. For the QSGD scheme, we set the number of quantization levels $s$ to $2$. As it can be observed, this choice allows us to compare the average number of bits for the ANQ and QSGD quantizers for similar values of steady-state MSD. From Table~\ref{table: examples of quantizers} (row 7, column 5), the bit-budget required to encode a $5\times 1$ vector using the QSGD scheme ($s=2$) is given by $B_{\text{HP}}+10$. Now, by replacing $B_{\text{HP}}$ by 32 (since we are performing the experiments on MATLAB 2022a which uses 32 bits to represent a floating number in single-precision), we find that the QSGD quantizer requires, at each iteration $i$, an average number of bits per node, per component, equal to $\frac{42}{5}=8.4$, which is almost three times higher than the one obtained in steady-state when the probabilistic ANQ is used. This is expected since the QSGD scheme requires encoding the norm of the input vector with very high precision.

\subsection{Role of the mixing parameter}
To illustrate the impact of the quantization noise parameter $\beta^2_{q,k}$ and of the mixing parameter~$\gamma$ on the network stability, we consider in Figure~\ref{fig: stability conditions reviewer 4} the same settings as in Fig.~4 and Fig.~5 (left) of the manuscript. The step-size $\mu$ is set to $0.006$.  The impact of $\beta^2_{q,k}$ can be seen by comparing the probabilistic ANQ quantizer for two different values of parameter $\omega_k$ (according to~\eqref{eq: beta and sigma relation}, a larger value of $\omega_k$ leads to a larger value of $\beta^2_{q,k}$). By comparing the learning curves of Figure~\ref{fig: stability conditions reviewer 4} corresponding to $\gamma=0.88$, we can see that a large value of  $\beta^2_{q,k}$ can lead to a network instability since condition~\eqref{eq: mixing parameter condition theorem} can no longer be satisfied. To ensure stability, and according to~\eqref{eq: mixing parameter condition theorem}, we need to decrease the value of $\gamma$ -- this is illustrated in  the learning curve of Figure~\ref{fig: stability conditions reviewer 4}  corresponding to $\gamma=0.1$ where a smaller value of $\gamma$ leads to network stability. 
\newpage
\begin{figure*}
\begin{center}
\includegraphics[scale=0.32]{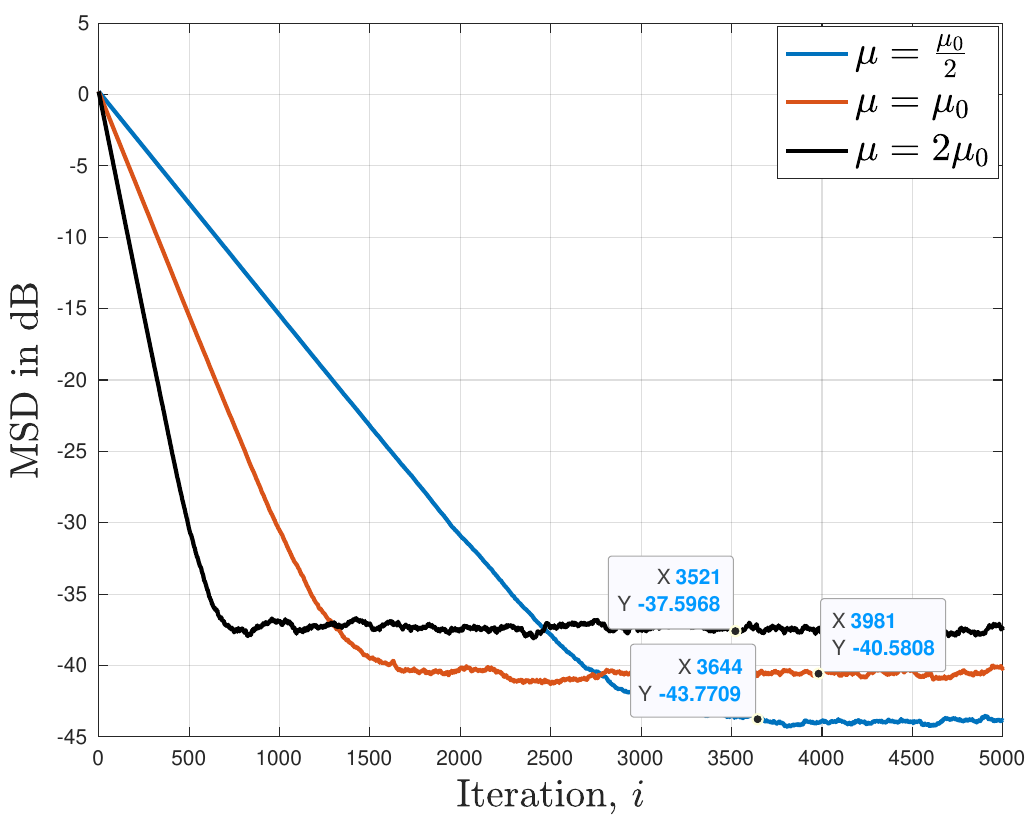}\quad
\includegraphics[scale=0.32]{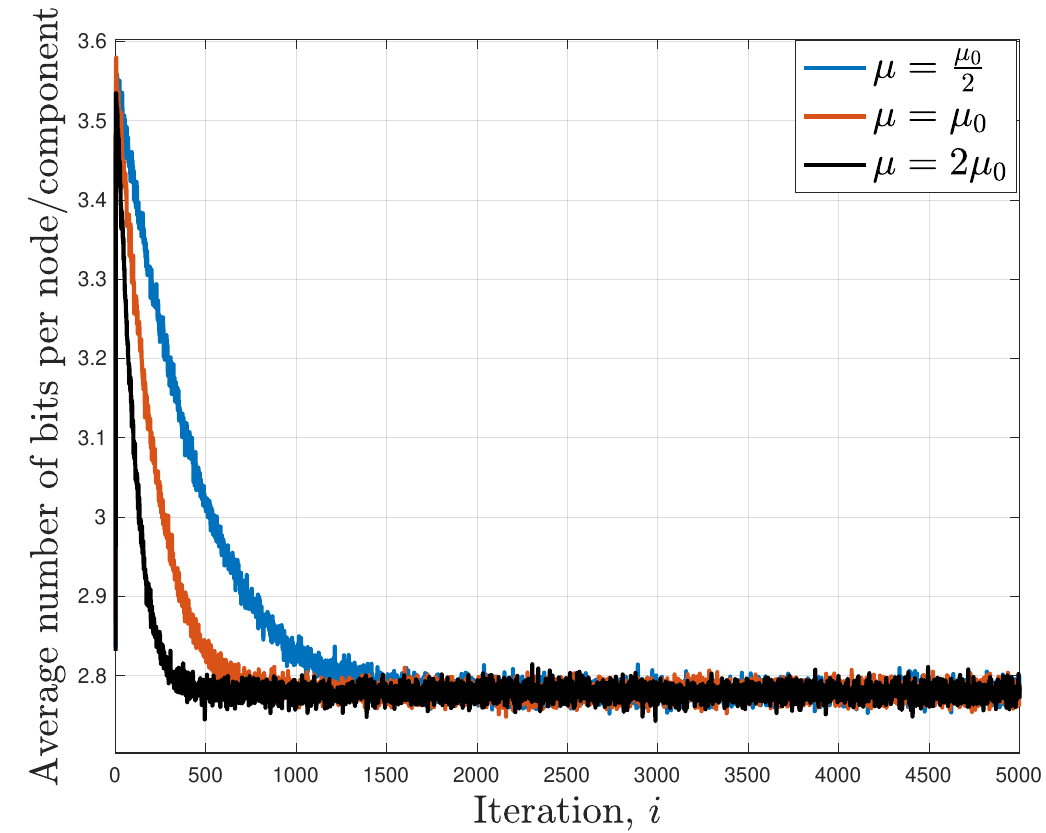}
\includegraphics[scale=0.32]{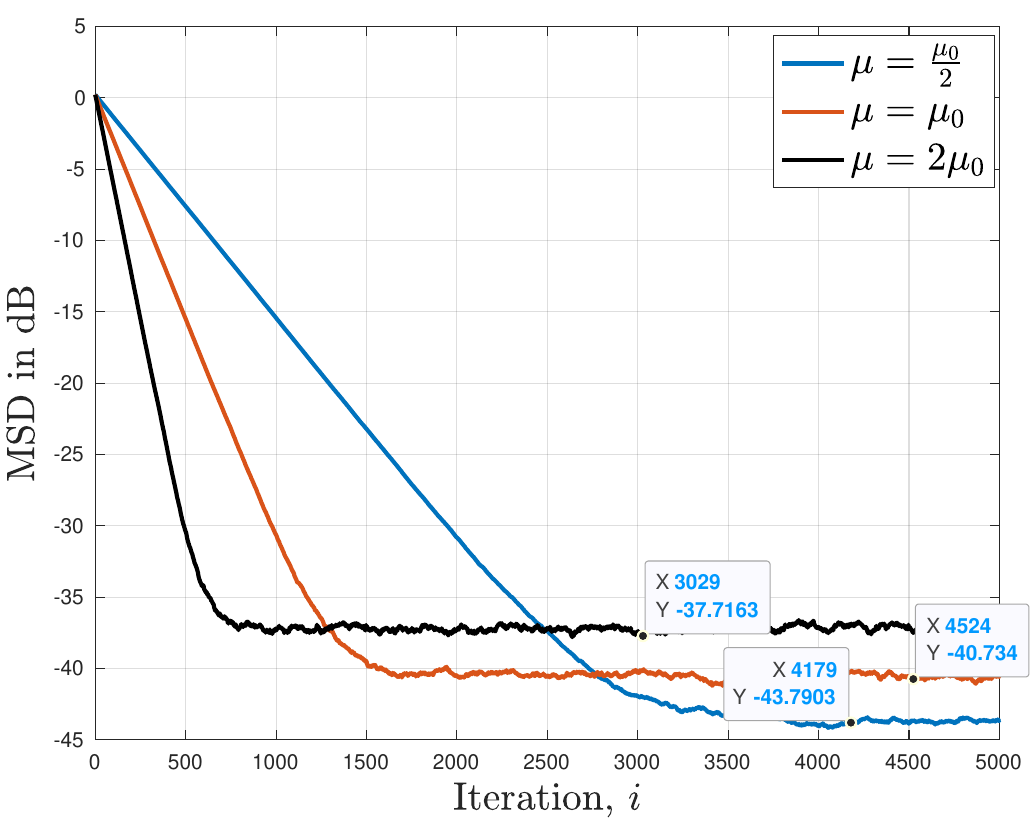}
\caption{Network performance w.r.t. $\cw^o$ in~\eqref{eq: network constrained problem} for three different values of the step-size ($\mu_0=0.003$). {In the {left} and {middle} plots, the probabilistic ANQ quantizer of Example 2 is employed.} \textit{(Left)}  {Evolution of the MSD learning curves. \textit{(Middle)}} Evolution of the average number of bits per node, per component, when  the variable-rate scheme described in {Sec.~\ref{subsec: Uniform and non-uniform randomized quantizers}} is used {to encode the difference $\bchi_{k,i}=\bpsi_{k,i}-\bphi_{k,i-1}$ in~\eqref{eq: stepb}}. {\textit{(Right)}  Evolution of the MSD learning curves when the QSGD quantizer (with $s=2$) of Table~\ref{table: examples of quantizers} is used.}
}
\label{fig: variable step-size supplementary}
\end{center}
\end{figure*}

 \thispagestyle{empty}
 \begin{figure*}
\begin{center}
\includegraphics[scale=0.4]{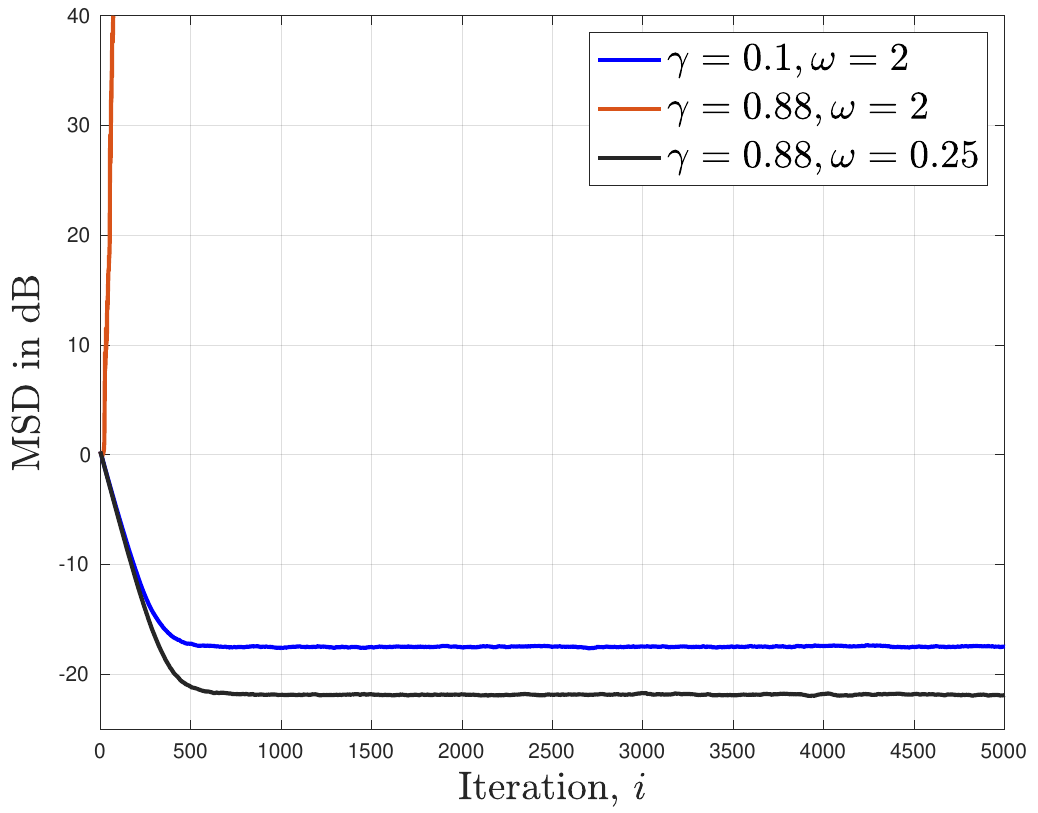}
\caption{Network performance w.r.t. $\cw^o$ in~(3) for different values of the ANQ quantizer's parameter $\omega$ and mixing parameter $\gamma$.}
\label{fig: stability conditions reviewer 4}
\end{center}
\end{figure*}
\end{document}